\documentclass[aop,MSNbibl,dvips]{arximspdf}
\usepackage{mathbh,multirow}
\usepackage{graphicx}

\doi{10.1214/14-AOP924} %kopijuoti is PTS
\volume{42}
\issue{4}
\pubyear{2014}
\firstpage{1590}
\lastpage{1634}

\makeatletter
\def\implies{\Longrightarrow}
\newcommand{\eqref}[1]{(\ref{#1})}
\newtheorem{theorem}{Theorem}[section]
\newtheorem{claim}[theorem]{Claim}
\newtheorem{prop}{Proposition}[section]
\newtheorem{thmm}[prop]{Theorem}
\newtheorem{lem}[prop]{Lemma}
\newtheorem{cor}[prop]{Corollary}
\newproclaim{defn}[prop]{Definition}
\newproclaim{rem}[prop]{Remark}
\newtheorem{clm}[prop]{Claim}

\newcommand{\BN}{{\mathbb{N}}}
\newcommand{\BZ}{{\mathbb{Z}}}

\newcommand{\CA}{{\mathcal{A}}}
\newcommand{\CD}{{\mathcal{D}}}
\newcommand{\CF}{{\mathcal{F}}}
\newcommand{\CG}{{\mathcal{G}}}
\newcommand{\CI}{{\mathcal{I}}}
\newcommand{\CJ}{{\mathcal{J}}}
\newcommand{\CO}{{\mathcal{O}}}
\newcommand{\CP}{{\mathcal{P}}}
\newcommand{\CR}{{\mathcal{R}}}
\newcommand{\CT}{{\mathcal{T}}}

\newcommand{\ind}{{\mathbh{1}}}
\newcommand{\prob}{\mathbf{P}}
\newcommand{\pr}{\mathbb{P}}
\newcommand{\Z}{\mathbb{Z}}
\newcommand{\proj}{\Pi_N}
\newcommand{\om}{{\omega}}
\newcommand{\ep}{{\varepsilon}}
\newcommand{\lam}{{\lambda}}
\makeatother

\begin{document}
\begin{frontmatter}

\title{On the range of a random walk in a torus and random interlacements}
\runtitle{Range of a random walk in a torus}

\begin{aug}
\author[a]{\fnms{Eviatar~B.}~\snm{Procaccia}\corref{}\thanksref{t1}\ead[label=e1]{procaccia@math.ucla.edu}\ead[label=u1,url]{http://www.math.ucla.edu/\textasciitilde procaccia}}
%x\author{\fnms{Eric} \snm{Shellef}\thanksref{t3,m1,m2}
\and
\author[b]{\fnms{Eric}~\snm{Shellef}\thanksref{T2}\ead[label=e3]{shellef@gmail.com}}

\thankstext{t1}{Work on this project was done while the author
was in the Weizmann Institute of Science.}
\thankstext{T2}{Supported by ISF Grant 1300/08
and EU Grant PIRG04-GA-2008-239317.}
\runauthor{E. B Procaccia and E. Shellef}

\affiliation{UCLA and Weizmann Institute of Science,\break
and Weizmann Institute of Science}

\address[a]{Department of Mathematics\\
UCLA\\
520 Portola Plaza\\
Los Angeles, California 90095\\
USA\\
and\\
Faculty of Mathematics \\
\quad and Computer Science\\
Weizmann Institute of Science\\
POB 26\\
Rehovot 76100\\
Israel\\
\printead{e1}\\
\printead{u1}}

\address[b]{Faculty of Mathematics\\
\quad and Computer Science\\
Weizmann Institute of Science\\
POB 26\\
Rehovot 76100\\
Israel\\
\printead{e3}}
\end{aug}

% HISTORY:
\received{\smonth{6} \syear{2012}}
\revised{\smonth{11} \syear{2013}}

% ABSTRACT
%
\begin{abstract}
Let a simple random walk run inside a torus of dimension three or
higher for a number of steps which is a constant proportion of the
volume. We examine geometric properties of the range, the random subgraph
induced by the set of vertices visited by the walk. Distance and mixing
bounds for the typical range are proven that are a $k$-iterated log
factor from those on the full torus for arbitrary $k$. The proof
uses hierarchical renormalization and techniques that can possibly
be applied to other random processes in the Euclidean lattice. We use
the same technique to bound the heat kernel of a random walk on random
interlacements.
\end{abstract}

% KEYWORDS
%

\end{frontmatter}

%s1 #&#
\section{Introduction}\label{sec1}
Consider a discrete torus of side length $N$ in dimension \mbox{$d\ge3$}.
Let a simple random walk run in the torus until it fills a constant
proportion of the torus and examine the \emph{range}, the random subgraph
induced by the set of vertices visited by the walk. How well does
this range capture the geometry of the torus? Viewing the range as
a random perturbation of the torus, we can draw hope that at least
some geometric properties of the torus are retained, by considering
results on a more elementary random perturbation, Bernoulli percolation.

It is now known that various properties of the Euclidean lattice ``survive''
Bernoulli percolation with density $p>p_{c}(\Z^d)$. In \cite
{antal1996chemical},
Antal and Pisztora proved that there is a finite $C(p,d)$ such that
the graph distance between any two vertices in the infinite cluster
is more than $C$ times their $l_{2}$ distance, with probability
exponentially low in this distance. Isoperimetric bounds for the
largest connected cluster in a fixed box of side $n$ were given by
Benjamini and Mossel for $p$ sufficiently close to $1$ in \cite
{benjamini2003mixing},
and by Mathieu and Remy for $p>p_{c}$ in \cite{mathieu2004isoperimetry}.
A consequence is that the mixing time for a random walk on this cluster
has the same order bound, $\theta(n^{2})$, as on the full box. In
\cite{pete2007note}, Pete extends this result to more general graphs.

%
%f1 #&#
\begin{figure}

\includegraphics{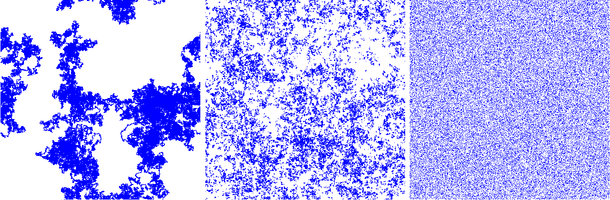}

\caption{From left to right, the range in
2 dimensions, a slice in 3 dimensions and Bernoulli percolation, all
of density 0.3.}\label{figrangevsperc}
\end{figure}

Returning to our process, in Figure~\ref{figrangevsperc} simulation
pictures are shown that give heuristical support to the view that
although the range for $d\ge3$ has long range dependence, it bears
some similarities to i.i.d. site percolation. Indeed, one can see that
the middle picture, a 2d slice of the range of a walk that filled
30\% of a 3d torus, is ``in between,'' dependence-wise, the i.i.d.
picture on the right and the highly dependent picture on the left
where the effect of two-dimensional recurrence is evident. Thus, one
might expect analogous geometric behavior of the range for $d\ge3$
and i.i.d. percolation. This partially turns out to be the case.

In \cite{benjamini2008giant}, the complement of the range, called
the \emph{vacant set}, is investigated by Benjamini and Sznitman.
For positive $u$, it is shown $uN^{d}$ is indeed the proper timescale
to generate percolative behavior of the vacant set. Starting at the
uniform distribution, it is easily shown that for some $c(u,d)>0$,
the probability a given vertex in the torus is visited by the walk
is between $c$ and $1-c$, independently of $N$. A more difficult
result is that for small $u$, the vacant set typically contains a
connected component that is larger than some constant proportion of
the torus. Indeed, simulations support the existence of a phase transition
in $u$ of the vacant set geometry, where below some critical $u_{c}>0$,
a unique giant component appears, and above it all clusters are microscopic.

The range, unlike the vacant set, does not display an obvious phase
transition in $u$. It is connected for all positive $u$, and fills
a $c'(u,d)>0$ proportion of the torus with high probability. Despite
the analogy to percolation being flawed in this respect, the range does
display some
percolative behavior due to the Markov property and uniform transience
of a random walk in $d>2$. Roughly, conditioning on the vertices
by which the walk enters and exits a small box makes the path in between
them independent from the walk outside this box. Using this idea and
facts from percolation theory gathered in Section~\ref
{secSupercritical-percolation-percolates},
we prove the range does capture the distance and isoperimetric bounds
of the torus, though our methods require an iterated logarithmic correction
to the bounds of the full torus. In Section~\ref{secRenormalization},
it is shown that for arbitrarily small $u>0$, the range asymptotically
dominates a recursive structure, defined in Section~\ref
{secNotation-and-Definitions},
which can roughly be described as a finite-level supercritical fractal
percolation. From this structure, we extract distance bounds
(Appendix~\ref{secDistance-bound}) and mixing bounds (Section~\ref
{secMixing-bound})
that are a $\log^{(k)}(N)=\log(\log(\cdots(\log(N)\cdots k\cdots)))$
factor from those on the torus.

Let us expand a bit on the heuristics presented in the previous
paragraph. Since
the holes in the range are larger than those in i.i.d. percolation (see
the last comment in \cite{benjamini2008giant}), one can never hope
to dominate it. Instead, we formulate a notion of density of a box
of side $n$, which essentially means that it is crossed top to bottom
(\emph{traversed}) by the random walk an order of $n^{d-2}$ times.
A union bound then gives that w.h.p. all $\log^{4}N$-sided ``first-level''
boxes in the torus possess this property. Next, given this condition,
for each fixed first-level box, all internal ``second-level''
boxes of side $c\log^{4}(\log N)$ are dense w.h.p., and independently
from other disjoint first-level boxes. The probability for the denseness
of the second-level boxes is not high enough for a union bound on
all of them, however, it is enough such that first-level boxes whose
second-level boxes are all dense dominate $p$-percolation for arbitrarily
high $p<1$. This is the basis of the hierarchical renormalization
used below to prove the same fact for ``$k$-level'' boxes with
arbitrary $k$. A drawback of this method is that the density of boxes
becomes diluted by a constant factor from level to level, preventing
us from continuing this rescaling to reach boxes of a bounded size.
This dilution is the main source of the $\log^{(k)}(N)$ correction.
We believe this correction is an artifact of the method and that the
true bounds should be the same as those on the torus.

A central technical concept introduced in the paper is the recursively
defined $k$-goodness of a box, which is roughly that the $(k-1)$-good
smaller scale boxes inside satisfy some typical supercritical Percolation
properties. The main demand from $0$-good boxes is that the range
is connected in their interior. This provides a useful way to analyze
the range but perhaps a better formulated notion will get sharper
bounds. A second technique worth mentioning is the propagation of
isoperimetric bounds through multiple scales in Lemma~\ref{lemphilowbnd}.
This has been done for one level in \cite{mathieu2004isoperimetry},
but it is not clear how to extend the method there to more than one
level. Last, getting rid of dependence on time in the random walk
when moving to smaller scale boxes is not trivial. To do this, we
prove the domination of the $k$-good recursive structure mentioned
above simultaneously for all $\{ \mathcal{R}_{N}(t)\} _{t\ge
uN^{d}}$,
where $\mathcal{R}_{N}(t)$ is the range of the walk up to time~$t$.
This is facilitated by results on conditioned random walks from
Section~\ref{secDense-Random-Walk}, in particular by Lemma~\ref
{lemallconnected}.
The lemma shows that given any fixed ``boundary-connected-path''
$f(t)$ in a dense box (see definition above Lemma~\ref
{lemPNgePrhoandFt}),
the random walk traversals will merge it w.h.p. into a single connected
component, for all $t\ge0$.

Using the results proved for the random walk on the torus, we prove a
bound on the Heat kernal of random walk on Random Interlacements. In
Appendix \ref{appC}, we write a short introduction on Random Interlacements
where one can find the notation used in Section~\ref{secinter}.

It should be mentioned that while all sections ahead require the terminology
introduced in Section~\ref{secNotation-and-Definitions}, all remaining
sections apart from Section~\ref{secRenormalization} may be read
quite independently from one another. Section~\ref{secRenormalization}
also relies on random walk definitions from Section~\ref
{secDense-Random-Walk}. For reading convenience, one can find an index
of symbols in Appendix~\ref{secinde}.

%s2 #&#
\section{Result and notation}\label{secNotation-and-Definitions}

\def\Pw{\mathtt{P}}
\def\Dns{\Lambda}
\newcommand{\Bx}[1]{\mathbf{#1}}
\def\Good{\mathcal{G}}
\newcommand{\SubBoxSet}[1]{\mathbb{#1}}

Let $\mathcal{T}(N,d)$\label{pg1} be the discrete $d$-dimensional torus with
side length $N$, for $d\ge3$. Fixing $d$, $\mathcal{T}(V,E)$ is
a graph with
\[
V(N)=\bigl\{ \mathbf{x}\in\mathbb{Z}^{d}\dvtx0\le x_{i}<N,
1\le i\le d\bigr\}
\]
and
\[
E(N)=\bigl\{ \{ \mathbf{x},\mathbf{y}\} \subset V\bigl(\mathcal {T}(N)\bigr)
\dvtx\proj(\mathbf{x}-\mathbf{y})\in\{ \pm\mathbf {e}_{1},\ldots,\pm
\mathbf{e}_{d}\} \bigr\},
\]
where\label{pg2} $\proj\dvtx\mathbb{Z}^{d}\to V(N)$ for $\mathbf{x}\in\mathbb{Z}^{d}$
is ${\proj}(\mathbf{x})=(x_{1}\operatorname{ mod }N,\ldots, x_{d}\operatorname{
mod }N)$
and $\{ \mathbf{e}_{i}\} _{i=1}^{d}$ is the standard basis
of $\mathbb{Z}^{d}$.

Note that if $S(\cdot)$ is a simple random walk (SRW) in $\mathbb{Z}^{d}$,
$S_{N}(\cdot)=\proj\circ S(\cdot)$ is a SRW in $\mathcal{T}(N)$.
Let\label{pg3} $\mathcal{R}(t_{1},t_{2})=\{ S(s)\dvtx t_{1}\le s<t_{2}\} $
and call $\mathcal{R}(t)=\mathcal{R}(0,t)$ the \emph{range} (until
time $t$) of the walk. We consider $\mathcal{R}_{N}(t)$, the random
connected subgraph of $\mathcal{T}$ induced by $\proj\circ\mathcal
{R}(t) $, where we include only edges traversed by the random walk.
Throughout the paper, when no ambiguity is present, we identify a
graph with its vertices.

Let $\mathbf{P}_{\mathbf{x}}[\cdot]$ be the law that makes
$S(\cdot)$ an independent SRW starting at $\mathbf{x}\in\mathbb{Z}^{d}$.
Below are the main three results of the paper.
%
%th2.1 #&#
\begin{thmm}\label{thmmdistthm}
Set $u>0$ and for a graph $G$, let $d_{G}(\cdot,\cdot)$ denote
graph distance. Then for any $k$,
\begin{eqnarray*}
&&\lim_{N\to\infty}\mathbf{P}_{\mathbf{0}}\biggl[\max
_{t\ge uN^{d}}\biggl\{ \frac{d_{\mathcal{R}_{N}(t)}(\mathbf{x},\mathbf{y})}{d_{\mathcal
{T}(N)}(\mathbf{x},\mathbf{y})} \dvtx\mathbf{x},\mathbf{y}\in
\mathcal {R}_{N}(t),d_{\mathcal{T}(N)}(\mathbf{x},\mathbf{y})>(\log N
)^{5d}\biggr\} \\
&&\hspace*{262pt}{}>\log^{(k)}N\biggr]\\
&&\qquad=0,
\end{eqnarray*}
where $\log^{(k)}N$ is $\log(\cdot)$ iterated $k$-times of $N$.
\end{thmm}
Since this paper was uploaded to the arXiv on 2010, the distance bounds
where improved in \cite{cerny2011internal} by Cern{\`y} and Popov.
They managed to get a tight result without the $\log$ correction. Due
to the improvement, the proof of Theorem~\ref{thmmdistthm} is
postponed to Appendix \ref{appB}. Note that since distance bounds require
finding one good path and isoperimetric bounds require a uniform bound
on all subsets, the rest of the results in this paper do not follow the
techniques of \cite{cerny2011internal}.

%th2.2 #&#
\begin{thmm}\label{thmmmixing}
Set $u>0$ and let $\tau(G)$ be the (e.g., uniform) mixing time
of a simple random walk on a graph $G$. Then for any $k$,
\[
\lim_{N\to\infty}\mathbf{P}_{\mathbf{0}}\biggl[\max
_{t\ge uN^{d}}\frac
{\tau(\mathcal{R}_{N}(t))}{N^{2}}>\log^{(k)}N\biggr]=0.
\]
\end{thmm}
The two theorems are a direct consequence of Theorem~\ref{thmmRenormalization}
and Theorems \ref{thmmDistancebound}, \ref{thmmmixingbound}, respectively.

Using the same techniques for proving Theorem~\ref{thmmdistthm} and
Theorem~\ref{thmmmixing}, we can show the next result for a random
walk on the range of random interlacements (see Appendix \ref{appC}
for notation).
%
%th2.3 #&#
\begin{thmm}\label{thmmInterlacementmain}
Let $u>0$ and $k\in\BN$. Then there exists a constant $C(u,k)$ such
that for $\pr^u_0$ almost every $\CI^u$, and for all $n$ large enough
\[
\mathbf{P}_0^{u}[0,n]\le\frac{C\cdot\log^{(k)}(n)}{n^{d/2}}.
\]
\end{thmm}

This theorem quantifies the result of R\'ath and Sapozhnikov in \cite
{rath2011transience}. R\'ath and Sapozhnikov proved the graph of random
interlacements is transient a.s.

The main purpose of the remainder of the section is to define a $k$-good
configuration, and to establish notation used throughout the paper.

%s2.1 #&#
\subsection{Graph notation}\label{subGraph-notation}

Given a graph $G$, we identify a subset of vertices $V$ with its
induced subgraph in $G$. We denote $G\setminus V$, the complement
of $V$ relative to $G$, by $V_{G}^{c}$. Writing $d_{G}(\cdot,\cdot)$
for the graph distance in $G$, we let $d_{G}(\mathbf{v},V)=\inf\{
d_{G}(\mathbf{v},\mathbf{x})\dvtx\mathbf{x}\in V\} $.
For the outer and inner boundary, we respectively write
\begin{eqnarray*}
\partial_{G}(V)\label{pg4}&=&\bigl\{ \mathbf{v}\in G\dvtx d_{G}(
\mathbf{v},V)=1\bigr\},
\\
\partial_{G}^{\mathrm{in}}(V)\label{pg5}&=&\partial_{G}
\bigl(V_{G}^{c}\bigr)=\bigl\{ \mathbf{v}\in G\dvtx d\bigl(
\mathbf{v},V_{G}^{c}\bigr)=1\bigr\}.
\end{eqnarray*}
We often omit $G$ from the notation when the ambient graph is clear.
We say $V$ \emph{is} \emph{connected in} $G$ if any two vertices
in $V$ have a path in $G$ connecting them. $V_{1},V_{2}\subset G$
are connected in $G$ if $V_{1}\cup V_{2}$ is connected in $G$.
Given $V\subset G$, we call a set that is connected in $V$ and is
maximal to inclusion a \emph{component of }$V$.

As noted above, we identify graphs and their vertices. Thus, $\mathbb{Z}^{d}$
denotes the $d$-dimensional integers as well as the graph on these
vertices in which two vertices are connected if they differ by a unit
vector.

Last, if $V\subset\mathbb{Z}^{d},\mathbf{z}\in\mathbb{Z}^{d}$ then
$V\pm\mathbf{z}=\{ \mathbf{x}\pm\mathbf{z}\dvtx\mathbf{x}\in V\} $.

%s2.2 #&#
\subsection{Box notation}\label{subBox-notation}

For $\mathbf{x}\in\mathbb{Z}^{d},n>0$, let
\[
\label{pg6}B(\mathbf{x},n)=\bigl\{ \mathbf{y}\in\mathbb{Z}^{d}\dvtx\forall i,1
\le i\le d, -n/2\le\mathbf{x}(i)-\mathbf{y}(i)<n/2\bigr\}.
\]
We write $B(n)$ if $\mathbf{x}$ is the origin, and when length and
center are unambiguous we often just write $B$. Occasionally, we use
lowercase $b$ for a smaller instance of a box. We denote the side
length of a box by $\|B\|$, that is,
\[
\|B\|=|B|^{{1}/{d}}.
\]
Let $\operatorname{sp}\{ B(\mathbf{x},n)\} =\{ B(\mathbf{x}+\sum_{i}\mathbf{e}_{i}k_{i}n,n)\dvtx(k_{1},\ldots,k_{d})\in\mathbb{Z}^{d}
\} $
where $\mathbf{e}_{1},\ldots,\mathbf{e}_{d}$ are the unit vectors
in $\mathbb{Z}^{d}$, that is, all the nonintersecting translations of
$B$ in $\Z^d$. We attach a graph structure to $\operatorname{sp}\{
B(\mathbf{x},n)\} $\label{pg7}
by defining the neighbors of a box $B(\mathbf{x},n)$ as $B(\mathbf
{x}\pm
\mathbf{e}_{i}n,n)$,
$1\le i\le d$. Henceforth, any graph operators on a subset of some
$\operatorname{sp}\{B\}$ refer to this graph structure.
Observe that $\operatorname
{sp}\{ B(\mathbf{x},n)\} $
is isomorphic as a graph to $\mathbb{Z}^{d}$. We fix an isomorphism\label{pg8}
$\Delta\dvtx\operatorname{sp}\{B\}\to\mathbb{Z}^{d}$, $\Delta(B(\mathbf
{x}+\sum_{i}\mathbf{e}_{i}k_{i}n,n))=\mathbf{x}+\sum_{i}\mathbf{e}_{i}k_{i}$.
Using $\Delta$, we
extend the definitions of a box to boxes as well. Thus, for a box $b=b(n)$
and an integer $m>0$, $B_{\Delta}(b,m)$ is a set of $m^{d}$ boxes.
We use a big union symbol to denote internal union, that is, $\bigcup
\mathbf{A}=\{ \mathbf{x}\in A\dvtx A\in\mathbf{A}\} $.
So in the preceding example, we have $\bigcup B_{\Delta}(b,m)=B(mn)$.

To ease the reading, we often refer to boxes that are neighbors under
the above relationship as $\Delta$-\emph{neighbors}, a connected
set of boxes as $\Delta$-\emph{connected}, and a component under
$\Delta$-neighbor relationship a $\Delta$\emph{-component.}

%de2.4 #&#
\begin{defn}
Given a box $B(\mathbf{x},n)$, and $\alpha>0$, we write $B^{\alpha}$\label{pg9}
for $B(\mathbf{x},\alpha n)$.
Let
\[
\label{pg10}s(n)=\lceil\log n\rceil^{4}.
\]
We write \label{pg11}$s^{(i)}(n)$ to denote $s(\cdot)$ iterated $i$ times.
\end{defn}
%
%de2.5 #&#
\begin{defn}
Let
\[
\label{pg12}\sigma\bigl(B(\mathbf{x},n)\bigr)=\operatorname{sp}\bigl\{ b\bigl(\mathbf {x},s(n)\bigr)
\bigr\} \cap\bigl\{ b\bigl(\mathbf{y},s(n) \bigr)\dvtx\mathbf{y}\in B\bigl(
\mathbf{x},5n+3\lceil\log n\rceil^{6}\bigr)\bigr\}
\]
be the \emph{subboxes} of $B(\mathbf{x},n)$. Note that $B^5\subset
\bigcup\sigma(B)$. $\sigma(B)$ is a collection of sub-boxes of side
length $s(n)$ covering $B^5$; see Figure~\ref{figzerogood} for visualization.
\end{defn}
We write $2^{A}$ for the power set of a set $A$, that is, the
collection of
subsets of $A$. We refer to finite
subsets of $\mathbb{Z}^{d}$ as configurations.

%
%f2 #&#
\begin{figure}[b]

\includegraphics{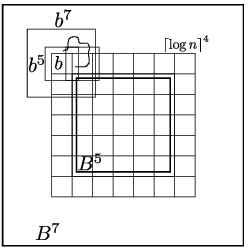}

\caption{$0$-good configuration.}\label{figzerogood}
\end{figure}

%s2.3 #&#
\subsection{Percolating configurations}\label{subperc}

Let $c_{a},c_{b}$ be fixed positive constants dependent only on dimension
($c_{a},c_{b}$ are determined in Lemma~\ref{lemAntalPiszforpercinbox}
and Corollary~\ref{corstarboundariestypical}, resp.). $\omega
\in2^{B(n)}$
is a \emph{percolating} configuration, denoted by $\om\in\CP(n)$,\label{pg13} if
there exists a subset which we call a \emph{good
cluster} $\mathcal{C}=\mathcal{C}(\omega)\subset\omega$, connected
in $\omega$ (not necessarily maximal) for which the following properties
hold:
\begin{longlist}[1.]
\item[1.]%\label{enupercgiantcluster}
$|\mathcal{C}
|>(1-10^{-d})|B(n)|$.
\item[2.]%\label{enupercbndonholesize}
The largest component in
$B(n)\setminus\mathcal{C}$
is of size less than $(\log n)^{2}$.
\item[3.]%\label{enuAntalPisztora}
For any $\mathbf{v},\mathbf{w}\in
\mathcal{C}\cap B(n-c_{a}\log n)$
we have $d_{\mathcal{C}}(\mathbf{v},\mathbf{w})<c_{a}(d_{B}(\mathbf
{v},\mathbf{w})\vee\log n)$.
Moreover, a configuration $\omega\in\CP(n)$ admits an isoperimetry property:
\item[4.]%\label{enuconnectedboundarylarge}
Let $T\subset B(n)$ satisfy
$n^{1/5d}<|T|\le n^{d}/2$, and assume both $T$ and
$B(n)\setminus T$
are connected in $B(n)$. Then $|\partial_{B}T\cap\omega
|,|\partial_{B}^{c}T\cap\omega|>c_{b}|T|^{(d-1)/d}$.
\end{longlist}
The following claim is easy to check.
%
%cl2.1 #&#
\begin{claim}
\label{clapercolationmonotone}$\mathcal{P}(n)$ is a \emph{monotone}
set, that is, if $\omega\in\mathcal{P}(n)$ and $\omega\subset
\omega
^{+}\subset B(n)$
then $\omega^{+}\in\mathcal{P}(n)$.
\end{claim}

%s2.4 #&#
\subsection{$k$-good configurations}\label{subk-good}

Let $c_{h}$ be a fixed positive constant dependent only on dimension
($c_{h}$ is determined in Theorem~\ref{thmmdenseboxesare0good} below).
For $n\in\mathbb{N},\rho>0$, and setting $B=B(n)$, a configuration
$\omega\subset B^7$ belongs to $\Good_{0}^{\rho}(n)$\label{pg14} if and only
if the following properties hold:
\begin{longlist}[1.]
\item[1.]%\label{enurhodense}
For each $b\in\sigma(B)$, $|\omega\cap
b|>(\rho c_{h}\wedge\frac{1}{2})|b|$.
\item[2.]%\label{enuAnyxyconnected}
For each $b\in\sigma(B)$, $\omega
\cap b^5$
is connected in $\omega\cap b^7$.
\end{longlist}
%
%re2.6 #&#
\begin{rem}
\label{rem0goodconnected}If $\omega\in\Good_{0}^{\rho}(n)$, then
for all $n>(\rho c_{h})^{-{1}/{d}}$: (i) $\omega$ intersects
all $b\in\sigma(B)$ (property~1), and (ii) for
any two $\Delta$-neighbors $b_{1},b_{2}\in\sigma(B)$, since
$b_{2}\subset b_{1}^5$,
$\omega\cap b_{1}$ and $\omega\cap b_{2}$ are connected in $\omega
\cap
b_{1}^{7}$
(property~2). In particular, $\omega\cap B^5$
is connected in $\omega\cap B^7$. See Figure~\ref{figzerogood} for a
graphical explanation.
\end{rem}

Let $\Dns$ be a fixed positive constant dependent only on dimension
($\Dns$ is determined in Theorem~\ref{thmmdenseimpliessubboxdensewhp}).
For $k>0$, $\Good_{k}^{\rho}(n)$ is defined recursively. Given
$\omega
\subset\mathbb{Z}^{d}$, $i\in\BN$
and a box $b(\mathbf{x},m)$, we say $b$ is $(\omega,i,\rho)$\emph{-good}
if $(\omega\cap b^7)-\mathbf{x}\in\Good_{i}^{\rho}(m)$.
Let
\[
\mathcal{S}=\bigl\{ b\in\sigma(B)\dvtx b\mbox{ is }(\omega,k-1,\rho \Dns)\mbox
{-good}\bigr\},
\]
and let $\sigma_{B}=\|\Delta(\sigma(B))\|=|\sigma(B)|^{{1}/{d}}$.
Then $\omega\in\Good_{k}^{\rho}(n)$ if $\omega\in\Good_{0}^{\rho}(n)$
and $\Delta(\mathcal{S})\in\mathcal{P}(\sigma_{B})$. See
Figure~\ref{figkgood} for a graphical explanation.
%
%f3 #&#
\begin{figure}

\includegraphics{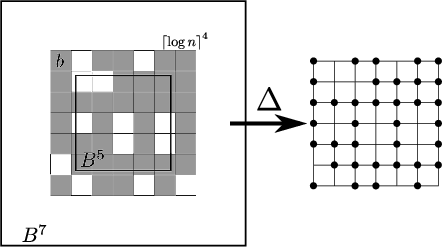}

\caption{$k$-good configuration. All the grey subboxes
are $k-1$-good, that is, $\om\cap b\in \CG_{k-1}^{\rho\Lambda
}(\lceil
\log n\rceil^4)$. The configuration on the right is in $\CP(\sigma
_B)$.}\label{figkgood}
\end{figure}

%s2.5 #&#
\subsection{$k$-good torus}\label{subk-goodtorus}

Let $\mathcal{T}=\mathcal{T}(N)$ and fix $\omega\subset\mathcal{T}$.
Let $k\ge0,\rho>0$. We define $(\omega,k,\rho)$\emph{-goodness of
a} \emph{torus}. Let $n=\lceil N/10\rceil$. We call
\[
\Bx T=\operatorname{sp}\bigl\{B(n)\bigr\}\cap\bigl\{B(\mathbf{y},n)\dvtx\mathbf{y}\in
B(N)\bigr\}
\]
the \emph{top-level} boxes for $\mathcal{T}$. Then $\mathcal{T}$
is a $(\omega,k,\rho)$-good torus if all boxes in $\Bx T$ are $(\proj
^{-1}\omega,k,\rho)$-good.

Remark~\ref{rem0goodconnected} therefore implies the following.
%
%re2.7 #&#
\begin{rem}
\label{remgoodtorusconnected}If $\mathcal{T}(N)$ is a $(\omega,k,\rho)$-good
torus, then $\omega$ is connected for all $N>C(\rho)$.
\end{rem}

%s2.6 #&#
\subsection{Constants}

All constants are dependent on dimension by default and independent
of any other parameter not appearing in their definition. Constants
like~$c,C$ may change their value from use to use. Numbered constants
(e.g., $c_{1},C_{2}$) retain their value in a proof but no more than
that, and constants tagged by a letter $(c_{a},c_{\Dns})$ represent
the same value throughout the paper.

%%%%%%%%%%%%%%%%%%%%%%%%%%%%%%%%%%%%%%%%%%%%%%%%%%%%%%%%%%%%%%%%%%%%%%%%%%%%%%%%%%%%%%%%%%%
%%%%%%%%%%%%%%%%%%%%%%%%%%%%%%%% Section~5 Mixing bounds
%%%%%%%%%%%%%%%%%%%%%%%%%%%%%%%%%%
%%%%%%%%%%%%%%%%%%%%%%%%%%%%%%%%%%%%%%%%%%%%%%%%%%%%%%%%%%%%%%%%%%%%%%%%%%%%%%%%%%%%%%%%%%%%%

%s3 #&#
\section{Mixing bound}\label{secMixing-bound}

Given a finite connected graph $G$, let $X(t)$ be a lazy random
walk on $G$. That is, denoting the walk's transition matrix by $p(\cdot,\cdot)$,
for any $\mathbf{v}\in G$ of degree $m$, $p(\mathbf{v},\mathbf{v})=1/2$
and $p(\mathbf{v},\mathbf{w})=1/2m$ for\vadjust{\goodbreak} any neighbor $\mathbf{w}\in
\partial\{\mathbf{v}\}$.
We write $\tau(G)$ for the mixing time of $X(t)$ on $G$, that is,
\[
\tau(G)=\min\biggl\{ n\dvtx\biggl|\frac{p^n(x,y)-\pi(y)}{\pi(y)}\biggr|\le \frac{1}{4},~\forall x,y
\in V(G)\biggr\},
\]
where $\pi$ is the stationary measure of the random walk on $G$. See
\cite{morris2005evolving}
a thorough introduction on mixing times.
%
%th3.1 #&#
\begin{thmm}
\label{thmmmixingbound}Let $\omega_{0}\subset\mathcal{T}(N)$, $\rho>0$,
$k\ge1$. There is a $C(k,\rho)$ such that if $\mathcal{T}(N)$ is
a $(\omega_{0},k,\rho)$-good torus then
\[
\tau(\omega_{0})<CN^{2}\log^{(k-1)}N,
\]
where $\log^{(m)}N$ is $\log(\cdot)$ iterated $m$ times of $N$.
\end{thmm}

We begin by stating and proving propositions required for
Corollary~\ref
{corkgoodtorusisopbnd}, then using the corollary we prove Theorem~\ref
{thmmmixingbound}.

Recall the definition of $\Good_{l}^{\rho}(n)$ from Section~\ref{subk-good}.
Let $c_{\rho}=(\rho c_{h}\wedge\frac{1}{2})/3$.
We assume $n$ is large enough such that $\Good_{l}^{\rho}(n)$ is
nonempty, and that for any $\omega\in\Good_{l}^{\rho}(n)$, $\omega
\cap B^5(n)$
is connected in $\omega$ and satisfies $|\omega\cap B^5(n)
|>3c_{\rho}n^{d}$
(see property~1 of $\Good_{0}^{\rho}$ in
Section~\ref{subk-good} and Remark~\ref{rem0goodconnected}). In particular,
there exists a set $S\subset\omega,|S\cap B^5(n)|\wedge
|(\omega\setminus S)\cap B^5(n)|\ge c_{\rho}n^{d}$.

Since $\omega\cap B^5(n)$ is connected in $\omega$, we have the following.
%
%pr3.2 #&#
\begin{prop}
\label{proinductionbase}For any $l\ge0$ and all large $n$, and
$S\subset\om\in\Good_{l}^{\rho}(n)$
\[
|\partial_\om S|\ge1.
\]
\end{prop}
Next, we bound $|\partial S|$ more accurately. The next theorem is one
of the main results and techniques introduced in this paper.
The theorem proves an almost tight isoperimetric inequality (up to an
iterated $\log$). The main idea of the proof is induction on the number
of iterations (which provide the iterated $\log$) and analyzing the
geometry of renormalized subsets, that is, use the geometrical
properties of the percolation configuration of good subboxes.
%
%th3.3 #&#
\begin{thmm}
\label{lemphilowbnd}Let $l\ge0$, $\rho>0$, $\om\in\CG_l^\rho$ and
$S\subset\om$ such that $|S\cap B^5(n)|\wedge|(\omega\setminus
S)\cap
B^5(n)|=r\ge n^{{1}/{3}}$. There exists a constant $c_1(l,\rho)>0$,
such that
%
%e1 #&#
\begin{equation}
\label{eqphiineq} |\partial_\om S|>c_1(l,
\rho)r^{{(d-1)}/{d}}\bigl(s^{(l)}(n)\bigr)^{1-d}.
\end{equation}
\end{thmm}
\begin{pf}
The proof is by induction on $l$. For $l=0$, since $s^{(0)}(n)=n$,
$|B^5(n)|^{{(d-1)}/{d}}s^{(0)}(n)^{1-d}$ is less
than some $C_{1}$ for any $r\le|B^5(n)|$. Thus, the
base case of $l=0$ is given in Proposition \ref{proinductionbase}
and the connectedness of $\om$ with $c_{1}(0)=C_{1}^{-1}$. Now fix
$l>0,\rho>0$ and assume \eqref{eqphiineq}
is true for $l-1$ with constant $c_{1}(l-1,\rho\Dns)>0$, for all
large $n$ and $n^{1/3}\le r\le|B^5(n)|$.

Our default ambient
graph for $S$ is $\omega$. Thus, for $S\subset\omega$, $S^{c}=\omega
\setminus S$
and $\partial S=\partial_{\omega}S$. Note that as $|S|\ge r$,
if $|\partial S|>|S|^{(d-1)/d}$
we are done. W.l.o.g. assume $|S^c\cap B^5|\ge|S\cap B^5|$ since
$|\partial
_\om S^c|\sim|\partial_\om S|$.

Let $B=B(n)$ and let $m=s(n)$. For $0<\alpha<1$, let
\[
\Bx F=\Bx F(\omega,S,\alpha)=\bigl\{ b\in\sigma(B)\dvtx|b\cap S |\ge\alpha|b\cap
\omega|\bigr\},
\]
be the $\alpha$-filled subboxes. By the pigeon hole principle, there
are ${\alpha}(\rho)<1$, $c_{2}(\rho)>0$, such that
%
%e2 #&#
\begin{equation}
|\Bx F|<(1-c_{2})\bigl|\sigma(B)\bigr|.\label{equpbndfilled}
\end{equation}

Let $\Bx T=\Bx T(\omega,S)=\{ b\in\sigma(B)\dvtx b\cap S\neq
\varnothing\} $, then $|\Bx T|\ge|S|m^{-d}$.
The proof is separated into cases depending on the size of $F$. We
begin with the case that $|\Bx F|$ is small.

If $|\Bx F|\le\frac{1}{2}|S|m^{-d}$ then by
the trivial lower bound on $\Bx T$, $|\Bx T\setminus\Bx F
|\ge\frac{1}{2}|S|m^{-d}$.
For any box $b\in\Bx T\setminus\Bx F$, we have $\mathbf{x},\mathbf
{y}\in b$
such that $\mathbf{x}\in S,\mathbf{y}\in S^{c}$. Since $\mathbf
{x},\mathbf{y}$
are connected in $\omega\cap b^7$ (property~2
of $\Good_{0}^{\rho}$), $\partial S\cap b^7\ne\varnothing$. For any box
$b\in\sigma(B)$, there are at most $50d$ boxes $b'\in\sigma(B)$ such
that $b^7\cap b'^7\neq\phi$.
Since $|S|\ge n^{1/3}$ and $m^{d}$ is $o(n^{1/4d})$
we have for all large $n$,
\[
|\partial S|\ge\frac{1}{50d}|\Bx T\setminus\Bx F|\ge\frac
{1}{100d}|S|m^{-d}>|S|^{1-{3}/{(4d)}}>|S|^{{(d-1)}/{d}},
\]
and are done with this case.

Our default ambient graph for sets of subboxes is $\sigma(B)$ with
the box ($\Delta$) neighbor relationship (see Section~\ref{subBox-notation}).
Thus, for $\Bx A\subset\sigma(B)$, $\Bx A^{c}=\sigma(B)\setminus\Bx A$,
$\partial\Bx A=\partial_{\sigma(B)}\Bx A$, $\partial^{\mathrm{in}}\Bx
A=\partial
_{\sigma(B)}^{\mathrm{in}}\Bx A$.
We introduce edge boundary notation
\[
\partial^{e}(\Bx Q)=\bigl\{ \bigl\{b,b'\bigr\}\dvtx b
\sim b',b\in\Bx Q,b'\in\Bx Q^{c}\bigr\}.
\]
In the case that remains, $|\Bx F|>\frac{1}{2}|S|m^{-d}$.
Note that any box $b\in\partial\Bx F$ satisfies $|b^5\cap S
|\wedge|b^5\cap S^{c}|>c'(\rho)m^{d}$.
Hence, if we knew that $\Bx F$ was a single $\Delta$-connected component
with a connected complement, we could lower bound $|\partial\Bx
F|$
and use the fact that $\partial\Bx F$ is a typical set (Percolation
property~4) to get that a constant
proportion of $\partial\Bx F$ are $(\omega,l-1,\rho\Dns)$-good
boxes. Together with our induction hypothesis, this would complete the
proof.
%
%e3 #&#
%e4 #&#
%e5 #&#
\begin{eqnarray}
|\partial S|&\ge&\bigl|\partial S\cap\partial\Bx F\cap
_{b\in\CG
_{l-1}^{\lam
\rho}(m)}\{b\}\bigr|\ge c|\partial\Bx F|m^{d-1}
\bigl(s^{l-1}(m)\bigr)^{1-d}\nonumber
\\
&\ge &c|\Bx F|^{{(d-1)}/{d}}m^{d-1}\bigl(s^{l}(n)
\bigr)^{1-d}\ge\frac
{c}{2}|S|^{{(d-1)}/{d}}m^{1-d}m^{d-1}
\bigl(s^{l-1}(m)\bigr)^{1-d}
\\
&=& \frac{c}{2}|S|^{{(d-1)}/{d}}\bigl(s^{l-1}(m)
\bigr)^{1-d}.\nonumber
\end{eqnarray}

$\Bx F$ is not in general so nice. However, being of size greater
than $\frac{1}{2}|S|m^{-d}$ implies there is a $c_{3}(\rho)>0$
and a set $\SubBoxSet K=\SubBoxSet K(\Bx F)\subset2^{\sigma(B)}$
with the following properties for all large $n$, allowing us to make
a similar isoperimetric statement:
%
%e6 #&#
\begin{eqnarray}
&\displaystyle\sum_{\Bx f\in\SubBoxSet K}\bigl(|\Bx f|\wedge\bigl|\Bx f^{c}\bigr|
\bigr)\ge c_{3}|S|m^{-d},\label{eqFalphalarge}&
\\
&\forall\Bx f\in\SubBoxSet K,\qquad \partial\Bx f\subset\Bx F^{c},\qquad \partial
^{\mathrm{in}}\Bx f\subset\Bx F,\label{eqedgeoffexactlyoneisfilled}&
\\
&\forall\Bx f_{1},\Bx f_{2}\in\SubBoxSet K,\qquad\Bx
f_{1}\neq\Bx f_{2}\quad\implies\quad \partial^{e}\Bx
f_{1}\cap\partial^{e}\Bx f_{2}=
\varnothing,\label
{eqFalphaboundariesdistinct}&
\\
&\forall\Bx f\in\SubBoxSet K,\qquad \Bx f,\Bx f^{c}\mbox{ are }\Delta\mbox
{-connected},\label{eqFalphaconn}&
\\
&n^{1/5d}<|\Bx f|\wedge\bigl|\Bx f^{c}\bigr|\le\bigl|\sigma
(B)\bigr|/2.\label{eqfminfcrightsize}&
\end{eqnarray}

First, we show how the proof follows from the existence of $\SubBoxSet K$.
Let $\Bx G=\Bx G(\omega,l,\rho)$ be the set of $(\omega,l-1,\rho
\Dns)$-good
subboxes in $\sigma(B)$. By \eqref{eqFalphaconn}, \eqref
{eqfminfcrightsize}
and Percolation property~4 (see
Section~\ref{subperc}), for all large enough $n$, for any $\Bx f\in
\SubBoxSet K$,
$|\partial\Bx f\cap\Bx G|>c_{b}(|\Bx f|\wedge
|\Bx f^{c}|)^{(d-1)/d}$.
Let $\SubBoxSet K^{\partial}=\{ \partial\Bx f\dvtx\Bx f\in
\SubBoxSet
K\} $.
By \eqref{eqFalphaboundariesdistinct}, for any $b\in\sigma(B)$,
$|\{ \Bx f\in\SubBoxSet K^{\partial}\dvtx b\in\Bx f\}
|\le2d$.
Thus,
\[
\Bigl|\bigcup\SubBoxSet K^{\partial}\cap\Bx G\Bigr|\ge\frac{1}{2d}\sum
_{\Bx f\in\SubBoxSet K}c_{b}\bigl(|\Bx f|\wedge\bigl|\Bx
f^{c}\bigr|\bigr)^{(d-1)/d}.
\]
By subadditivity of $x^{\beta}$ where $\beta<1$ and \eqref{eqFalphalarge}
this gives
%
%e11 #&#
\begin{equation}
\label{eqksize} \Bigl|\bigcup\SubBoxSet K^{\partial}\cap\Bx G\Bigr|\ge c\biggl[\sum
_{\Bx f\in\SubBoxSet K}\bigl(|\Bx f|\wedge\bigl|\Bx f^{c} \bigr|
\bigr)\biggr]^{(d-1)/d}\ge c'\frac{|S|^{(d-1)/d}}{m^{d-1}}.
\end{equation}

Let $\Bx A\subset\bigcup\SubBoxSet K^{\partial}\cap\Bx G$, be a subset
of size $|A|>c|\bigcup\SubBoxSet K^{\partial}\cap\Bx G|$, satisfying
that for any distinct $b_{1},b_{2}\in\Bx A$,
$b_{1}^7\cap b_{2}^7=\varnothing$, for example, $A=(\bigcup\SubBoxSet
K^{\partial}\cap\Bx G)\cap\Delta^{-1}(20\cdot\Z^d)$. By \eqref
{eqedgeoffexactlyoneisfilled},
for any $b\in\Bx A$, $b\in\Bx F^{c}$ but has a $\Delta$-neighbor
$b'\in\Bx F$, implying $|S\cap b^5|\wedge|S^{c}\cap
b^5|\ge c(\hat{\alpha},\rho)m^{d}=\hat{c}m^{d}$.
Since $\Bx A\subset\Bx G$, using our induction assumption and that
$|S|>r$,
%
%e12 #&#
\begin{eqnarray}
|\partial S|&\ge&\Bigl|\partial S\cap\bigcup\Bx F\Bigr|\stackrel{\scriptsize{\protect\eqref
{eqedgeoffexactlyoneisfilled}}} {\ge}|\partial S\cap A|\stackrel {\scriptsize{\protect
\eqref{eqksize}}} {\ge}c'\frac{|S|^{(d-1
)/d}}{m^{d-1}}m^{d-1}
\bigl(s^{l-1}(m)\bigr)^{1-d}
\nonumber
\\[-8pt]
\\[-8pt]
\nonumber
& =& c'|S|^{{(d-1)}/{d}}
\bigl(s^{l}(n)\bigr)^{1-d}
\end{eqnarray}
and we are done.

We return to proving the existence of $\SubBoxSet K$.

Recall, a $\Delta$-component of a set $\Bx Q\subset\sigma(B)$ is
a maximal connected component in $\Bx Q$ according to the box neighbor
relationship (see Section~\ref{subBox-notation}). Let $\SubBoxSet F$ be
the set of $\Delta$-components of $\Bx F$. Since $\Bx F\neq\sigma(B)$,
for any $\Bx f\in\SubBoxSet F$, there exists $b\in\Bx f$ with a
$\Delta$-neighbor $b'\in\Bx f^{c}$, such that $b'\subset b^5$.
As before, by property~2 of $\Good_{0}^{\rho}$
(see Section~\ref{subk-good}), $b^7\cap\partial S\neq\varnothing$. Letting
$\Bx F^{\partial}=\{ b\in\Bx F\dvtx b^7\cap\partial S\neq
\varnothing
\} $,
we then have $|\Bx F^{\partial}|\ge|\SubBoxSet F|$.
Since we can extract a subset $\Bx A\subset\Bx F^{\partial}$
where $|\Bx A|>c|\Bx F^\partial|$, and for any distinct
$b_{1},b_{2}\in
\Bx A$, $b_{1}^7\cap b_{2}^7=\varnothing$,
we only need deal with the case $|\SubBoxSet F|<|S
|^{1-{1}/{(2d)}}$.
Let $\SubBoxSet H$ be the set of $\Delta$-components of $\Bx F^{c}$.
In the same way, we may assume $|\SubBoxSet H|<|S
|^{1-{1}/{(2d)}}$.
By \eqref{equpbndfilled}, $|\Bx F^{c}|>c_{2}|\sigma
(B)|>2c_{3}|S|m^{-d}$.
We also assumed $|\Bx F|>\frac{1}{2}|S|m^{-d}$,
so w.l.o.g. $c_{3}<1/4$ and
%
%e13 #&#
\begin{equation}
\bigl|\Bx F^{c}\bigr|, |\Bx F|>2c_{3}|S |m^{-d}.\label{eqFalphaandcomplementlarge}
\end{equation}

Let $\widehat{\SubBoxSet F}=\{ \Bx f\in\SubBoxSet F\dvtx|\Bx
f|\ge  c_{3}|S|^{{1}/{(2d)}}m^{-d}\} $
and let $\widehat{\SubBoxSet H}=\{ \Bx h\in\SubBoxSet H\dvtx|\Bx
h|\ge\break  c_{3}|S|^{{1}/{(2d)}}m^{-d}\} $.
We assumed $|\SubBoxSet F|,|\SubBoxSet H|<
|S|^{1-{1}/{(2d)}}$,
and thus $\bigcup(\SubBoxSet F\setminus\widehat{\SubBoxSet
F}),\bigcup(\SubBoxSet H\setminus\widehat{\SubBoxSet
H})<c_{3}|S|m^{-d}$.
So, from \eqref{eqFalphaandcomplementlarge}, we get
%
%e14 #&#
\begin{equation}
\Bigl|\bigcup\widehat{\SubBoxSet F}\Bigr|,
\Bigl|\bigcup\widehat {\SubBoxSet H}\Bigr|>c_{3}|S|m^{-d}.\label
{eqFalphacomponentslarge}
\end{equation}
Let
\[
\SubBoxSet K=\bigl\{ \Bx f\subset\sigma(B)\dvtx\Bx f\mbox{ is a }\Delta \mbox
{-component of }\Bx h^{c}, \Bx h\in\SubBoxSet H,|\Bx f |\wedge\bigl|\Bx
f^{c}\bigr|>c_{3}|S|^{
{1}/{(2d)}}m^{-d}\bigr\}.
\]
Let $U\dvtx\SubBoxSet K\to\SubBoxSet H$ where for $\Bx f\in
\SubBoxSet K$,
$U(\Bx f)$ is the unique element in $\SubBoxSet H$ for which $\Bx f$
is a $\Delta$-component of $U(\Bx f)^{c}$. For each $\Bx f\in
\SubBoxSet K$,
$\partial\Bx f\subset U(\Bx f)\subset\Bx F^{c}$ and because $U(\Bx f)$
is a component of $\Bx F^{c}$, $\partial^{\mathrm{in}}\Bx f\subset\Bx F$,
giving us \eqref{eqedgeoffexactlyoneisfilled}. Let $\Bx h\in
\widehat{\SubBoxSet H}$.
For any $\Bx f\in\widehat{\SubBoxSet F}$, $\Bx f\subset\Bx h^{c}$
and thus $\Bx f$ is contained in some $\Delta$-component of $\Bx h^{c}$
which we denote $\hat{\Bx f}$. Since $\Bx h\subset\hat{\Bx f}^{c}$
and $\Bx f\subset\hat{\Bx f}$ we get $\hat{\Bx f}\in\SubBoxSet K$
and in particular, $\hat{\Bx f}\in U^{-1}(\Bx h)$. Thus, for any $\Bx
h\in\widehat{\SubBoxSet H}$,
$\bigcup\widehat{\SubBoxSet F}\subset\bigcup U^{-1}(\Bx h)$. In
Figure~\ref{figcompofcomp}, we give an example of some $\Bx F$ and the
resulting {\small$\SubBoxSet K$}.

%f4 #&#
\begin{figure}

\includegraphics{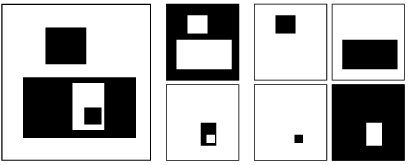}

\caption{Example of $\Bx F$ and resulting
$\SubBoxSet K=\{\mathbf{f}_{1},\mathbf{f}_{2},\mathbf{f}_{3},\mathbf{f}_{4}\}$.
$\left[ \mathbf{F}\hspace*{4pt} {\mathbf{h}_{1} \enskip \mathbf{f}_{1} \enskip \mathbf{f}_{2}\atop
\mathbf{h}_{2} \enskip \mathbf{f}_{3}\enskip \mathbf{f}_{4}}
\right]$
where the sets are in black and $ \mathbf{h}_{1}=U(\mathbf{f}_{1})=U(\mathbf{f}_{2})$,
$\mathbf{h}_{2}=U(\mathbf{f}_{3})=U(\mathbf{f}_{4})$.}\label{figcompofcomp}
\end{figure}

We regroup terms in the sum and use the fact that for any $\Bx h\in
\SubBoxSet H,\Bx f\in U^{-1}(\Bx h)$,
we have $\Bx h\subset\Bx f^{c}$ to get:
\[
\sum_{\Bx f\in\SubBoxSet K}\bigl(|\Bx f|\wedge\bigl|\Bx f^{c}\bigr|
\bigr)\ge\sum_{\Bx h\in\widehat{\SubBoxSet H}}\sum_{\Bx f\in
U^{-1}(\Bx h)}
\bigl(|\Bx f|\wedge\bigl|\Bx f^{c}\bigr|\bigr)\ge\sum
_{\Bx h\in\widehat{\SubBoxSet H}}\sum_{\Bx f\in U^{-1}(\Bx h)}\bigl( |\Bx f|\wedge|
\Bx h|\bigr).
\]
If there exists $\Bx h^{*}\in\widehat{\SubBoxSet H}$ such that for
any $f\in U^{-1}(\Bx h^{*})$, $|\Bx h^{*}|\ge|\Bx
f|$,
we have
\[
\sum_{\Bx f\in\SubBoxSet K}\bigl(|\Bx f|\wedge\bigl|\Bx f^{c}\bigr|
\bigr)\ge\sum_{\Bx f\in U^{-1}(\Bx h^{*})}|\Bx f |=\Bigl|\bigcup U^{-1}
\bigl(\Bx h^{*}\bigr)\Bigr|\ge\Bigl|\bigcup\widehat {\SubBoxSet F}\Bigr|.
\]
If none such exists, then
\[
\sum_{\Bx f\in\SubBoxSet K}\bigl(|\Bx f|\wedge\bigl|\Bx f^{c}\bigr|
\bigr)\ge\sum_{\Bx h\in\widehat{\SubBoxSet H}}|\Bx h |=\Bigl|\bigcup\widehat{\SubBoxSet
H}\Bigr|.
\]
Thus, from \eqref{eqFalphacomponentslarge}, we get \eqref{eqFalphalarge}.
Next, for $\Bx f_{1}\in\SubBoxSet K$, any edge $\{ b,\hat{b}
\} \in\partial^{e}\Bx f_{1}$
satisfies w.l.o.g. $\hat{b}\in U(\Bx f_{1})$ and $b\in\Bx f_{1}$. Thus,
if $\Bx f_{2}\in\SubBoxSet K$ shares the edge $\{ b,\hat{b}
\} $
with $\Bx f_{1}$, then $U(\Bx f_{1})=U(\Bx f_{2})$ and since $b\in\Bx
f_{1}\cap\Bx f_{2}$
and both are $\Delta$-components of $U(\Bx f_{1})^{c}$, we have
$\Bx f_{1}=\Bx f_{2}$, giving us \eqref{eqFalphaboundariesdistinct}.
To get \eqref{eqFalphaconn}, let $\Bx h\in\SubBoxSet H$, and let
$\Bx h^{c}=\Bx f_{1}\cup\cdots\cup\Bx f_{n}$ where $\Bx f_{i}$ are
the $\Delta$-components of $\Bx h^{c}$. Then $\forall i,\partial\Bx
f_{i}\subset\Bx h$,
and since $\Bx h$ is connected, $\Bx f_{i},\Bx f_{j}$ are connected
in $\Bx f_{i}\cup\Bx f_{j}\cup\Bx h$ for any $i,j$. This implies
$\Bx f_{i}^{c}=\Bx h\cup\Bx f_{1}\cup\cdots\cup\Bx f_{n}\setminus
\Bx f_{i}$
is $\Delta$-connected for any $i$. Last, since $|\Bx f|\wedge
|\Bx f^{c}|>c_{3}|S|^{{1}/{(2d)}}m^{-d}$
and $m^{d}$ is $o(n^{1/20d)}$, we get \eqref{eqfminfcrightsize}.
\end{pf}
In the below corollary, we transfer the isoperimetric bounds on
$\varphi$
from the setting of a box to a torus. The main idea of the proof is
to show that given any large set $S$ in a $(\omega,k,\rho)$-good
torus, there are two neighboring top-level boxes which have a large
intersection with $S$ and $\omega\setminus S$.
%
%co3.4 #&#
\begin{cor}
\label{corkgoodtorusisopbnd}Let $\omega\subset\mathcal{T}(N)$.
If $\mathcal{T}(N)$ is a $(\omega,k,\rho)$-good torus then for all
large enough $N$, and $r\ge N$
\begin{eqnarray*}
\label{pg15}\hat{\phi}(r) & = & \inf\biggl\{ \frac{|\partial_{\omega}S|}{
|S|}\dvtx S\subset
\omega,N^{{1}/{3}}\le|S|\le r\wedge \biggl(1-\frac{1}{4d}\biggr)|
\omega|\biggr\}
\\
& > & c(k,\rho)\frac{r^{-1/d}}{(s^{(k)}(N))^{d-1}}.
\end{eqnarray*}
\end{cor}
\begin{pf}
Let $\omega^{+}=\proj^{-1}(\omega)\cap B^{3}(N)$. Recall
from Section~\ref{subk-goodtorus} that all top-level boxes for $\mathcal{T}(N)$
are $(\omega^{+},k,\rho)$-good, so by property~1
of $\Good_{0}^{\rho}$, for any top-level box $B$, there is a
$c_{1}(\rho)>0$
such that
%
%e15 #&#
\begin{equation}
\bigl|B\cap\omega^{+}\bigr|>c_{1}N^{d}.\label{eqtopleveldense}
\end{equation}
Fix $r\ge N$. By construction, $\frac{1}{2d}|\omega^{+}
|=|\omega|\ge|B\cap\omega^{+}|$
for any top-level box $B$. We assume that $N$ is large enough so
that $c_{1}N^{d-1}>4d$, and $|B\cap\omega^{+}|>4\,dN$.
In particular, this implies that the infimum is not on an empty set.
Let $S$ satisfy the conditions to be a candidate for the infimum
in $\hat{\phi}(r)$ and extend it to $S^{+}=\proj^{-1}(S)\cap B^{3}(N)$.
Let $\hat{r}=|S|\wedge|\omega\setminus S|$.
Again by \eqref{eqtopleveldense}, for each top-level box $B$,
$|B\cap S^{+}|\vee|B\cap(\omega^{+}\setminus
S^{+})|\ge\frac{1}{2}c_{1}N^{d}>c_{2}\hat{r}$.
On the other hand, since there are $10^{d}$ top-level boxes whose
union covers $B(N)$, by the pigeonhole principle, there must be some
box $B$ for which $|B\cap S^{+}|\ge10^{-d}|S|$
and likewise a box $B'$ for which $|B'\cap(\omega^{+}\setminus
S^{+})|\ge10^{-d}|(\omega\setminus S)|$.
Let $c_{3}=c_{2}\wedge10^{-d}$. Since the top-level boxes are $\Delta
$-connected,
there are two $\Delta$-neighboring top-level boxes $B_{1},B_{2}$
such that $|B_{1}\cap S^{+}|,|B_{2}\cap(\omega
^{+}\setminus S^{+})|\ge c_{3}\hat{r}$.
This implies $|B_{1}^5\cap S^{+}|\wedge|B_{1}^5\cap
(\omega^{+}\setminus S^{+})|\ge c_{3}\hat{r}$.
By construction, $|\partial_{B_{1}^{7}\cap\omega^{+}}S^{+}
|\le|\partial_{\omega}S|$.
Since $B_{1}$ is $(\omega^{+},k,\rho)$-good, we can use Theorem~\ref
{lemphilowbnd}
to lower bound $|\partial_{B_{1}^7\cap\omega^{+}}S^{+}|$
by $c\hat{r}^{{(d-1)}/{d}}(s^{(k)}(N))^{1-d}$
for all large~$N$. Note that as $|\omega|>4\,dN$, implying
$|\omega\setminus S|\ge N$, we have $\hat{r}\ge N$.
Since $|\omega\setminus S|\ge\frac{1}{4d}|\omega
|>\frac{1}{4d}|S|$,
we can bound $|S|$, the denominator in the infimum, from
above by $4d\hat{r}$, giving us $\frac{|\partial_{\omega}S
|}{|S|}\ge c\hat{r}^{-1/d}(s^{(k)}(N))^{1-d}$.
Since $\hat{r}\le r$ we are done.
\end{pf}
We now proceed to prove the main theorem of this section.
\begin{pf*}{Proof of Theorem~\ref{thmmmixingbound}}
The following proof makes assumptions which are valid for all but
a finite number of $N$, and those are resolved by the large constant
above. Note that $\omega_{0}$ is viewed as a subgraph of $\mathcal{T}(N)$
as far as connectivity is concerned. We present an upper bound to the
mixing time $\tau$ of
$X(t)$ using average conductance, a method developed in \cite{lovasz1999faster}
and refined in subsequent papers.

We follow notation of \cite{morris2005evolving}. Let $\pi(\cdot)$
be the stationary distribution of $X(t)$ and for $\mathbf{x},\mathbf
{y}\in\omega_{0}$
let $Q(\mathbf{x},\mathbf{y})=\pi(\mathbf{x})p(\mathbf{x},\mathbf{y})$.
For $S,A\subset\omega_{0}$ let $Q(S,A)=\sum_{\mathbf{s}\in
S,\mathbf{a}\in A}Q(\mathbf{s},\mathbf{a})$.
Let $\Phi_{S}=\frac{Q(S,S^{c})}{\pi(S)}$\label{pg16} and let $\Phi(u)=\inf\{
\Phi_{S}\dvtx0<\pi(S)\le u\wedge\frac{1}{2}\} $.\label{pg17}
Let $\pi_{*}=\min_{\mathbf{x}\in\omega_{0}}\pi(\mathbf{x})$.

By \cite{morris2005evolving},
%
%e16 #&#
\begin{equation}
\tau=\tau\biggl(\omega_{0},\frac{1}{4}\biggr)\le1+\int
_{4\pi_{*}}^{16}\frac
{4\,du}{u\Phi^{2}(u)}.\label{eqtaustatdist}
\end{equation}
Recall the notation from Section~\ref{subGraph-notation}. In this proof,
our ambient graph is $\omega_{0}$ and thus $S^{c}=\omega_{0}\setminus S$
and $\partial S=\partial_{\omega_{0}}S$. To simplify notation in
the proof, we restate \eqref{eqtaustatdist} in terms of internal
volume and boundary size.

For $S\subset\omega_{0}$, if $\pi(S)\le u$, then we have by definition
$u\ge\sum_{\mathbf{v}\in S}\deg(\mathbf{v})\times [\sum_{\mathbf{v}\in\omega_{0}}\deg(\mathbf{v})]^{-1}$.
Using the bound on degree and connectedness of $\om_0$, we get $
|S|\le2ud|\omega_{0}|$.
In the same way, $2d\frac{|S^{c}|}{|\omega_{0}
|}>\pi(S^{c})\ge1-u$
which gives $|S|\le(1-\frac{1}{2d}(1-u))
|\omega_{0}|$,
and thus for $u\le\frac{1}{2}$,
%
%e17 #&#
\begin{equation}
|S|\le2ud|\omega_{0}|\wedge\biggl(1-\frac{1}{4d}\biggr) |
\omega_{0}|.\label{equpbndonsintermsofu}
\end{equation}

Let $\phi_{S}=\frac{|\partial S|}{|S|}$. Since
$\omega_{0}$ is a bounded degree graph and $\mathbf{x}\sim\mathbf
{y}\iff
\frac{1}{4d}\le p(\mathbf{x},\mathbf{y})\le\frac{1}{2}$,
for some $C(d)$ and all $S\subset\omega_{0}$ we have $\phi_{S}<C\Phi_{S}$.
%
%$\partial S$ and $\partial S^{c}$ are the same up to a factor of
%$2d$ and $\pi()$ is the same on $\omega_{0}$ up to a factor of
%$2d$
%
{} Let $\phi(r)=\inf\{ \phi_{S}\dvtx0<|S|\le r\wedge(1-\frac
{1}{4d})|\omega_{0}|\} $.
Then by \eqref{equpbndonsintermsofu} the infimum in $\phi
(2ud|\omega_{0}|)$
is on a larger set than the infimum in $\Phi(u)$ giving us $\phi
(2ud|\omega_{0}|)<C\Phi(u)$.
Thus, by the change of variables $r=2ud|\omega_{0}|$ in
\eqref{eqtaustatdist}, we get%
%%
%%
%|\omega_{0}|)\} <\inf\{ \phi_{S}\dvtx\pi(S)\le\min(u,\frac
%{1}{2})\} <C\inf\{ \Phi_{S}\dvtx\pi(S)\le\min(u,\frac
%{1}{2})\}
%
%because inf on a larger set $\{ \pi(S)\le\min(u,\frac{1}{2})
%|,(1-\frac{1}{4d})|\omega_{0}|)\} $
%
%{}
%
%e18 #&#
\begin{equation}
\tau<C\int_{1}^{32dN^{d}}\frac{dr}{r\phi^{2}(r)}.\label{eqmixtimeineq}
\end{equation}

We continue by showing that for our purposes, a rough estimate of
$\phi_{S}$ for sufficiently small sets $S$ is enough. Let
\[
\hat{\phi}(r)=\inf\biggl\{ \phi_{S}\dvtx N^{{1}/ {3}}
\le|S|\le r\wedge\biggl(1-\frac{1}{4d}\biggr)|\omega_{0}|\biggr
\},
\]
where the infimum of an empty set is $\infty$. Since $\omega_{0}$
is connected (see Remark~\ref{remgoodtorusconnected}), $\phi(r)\ge1/r$
for any $1\le r<|\omega_{0}|$. For large $N$, by property~1 of $\Good_{0}^{\rho}$ (see Section~\ref{subk-good}),
$N<(1-\frac{1}{4d})|\omega_{0}|$. Thus,
\begin{eqnarray*}
\phi(r) & = & \inf\bigl\{ \phi_{S}\dvtx|S|\le r\wedge
N^{ {1}/ {3}}\bigr\} \wedge\hat{\phi}(r)
\\
& \ge& \bigl[r^{-1}\vee N^{-{1}/{3}}\bigr]\wedge\hat{\phi}(r).
\end{eqnarray*}
By Corollary~\ref{corkgoodtorusisopbnd} below, $\hat{\phi
}(r)>c(k,\rho)(s^{(k)}(N))^{1-d}r^{-1/d}$.
Integrating \eqref{eqmixtimeineq} with the above lower bound for
$\phi(r)$, we thus get
\begin{eqnarray*}
\tau& < & C\int_{1}^{32N^{d}}\frac{dr}{r(N^{-{1}/{3}}
)^{2}}+C\int
_{10^{d}N^{{1}/{3}}}^{32N^{d}}\frac{dr}{r\hat{\phi
}^{2}(r)}
\\
& < & o\bigl(N^{2}\bigr)+C\bigl(s^{(k)}(N)
\bigr)^{2d-2}N^{2}=o\bigl(\bigl(\log ^{(k-1)}N\bigr)
\bigr)N^{2}
\end{eqnarray*}
as required.
\end{pf*}

%%%%%%%%%%%%%%%%%%%%%%%%%%%%%%%%%%%%%%%%%%%%%%%%%%%%%%%%%%%%%%%%%%%%%%%%%%%%%%%%%%%%%%%%%%%%
%%%%%%%%%%%%%%%%%%%%%%%% Section~6 High density percolation percolates
%%%%%%%%%%%%%%%%%%%%%%%%%%%%%%%%%%%%%%%%%%%%%%%%%%%%%%
%%%%%%%%%%%%%%%%%%%%%%%%%%%%%%%%%%%%%%%%%%%%%%%%%%%%%%%%%%%%%%%%%%%%%%%%%%%%%%%%%%%%%%%%%

%s4 #&#
\section{High density
percolation
percolates}\label{secSupercritical-percolation-percolates}

This section presents results used in the renormalization arguments
of Section~\ref{secRenormalization}. See Section~\ref{subperc} for
the properties of percolating configuration.
Note that many of the lemmas in this section deal with i.i.d. Bernoulli
percolation.
%
%le4.1 #&#
\begin{lem}
\label{lemSCPercpercolates}For $n\in\mathbb{N}$, let $\{
Y(\mathbf{z})\} _{\mathbf{z}\in B(n)}$
be i.i.d. $\{ 0,1\} $ r.v.'s, and write $\mathcal{S}(n)=
\{ \mathbf{z}\in B(n)\dvtx Y(\mathbf{z})=1\} $
for the random support of $Y$. Then there are dimensional dependent
constants, $C>0$ and $p_{b}<1$, such that if $\Pr[Y(\mathbf
{0})=1]=p_{b}$,
\[
\Pr\bigl[\mathcal{S}(n)\in\mathcal{P}(n)\bigr]\ge1-
\frac{C(\log n)^{d-1}}{n^d}.
\]
%
%polynomially fast but summable.
\end{lem}
\begin{pf}
Lemmas \ref{lemlargestcomponent}, \ref{lemupbndsizecmpincomplement},
\ref{lemAntalPiszforpercinbox} and Corollary~\ref
{corstarboundariestypical}
prove Percolation properties~1--4, respectively.
\end{pf}
The next lemma assures a percolation configuration given a finite range
dependance requirement.
%
%co4.2 #&#
\begin{cor}
\label{corDominatingfieldpercolates}For $n\in\mathbb{N}$, let
$\{ Y(\mathbf{z})\} _{\mathbf{z}\in B(n)}$ be $\{
0,1\} $
r.v.'s, not necessarily i.i.d., and write $\mathcal{S}(n)=\{
\mathbf{z}\in B(n)\dvtx Y(\mathbf{z})=1\} $
for the random support of $Y$. Assume the r.v.'s have the property
that for any $\mathbf{x}\in B(n)$ and any $A\subset B(n)\setminus
b(\mathbf{x},20)$,
\[
\Pr\bigl[Y(\mathbf{x})=1 | \mathcal{S}\cap B(n)\setminus b(\mathbf {x},20)=A
\bigr]>p_{d},\vadjust{\goodbreak}
\]
where $p_{d}<1$ is a fixed constant dependent only on $p_{b}$ (from
Lemma~\ref{lemSCPercpercolates}) and dimension.

Then for all $p<1$, there is a $C(p)<\infty$ such that for all $n>C$,
\[
\Pr\bigl[\mathcal{S}(n)\in\mathcal{P}(n)\bigr]>p.
\]
\end{cor}
\begin{pf}
The domination of product measures result of Liggett, Schonmann and
Stacey \cite{liggett1997domination}, implies there is a $p_{d}<1$
for which $\mathcal{S}(n)$ stochastically dominates an i.i.d. product
field with density $p_{b}$ on $B(n)$. Lemma~\ref{lemSCPercpercolates}
tells us that the probability such an i.i.d. field belongs to $\mathcal{P}(n)$
approaches one as $n$ tends to infinity. Since Percolation properties
are monotone (Claim~\ref{clapercolationmonotone}), we are done.%
%
%Result states $\mathcal{S}(\infty)$ dominates $\mathbf{P}_{p_{b}}(
%For $\mathcal{S}(n)$, attached i.i.d. to $\mathbb{Z}^{d}\setminus B(n)$
%then condition holds and use result.
%%
%{}
\end{pf}
Write $\mathbf{P}_{p}[\cdot]$ for the law that makes $
\{ Y(\mathbf{z})\} _{\mathbf{z}\in\mathbb{Z}^{d}}$
i.i.d. $\{ 0,1\} $ r.v.'s where $Y(\mathbf{z})=1$ w.p. $p$.
Let $B=B(n)$ and write $\mathcal{S}=Y^{-1}(1)\cap B$ for the random
set of open sites in $B$. Denote by $\mathcal{C}$ the largest connected
component in $\mathcal{S}$.

We write a consequence of Theorem~1.1 of \cite{deuschel1996surface}.
One can find the proof in the appendix of \cite
{procaccia2011concentration}.
%
%le4.3 #&#
\begin{lem}
\label{lemlargestcomponent}There is a $p_{0}(d)<1$ such that for every
$p>p_{0}$, there exists a $c>0$ such that
\[
\mathbf{P}_{p}\bigl[|\mathcal{C}|<\bigl(1-10^{-d}
\bigr)\bigl|B(n)\bigr |\bigr]\le ce^{-cn}.
\]
\end{lem}
%
%de4.4 #&#
\begin{defn}
Let $B^{*}$ be the graph of $B(n)$ where we add edges between any
two vertices in $B$ of $l_{\infty}$ distance one. We call a set
$A$ in $B$ \emph{$\ast$-connected}, if it is connected in $B^{*}$.
\end{defn}
%
%le4.5 #&#
\begin{lem}
\label{lemLargestarsetsaretypical}There is a $\beta_1,\beta
_2,c_{d}(d)>0,\beta_3,p_{1}(d)<1$
and $C(d)<\infty$ such that for any $p>p_{1}$
%
%e19 #&#
\begin{equation}
\mathbf{P}_{p}\bigl[\exists A,\ast\mbox{-connected}, |A|>C
\log n,|A\cap\mathcal{S}|<c_d|A| \bigr]\le\beta_1e^{-\beta_2n^{\beta_3}}\label{eqexistsAfewopen}.
\end{equation}
\end{lem}
\begin{pf}
Fix a vertex $\mathbf{v}\in B$ and let $A$ be $\ast$-connected
such that $\mathbf{v}\in A$ and $|A|=k$. The number of
such components is bounded by $(3^{d}-1)^{2k}<e^{\hat{c}k}$.
To see this, fix a spanning tree for each such set and explore the
tree starting at $\mathbf{v}$ using a depth first search. Each edge
is crossed at most twice and at each step the number of directions
is bounded by the degree. Using Cram\'er's theorem for i.i.d. (large
deviations), for large enough $p_{1}(d)<1$ and small enough
$p_1(d)>c_{d}(d)>0$, $\mathbf{P}_{p}[|A\cap\mathcal{S}
|<c_{d}|A|]<\exp(-2\hat{c}|A|)$.
To bound the probability of the event in~\eqref{eqexistsAfewopen},
we union bound over $\ast$-connected components larger than $n^{1/3}$
that contain a fixed vertex in $B$ to get
\[
n^{d}\sum_{k\ge n^{1/3}}e^{\hat{c}k}e^{-2\hat{c}k},
\]
which is smaller than $\beta_1e^{-\beta_2n^{\beta_3}}$ for appropriate
constants.
\end{pf}
%
%co4.6 #&#
\begin{cor}
\label{corstarboundariestypical}There is a $c_{b}>0$, $C_{b}<\infty$
such that for all $p>p_{1}(d)$, with probability greater than $1-\beta
_1e^{-\beta_2n^{\beta_3}}$, any connected set $A\subset B$
such that $B\setminus A$ is also connected and $C_{b}\log^{{d}/{(d-1)}}n<|A|\le n^{d}/2$.
\[
|\partial_{B}A\cap\mathcal{S}|,\bigl|\partial_{B}^{\mathrm{in}}A
\cap \mathcal{S}\bigr|>c_{b}|A|^{{(d-1)}/{d}}.
\]
\end{cor}
\begin{pf}
By Lemma~2.1(ii) in \cite{deuschel1996surface}, $\partial
_{B}A,\partial
_{B}^{c}A$
are $\ast$-connected. By well-known isoperimetric inequalities for
the grid; see, for example, Proposition~2.2 in \cite{deuschel1996surface},
there is a $c_{I}>0$ such that for $|A|\le n^{d}/2$,
$|\partial_{B}A|,|\partial_{B}^{c}A|>c_{I}
|A|^{{(d-1)}/{d}}$.
For appropriate $C_{b}$, $c_{I}|A|^{{(d-1)}/{d}}>C_{a}\log n$,
and thus Lemma~\ref{lemLargestarsetsaretypical} gives the result
with $c_{b}=c_{I}c_{d}$.
\end{pf}
%
%le4.7 #&#
\begin{lem}
\label{lemupbndsizecmpincomplement}Let $\mathcal{K}$ denote
the largest connected component in $B\setminus\mathcal{C}$. There
are $c>0$, $\gamma<1$ and $p_{2}(d)<1$ such that for all $p>p_{2}$,
\[
\mathbf{P}_{p}\bigl[|\mathcal{K}|>
\log^{2}n\bigr]\le e^{-cn^\gamma}.
\]
%
%stretch exponentially fast with $n$.
\end{lem}
\begin{pf}
Choose a component $\mathcal{K}$ of $B\setminus\mathcal{C}$. Since
$\mathcal{C}$ is connected and $\mathcal{K}$ is maximal, $B\setminus
\mathcal{K}$
is also connected. This easy fact is proved in Theorem~\ref{lemphilowbnd}.
From Lemma~\ref{lemlargestcomponent}, we have for $p>p_{0}$, $k=
|\mathcal{K}|<|B|/2$.
It is not true in general that $Y(\mathcal{K})=0$ but since $\partial
_{B}^{\mathrm{in}}\mathcal{K}$
separates $\mathcal{K}$ from $\mathcal{C}$, $Y(\partial
_{B}^{\mathrm{in}}\mathcal{K})=0$.
Thus, from Corollary~\ref{corstarboundariestypical}, for $p_{2}>p_{1}$,
w.h.p., $|\mathcal{K}|<C_{b}\log^{{d}/{(d-1)}}n$.
\end{pf}
%
%le4.8 #&#
\begin{lem}
\label{lemAntalPiszforpercinbox}There is a $c_{a}>0$ such
that for $p>p_{1}>p_{c}$
\[
\mathbf{P}_{p}\bigl[\exists\mathbf{v},\mathbf{w}\in\mathcal{C}\cap
B(n-c_{a}\log n), d_{\mathcal{C}}(\mathbf{v},\mathbf
{w})>c_{a}\bigl(d_{B}(\mathbf{v},\mathbf{w})\vee\log n
\bigr)\bigr]\le\frac
{C(\log n)^{d-1}}{n^d}.
\]
\end{lem}
\begin{pf}
Recall $Y(\mathbf{z})$ are defined for all $\mathbf{z}\in\mathbb{Z}^{d}$.
Let $\mathcal{C}_{\infty}$ be the infinite component of $Y^{-1}(1)$.
We start by showing that w.h.p., $\mathcal{C}$, the largest cluster
in $Y^{-1}(1)\cap B$ is contained in $\mathcal{C}_{\infty}$. By
Lemma~\ref{lemlargestcomponent}, the diameter of $\mathcal{C}$
is of order $n$ w.h.p. If in this case $\mathcal{C}\nsubseteq\mathcal
{C}_{\infty}$,
then $\mathcal{C}$ is a finite cluster in $Y^{-1}(1)$ of diameter
$n$. In the supercritical phase ($p>p_{c}$), the probability for
such a cluster at a fixed vertex decays exponentially in $n$ (see,
e.g., 8.4 in \cite{grimmett1999percolation}). Thus we may union bound
over the vertices of $B$ to get that w.h.p.
%
%e20 #&#
\begin{equation}
\mathcal{C}\subset\mathcal{C}_{\infty}.\label{eqCcontainedininfcluster}
\end{equation}
We assume henceforth that this is the case.

Next, by Theorem~1.1 of \cite{antal1996chemical}, we have that for
some $0<k,K_{0},K<\infty$, dependent on dimension and $p_{1}$,
\[
\mathbf{P}_{p}\bigl[d_{\mathcal{C}_{\infty}}(\mathbf{x},\mathbf
{y})>K_{0}m|\mathbf{x},\mathbf{y}\in\mathcal{C}_{\infty},d(
\mathbf {x},\mathbf{y})=m\bigr]<K\exp(-km).
\]
We use this to show that for appropriate $K_{1}<\infty$, the probability
of the following event decays to $0$. Let
\[
\mathcal{A}=\bigl\{ \exists\mathbf{x},\mathbf{y}\in B\cap\mathcal
{C}_{\infty}, K_{1}\log n<d(\mathbf{x},\mathbf
{y})<K_{0}^{-1}d_{\mathcal{C}_{\infty}}(\mathbf{x},\mathbf{y})\bigr
\}.
\]
Using a union bound,
\[
\mathbf{P}_{p}[\mathcal{A}]<n^{d}\sum
_{m=K_{1}\log
n}^{\infty}Cm^{d-1}\exp(-km)<Cn^{d}(
\log n)^{d-1}n^{-2d}.
\]
%
%By Theorem~1.1, for some $M$,
%%
%{C},d_{B}(x,y)=m>M]<\exp(-km)
%%
%Union bounding the probability for this over every pair of points
%in $B(n)$ of distance greater than $c\log n$ we get
%%
%n^{d}\int_{m=c\log n}^{\infty}m^{d-1}\exp(-km)<n^{d}\log
%n^{d-1}n^{-2d}\to0
%
%%
%{}
Let $B^{-}=B(n-4dK_{0}K_{1}\log n)$. We now show that $\mathcal{A}$
not occurring implies the event $\mathcal{B}$.
\[
\mathcal{B}=\biggl\{ \forall\mathbf{x},\mathbf{y}\in\mathcal{C}\cap
B^{-}\mbox{ s.t. }1<\frac{d_{\mathcal{C}}(\mathbf{x},\mathbf
{y})}{K_{1}\log n}<4d, d_{\mathcal{C}}(
\mathbf{x},\mathbf {y})<4dK_{0}K_{1}\log n\biggr\}.
\]
From $\mathcal{A}$ not occurring and \eqref{eqCcontainedininfcluster},
we get that for any $\mathbf{x},\mathbf{y}$ satisfying the condition
in $\mathcal{B}$, $d_{\mathcal{C}_{\infty}}(\mathbf{x},\mathbf
{y})<4dK_{0}K_{1}\log n$.
Since $\mathbf{x},\mathbf{y}\in B^{-}$, a path connecting $\mathbf{x}$
to $\mathbf{y}$ in $\mathcal{C}_{\infty}$ realizing this distance
is too short to reach $\partial^{\mathrm{in}}B$, and thus by \eqref
{eqCcontainedininfcluster}
is contained in~$\mathcal{C}$.

Next, for any $\mathbf{x},\mathbf{y}\in B^{-}$, there is a sequence
of boxes $b_{1},\ldots,b_{m}$ where $\mathbf{x}\in b_{1},\mathbf
{y}\in b_{m}$
and the following conditions hold. For all $i$ for which it is defined,
$\|b_{i}\|=\lceil K_{1}\log n\rceil$, the diameter of $b_{i}\cup b_{i+1}$
is less than $4dK_{1}\log n$, $d(b_{i},b_{i+1})>K_{1}\log n$ and
for some $K_{2}<\infty$, $m<K_{2}d(x,y)/\log n+2$. The left term
in the bound for $m$ can be achieved for example by placing boxes
with order $\log n$ spacing in lines parallel to the coordinate axes.
The constant $2$ appears for the case where $d(\mathbf{x},\mathbf
{y})<K_{1}\log n$
and we use an intermediary box.

Lemma~\ref{lemupbndsizecmpincomplement} tells us that for all
large $n$, w.h.p. every box $b$ with $\|b\|\ge\log n$ intersects
$\mathcal{C}$.
Assuming that this and the high probability $\mathcal{B}$ event occur,
we have that for $\mathbf{x},\mathbf{y}$ as in $\mathcal{B}^{-}$,
$d_{\mathcal{C}}(\mathbf{x},\mathbf{y})<4dK_{0}K_{1}
(K_{2}d(x,y)+2\log n)$,
and we are done.
\end{pf}

%%%%%%%%%%%%%%%%%%%%%%%%%%%%%%%%%%%%%%%%%%%%%%%%%%%%%%%%%%%%%%%%%%%%%%%%%%%%%%%%%%%%%%%%%%
%%%%%%%%%%%%%%% Godness of random walk range
%%%%%%%%%%%%%%%%%%%%%%%%%%%%%%%%%%%%%%
%%%%%%%%%%%%%%%%%%%%%%%%%%%%%%%%%%%%%%%%%%%%%%%%%%%%%%%%%%%%%%%%%%%%%%%%%%%%%%%%%%%

%s5 #&#
\section{Goodness of random walk range}\label{secDense-Random-Walk}

%s5.1 #&#
\subsection{Random walk definitions
and notation}\label{subRandom-walk-definitions}

Given a box $B$, consider the two faces of $\partial B^7$ for
which the first coordinate is constant. We call the one for which
this coordinate is larger the \emph{top face} and call the other one
the \emph{bottom face.} Let Top$^{+}(B),\operatorname{Bot}(B)$ be the
projection
of $B^{3}$ on the top and bottom faces, respectively. Let Top$(B)$ be
the neighbors of Top$^{+}(B)$\label{pg18} inside $B^7$. Thus, Top$(B)\subset
\partial
^{\mathrm{in}}B^7$
is a translation along the first coordinate of $\operatorname{Bot}(B)\subset\partial B^7$.

Let $\mathbf{P}_{\mathbf{x}}[\cdot]$ be the law that makes
$S(\cdot)$ an independent SRW starting at $\mathbf{x}\in\mathbb{Z}^{d}$.
For a set $A\subset\mathbb{Z}^{d}$, let $\tau_{A}=\inf\{ t\ge
0\dvtx S(t)\in A\} $
be the first hitting time of $A$, and for a single vertex $\mathbf{v}$,
we write $\tau_{\mathbf{v}}=\tau_{\{ \mathbf{v}\} }$.
For $\mathbf{a}\in\operatorname{Top}(B),\mathbf{z}\in\operatorname{Bot}(B)$, we call
the ordered
pair $\eta=(\mathbf{a},\mathbf{z})$ a $B$\emph{-traversal}.\label{pg19} We write
$\mathbf{P}^{\eta}[\cdot]=\mathbf{P}_{\mathbf{a}}[\cdot
|\tau_{\partial B^7}=\tau_{\mathbf{z}}]$.

Let $H=(\eta_{1},\eta_{2},\ldots,\eta_{k})$ be an ordered
sequence of $B$-traversals. We call $H$ a $B$-\emph{itinerary}\label{pg20}
and write $\mathbb{P}_{H}=\mathbf{P}^{\eta_{1}}\times\cdots\times
\mathbf
{P}^{\eta_{k}}$
for the product probability space. For each $\eta\in H$, we denote
the associated independent conditioned random walk by $S_{\eta}(\cdot)$,
write $\mathcal{R}_{\eta}(t_{1},t_{2})=\{ S_{\eta}(s)\dvtx t_{1}\le
s\le t_{2}\} $
and simply $\mathcal{R}_{\eta}$ for $\mathcal{R}_{\eta}(0,\tau
_{\partial B^7})$.
We say $H$ is $\rho$\emph{-dense} if $|H|\ge\rho\|B\|^{d-2}$.

For a $B(\mathbf{x},n)$-itinerary $H$, we abbreviate notation inside
$\mathbb{P}_{H}[\cdot]$ by writing $\Good_{k}^{\rho}$
instead of $\Good_{k}^{\rho}(n)+\mathbf{x}$.

For a SRW $S(\cdot)$, we write $S(t_{1},t_{2})$ for the sequence
$(S(t_{1}),\ldots,S(t_{2}))$.

For $Q\subset H$ a set (subsequence) of $B$-traversals, let $\mathcal
{R}_{Q}=\bigcup_{\eta\in Q}\mathcal{R}_{\eta}$.
When in use under the law $\mathbb{P}_{H}$, we write $\mathcal{R}$
for $\mathcal{R}_{H}$.
%s5.2 #&#
\subsection{Independence of a random
walk traversing a box}\label{subIndependence-of-a-traversal}

Let\vspace*{1pt} $\mathcal{I}_{N}(\cdot)=\proj^{-1}\circ\proj(\cdot)$
and for $b=b(\lceil N/10\rceil)$ let $\mathcal{I}_{N}^{*}(\cdot
)=b^7\cap
\mathcal{I}_{N}(\cdot)$.
Since $\|b^7\|<N$, $\mathcal{I}_{N}(b^7)$ is an infinite disconnected
union of translated copies of $b^7$. Thus, we have that for any
$\mathbf{x}\in\mathcal{I}_{N}(b^7)$, $\mathcal{I}_{N}^{*}$ is
a graph isomorphism between $b^7$ and $\beta_{\mathbf{x}}$, the
component of $\mathbf{x}$ in $\mathcal{I}_{N}(b^7)$.

Given $S(\cdot)$, a simple random walk in $\mathbb{Z}^{d}$, we define
the following random set of triplets.
\begin{eqnarray*}
\mathfrak{T}_{N} & =&\bigl\{  \bigl(\gamma,\gamma^{+},
\beta\bigr)\dvtx0<\gamma <\gamma^{+},\beta\mbox{ a box},
\beta^{7}\mbox{ is a component of }\mathcal{I}_{N}
\bigl(b^7\bigr),\\
&&\hspace*{7pt}{}S(\gamma-1)\in\partial\beta^{7},
\\
& &\hspace*{50pt} {} S(\gamma)\in\operatorname{Top}(\beta),\mathcal{R}\bigl(\gamma,\gamma
^{+}-1\bigr)\subset\beta^{7},S\bigl(\gamma^{+}
\bigr)\in\operatorname{Bot}(\beta)\bigr\}.
\end{eqnarray*}
For any two distinct copies of $b^7$ in $\mathcal{I}_{N}(b^7)-\beta,\hat{\beta}$ we have $\partial\beta^{7}\cap\partial\hat
{\beta
}^{7}=\varnothing$.
Thus, for any two distinct triplets $(\gamma,\gamma^{+},\beta)$,
$(\hat
{\gamma},\hat{\gamma}^{+},\hat{\beta})$
either $\gamma>\hat{\gamma}^{+}$ or $\hat{\gamma}>\gamma^{+}$. Ordering
the triplets by increasing first coordinate, we write $(\gamma
_{i},\gamma_{i}^{+},\beta_{i})$
for the $i$th triplet by this order.

Since $\mathfrak{T}_{N}$ may be defined in terms of the finite state
Markov process $S_{N}(\cdot)$, $\mathbf{P}_{\mathbf{x}}[
|\mathfrak{T}_{N}|=\infty]=1$.
Thus, for $\rho>0$, $\gamma_{\lceil\rho n^{d-2}\rceil}^{+}$ is well defined.
%
%de5.1 #&#
\begin{defn}
Let $\tau_\rho(b)=\gamma^+_{\lceil\rho n^{d-2}\rceil}$.\label{pg21}
\end{defn}
The next lemma claims the following: Run a SRW up to time $uN^d$ from a
point $x\in B^{10}$. There exists a constant $\rho(u)$ such that with
high probability there are at least $\rho N^{d-2}$ traversals from Top
to Bot. See Figure~\ref{figtopbot} for graphical representation.

%
%f5 #&#
\begin{figure}

\includegraphics{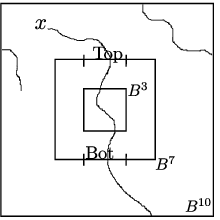}

\caption{Traversal and Top, Bot definition.}\label{figtopbot}
\end{figure}

%le5.2 #&#
\begin{lem}\label{lemforeveryurthereisrho}
\label{lemtoplevelhitnd-2times}For any $u>0$, there is a $\rho(u)>0$
such that
\[
\mathbf{P}_{\mathbf{x}}\bigl[\tau_\rho(b)<uN^{d}\bigr]
\stackrel{N} {\to}1,
\]
uniformly for any $\mathbf{x}\in\mathbb{Z}^{d}$.
\end{lem}
\begin{pf}
Let $n=\lceil N/10\rceil$ and let $b=b(n)$. By the central limit
theorem, there is a $c_{1}>0$ such that $\mathbf{P}_{\mathbf{x}}
[\tau_{\mathcal{I}_{N}(b)}<N^{2}/2]>c_{1}$
uniformly in $\mathbf{x}$. For $\mathbf{y}=(y_{1},\ldots,y_{d})\in
\mathbb{Z}^{d}$,
define $\mathcal{B}(\mathbf{y})$ to be the event that in $N^{2}/2$
steps the first coordinate of a $d$-dimensional random walk hits
$y_{1}+4n$ and then hits $y_{1}-8n$, while the maximal change in
the other coordinates is less than $n$. By the invariance principle,
there is a $c_{2}>0$ such that for all large $N$, $\mathbf{P}_{\mathbf
{y}}[\mathcal{B}(\mathbf{y})]>c_{2}$.

Let $\tau_{i}^{b}=\inf\{ t\ge iN^{2}\dvtx S(t)\in\mathcal
{I}_{N}(b)\} $,
let $\mathcal{A}_{i}=\{ \tau_{i}^{b}<(i+\frac{1}{2})N^{2}\} $
and let $\chi_{i}$ be the indicators of $\mathcal{A}_{i}$ occurring
for $S(t)$ and $\mathcal{B}(S(\tau_{i}^{b}))$ occurring for $S(\tau
_{i}^{b}+t)$.
Note that $\chi_{i}$ implies there is a $(\gamma,\gamma^{+},\beta
)\in
\mathfrak{T}_{N}$,
$iN^{2}\le\gamma<\gamma^{+}<(i+1)N^{2}$.

By the Markov property, $\chi_{i}$ dominates i.i.d. Bernoulli r.v.'s
that are $1$ w.p. $c_{3}>0$ for all large $N$. Thus, by the law
of large numbers, $\sum_{i=1}^{\lfloor uN^{d-2}\rfloor}\chi
_{i}>c_{3}uN^{d-2}/2$,
w.h.p. This event implies that $\tau_\rho(b)\le uN^{d}$
for $\rho<c_{3}u/2$, which completes the proof.
\end{pf}
Given a box $b$, and a set $\omega\subset b$, we call $\omega$
\emph{$b$-boundary-connected} if any $\mathbf{x}\in\omega\cap b^7$
is connected in $\omega$ to $\partial^{\mathrm{in}}b^7$. We call $F\dvtx
\mathbb
{Z}^{\ge0}\to2^{b^7}$
a \emph{$b$-boundary-connected-path} if $F(t)\subset F(t+1)$ and
$F(t)$ is $b$-boundary-connected for all $t\ge0$.

We write $\mathbf{S}$ for the set of all finite paths in $\mathbb{Z}^{d}$.
That is,
\[
\mathbf{S}=\bigl\{ s=(\mathbf{v}_{0},\ldots,\mathbf{v}_{n})
\dvtx\mathbf {P}_{\mathbf{v}_{0}}\bigl[S(0,n)=s\bigr]>0\bigr\}.
\]
For $s=(\mathbf{v}_{0},\ldots,\mathbf{v}_{n})\in\mathbf{S,}$ we let
$s_i=\mathbf{v}_{i}$ and write $\|s\|$ for $n$, the number of
edges traversed by the path $s$.
The next lemma attains stochastic domination between the range of the
random walk and a $\rho$-dense $b$-itinerary.
%
%le5.3 #&#
\begin{lem}
\label{lemPNgePrhoandFt}For $N>0$, fix a box $b=b(\lceil
N/10\rceil)$,
$\rho>0$ and $\mathbf{x}\in\mathbb{Z}^{d}$. Then for any $\mathcal
{A}\subset2^{b^7}$
there is a $\rho$-dense $b$-itinerary $H=H(\mathcal{A})$ and a
$b$-boundary-connected-path $F(t)=F(\mathbf{x},t)$ such that
\[
\mathbf{P}_{\mathbf{x}}\bigl[\bigl\{ \mathcal{I}_{N}^{*}
\circ\mathcal {R}(t)\dvtx t\ge\tau_{\rho}(b)\bigr\} \subset\mathcal{A}
\bigr]\ge\mathbb {P}_{H}\bigl[\bigl\{ \mathcal{R}\cup F(t)\dvtx t\ge0
\bigr\} \subset\mathcal {A}\bigr].
\]
\end{lem}
\begin{pf}
%
%$H(\mathcal{A})$ independent. from $\mathbf{x}$ since $H(\mathcal{A})$
%is some $H(\mathcal{B})$ for which $\mathbb{P}_{H(\mathcal{B})}
%[\ldots]$
%is {}``small enough'' and this only depends on $\mathcal{A}$. After
%finding such a $\mathcal{B}$, $F(t)$ depends on $\mathbf{x}$.
%%
%{}
Let $n=\lceil N/10\rceil$ and let $M=\lceil\rho n^{d-2}\rceil$.
For $1\le i\le M+1$ fix $s_{i}\in\mathbf{S}$. Let
\[
\tau_{\square}=\inf\bigl\{ t\ge\tau_{\rho}\dvtx\mathcal {T}(N)=
\mathcal {R}_{N}(\tau_{\rho},t)\bigr\}
\]
that is, the first time after $\tau_\rho$, the random walk (starting at
time $\tau_\rho$) covers the torus.
Since $\mathcal{R}_{N}$ takes values in the finite state space
$\mathcal{T}(N)$,
$\mathbf{P}_{\mathbf{x}}[\tau_{\square}<\infty]=1$. With
the convention that $\gamma_{0}^{+}=0$, we partition the probability
space of $S(\cdot)$ to events
\[
\mathcal{B}=\mathcal{B}(s_{1},\ldots,s_{M+1})=\Biggl\{
\bigcap_{i=1}^{M}S\bigl(
\gamma_{i-1}^{+},\gamma_{i}\bigr)=s_{i}
\Biggr\} \cap S\bigl(\gamma _{M}^{+},\tau_{\square}
\bigr)=s_{M+1}
\]
satisfying $\mathbf{P}_{\mathbf{x}}[\mathcal{B}(s_{1},\ldots,s_{M+1})]>0$.
For $i=1,\ldots,M$ let $\alpha(i)=s_{i}(\|s_{i}\|),\zeta
(i)=s_{i+1}(0)$, that is, the end point of the path $s_i$ and the
starting point of the path $s_{i+1}$.
By the Markov property (see Proposition~\ref{proproductformindependent}),
$\{ S(\gamma_{i},\gamma_{i}^{+})\} _{i=1}^{M}$ under \textbf
{$\mathbf{P}_{\mathbf{x}}[\cdot|\mathcal{B}]$}
are independent random vectors with the distribution of $S(0,\tau
_{\partial\beta_{i}^{7}})$
under $\mathbf{P}_{\alpha(i)}[\cdot|\tau_{\partial\beta
_{i}^{7}}=\tau_{\zeta(i)}]$.
Let $\mathbf{a}(i)=\mathcal{I}_{N}^{*}(\alpha(i))$, $\mathbf
{z}(i)=\mathcal{I}_{N}^{*}(\zeta(i))$
and let $H$ be a $b$-itinerary, $H=(\eta_{1},\ldots,\eta_{M})$
where $\eta_{i}=(\mathbf{a}(i),\mathbf{z}(i))$. Since
$\mathcal{I}_{N}^{*}$ is an isomorphism between $\beta_{i}^{+}$
and $b^7$, $\mathcal{I}_{N}^{*}\circ S(0,\tau_{\partial\beta_{i}^{7}})$
under $\mathbf{P}_{\alpha(i)}[\cdot|\tau_{\partial\beta
_{i}^{7}}=\tau_{\zeta(i)}]$
is distributed the same as $S(0,\tau_{\partial b^7})$ under $\mathbf
{P}^{\eta_{i}}[\cdot]$.
Thus, $\bigcup_{i=1}^{M}\{ \mathcal{I}_{N}^{*}\circ\mathcal
{R}(\gamma_{i},\gamma_{i}^{+})\} $
under $\mathbf{P}_{\mathbf{x}}[\cdot|\mathcal{B}]$
is distributed like $\mathcal{R}_{H}$ under $\mathbb{P}_{H}[\cdot
]$.
Let
\[
\hat{F}(t)=\bigcup_{i=1}^{M}\mathcal{R}
\bigl(\gamma_{i-1}^{+},\gamma _{i}\bigr)\cup
\mathcal{R}\bigl(\gamma_{M}^{+},\bigl(\gamma_{M}^{+}+t
\bigr)\wedge\tau _{\square}\bigr).
\]
Since $\tau_{\rho}=\gamma_{M}^{+}$, we have $\mathcal{R}(\tau
_{\rho
}+t)=\hat{F}(t)\cup\bigcup_{i=1}^{M}\{ \mathcal{R}(\gamma
_{i},\gamma_{i}^{+})\} $
for all $t\ge0$. Given~$\mathcal{B}$, $\hat{F}(t)$ is uniquely
determined. Let $F(t)=\mathcal{I}_{N}^{*}\circ\hat{F}(t)$. Since
$\mathcal{I}_{N}^{*}$ is either a local isomorphism to $b^7$ or
else gives the empty set, $F(t)$ is a $b$-boundary-connected-path.
Thus, for any $\mathcal{A}\subset2^{b^7}$
\[
\mathbf{P}_{\mathbf{x}}\bigl[\bigl\{ \mathcal{I}_{N}^{*}
\circ\mathcal {R}(t)\dvtx t\ge\tau_{\rho}(b)\bigr\} \subset\mathcal{A}|
\mathcal{B} \bigr]=\mathbb{P}_{H(\mathcal{B})}\bigl[\bigl\{ \mathcal{R}\cup
F_{(\mathcal
{B})}(t)\dvtx t\ge0\bigr\} \subset\mathcal{A}\bigr],
\]
which proves the lemma (see Proposition~\ref{propartition}).
\end{pf}
For a box $B$ and a $B$-itinerary $H$, we proceed to define the
event $\mathcal{D}_{\rho}^{\sigma}=\mathcal{D}_{\rho}^{\sigma}(H,B)$.
Roughly, $\mathcal{D}_{\rho}^{\sigma}$ is the event that all subboxes
are crossed a correct order of times by $B$-traversals. First, given
a box $B$ and $b\in\sigma(B)$, let us define for a random walk
$S(\cdot)$
the event $\mathcal{J}_{B}[b]$
\begin{eqnarray*}
\mathcal{J}_{B}[b] & =&\bigl\{  \exists t,t^{+}
\dvtx0<t<t^{+}<\tau _{\partial B^7}\dvtx S(t-1)\in\partial
b^7,S(t)\in\operatorname{Top}(b),
\\
& &\hspace*{105pt} \mathcal{R}\bigl(t,t^{+}-1\bigr)\subset b^7,S
\bigl(t^{+}\bigr)\in\operatorname{Bot}(b)\bigr\}.
\end{eqnarray*}
Given $H$ a $B$-itinerary, $\eta\in H$, and a subbox $b\in\sigma(B)$,
we write $\mathcal{J}_{\eta}[b]$ for the event $\mathcal
{J}_{B}[b]$
occurring on the random walk $S_{\eta}(\cdot)$.

Next, we would like to assign each box $b\in\sigma(B)$ a subset $H
[b]\subset H$
with the property that if two distinct subboxes intersect, they have
disjoint $H[\cdot]$ sets. Let us do this by first fixing
a function $(\cdot)_{50}\dvtx\mathbb{Z}^{d}\to\{ 0,1,\ldots,
50^{d}-1\} $
with the property that any distinct $\mathbf{x},\mathbf{y}\in\mathbb{Z}^{d}$
with $(\mathbf{x})_{50}=(\mathbf{y})_{50}$ are a distance of at least
$50$ in the $l_{\infty}$ norm. This can be induced by any bijection
from $(\mathbb{Z}/50\mathbb{Z})^{d}$ to $\{ 0,1,\ldots,
50^{d}-1\} $.

Recall that $\Delta$ is the isomorphism mapping $\sigma(B)$ into
$\mathbb{Z}^{d}$, and that $H=(\eta_{1},\ldots,\eta_{k})$
is an ordered sequence. We write
\[
H[b]=\bigl\{ \eta_{i}\in H\dvtx i\equiv(\Delta b )_{50}
\bigl(\operatorname{mod}50^{d}\bigr)\bigr\}.
\]
Next, for each $b\in\sigma(B)$ define the random set of $B$-traversals
$\psi_{H}[b]=\{ \eta\in H[b]\dvtx\mathcal{J}_{\eta}
[b]\} $.
Since $2\|b^7\|<50^d\|b\|$, we get the following desired property.
%
%cl5.4 #&#
\begin{clm}
\label{cladisjointtraversals}For any distinct $b_{0},b_{1}\in\sigma(B)$
satisfying $b_{0}^{7}\cap b_{1}^{7}\neq\varnothing$ we have $\psi_{H}
[b_{0}]\cap\psi_{H}[b_{1}]=\varnothing$.
\end{clm}
%
%de5.5 #&#
\begin{defn}
\label{defDrhosigma}Let $\mathcal{D}_{\rho}^{\sigma}$\label{pg221} be the
event that for each $b\in\sigma(B)$, $|\psi_{H}[b]
|\ge\rho\|b\|^{d-2}$.
\end{defn}
The next lemma identifies, given some set, an itinerary which minimizes
the probability to be contained in the set. The sets in mind are
non good sets.
%
%le5.6 #&#
\begin{lem}
\label{lemrhotraversedsubboxgerhodenseitinerary}Fix a box
$B$, a $B$-itinerary $H$, a $B$-boundary-connected-path $F(t)$,
and a subbox $b\in\sigma(B)$.\vspace*{-1pt} Let $\mathcal{H}=\mathcal{H}(\mathcal
{E})=\mathcal{D}_{\rho}^{\sigma},\mathcal{R}_{H}\setminus b^7\in
\mathcal{E}$
where $\mathcal{E}$ is a fixed subset of $2^{B^7}$. Assume $\mathbb
{P}_{H}[\mathcal{H}]>0$.
Then for any $\mathcal{A}\subset2^{b^7}$, there is
a $\rho$-dense $b$-itinerary $h=h(\mathcal{A})$ and a
$b$-boundary-connected-path
$f(t)=f(\mathcal{A},F)(t)\subset b^7$ satisfying
\[
\mathbb{P}_{H}\bigl[\bigl\{ \bigl(\mathcal{R}\cup F(t)\bigr)\cap
b^7\dvtx t\ge0 \bigr\} \subset\mathcal{A} | \mathcal{H}\bigr]\ge
\mathbb{P}_{h}\bigl[ \bigl\{ \mathcal{R}\cup f(t)\dvtx t\ge0\bigr\}
\subset\mathcal{A}\bigr].
\]
\end{lem}
\begin{pf}
If for $\eta\in H$ the event $\mathcal{J}_{\eta}[b]$
occurs, then we know there exists at least one time pair $(t,t^{+})$,
$0<t<t^{+}<\tau_{\partial B^7}$ satisfying the requirements of
$\mathcal
{J}_{\eta}[b]$---roughly that $b^7$ is crossed top to bottom by $S_{\eta}$. Since
these time pairs must be disjoint, we can consider the first, which
we shall denote by $(t_{\eta},t_{\eta}^{+})$.

Fix $Q\subset H[b]$, $s_{\eta}^{1},s_{\eta}^{2}\in\mathbf{S}$
for each $\eta\in Q$ and $s_{\eta}^{0}\in\mathbf{S}$ for each $\eta
\in
H\setminus Q$
and define the event
%
%e21 #&#
\begin{eqnarray*}
\mathcal{B}=\mathcal{B}\bigl(Q,s_{\eta}^{i}\bigr)&=&
\bigl\{ \psi_{H}[b ]=Q\bigr\} \cap\bigcap
_{\eta\in Q}\bigl\{ S_{\eta}(0,t_{\eta})=s_{\eta
}^{1},S_{\eta}
\bigl(t_{\eta}^{+},\tau_{\partial B^7}\bigr)=s_{\eta}^{2}
\bigr\}
\\
&&{} \cap\bigcap_{\eta\in H\setminus Q}\bigl\{ S_{\eta}(0,
\tau_{\partial
B^7})=s_{\eta}^{0}\bigr\}.
\end{eqnarray*}
We partition $\{ S_{\eta}(\cdot)\dvtx\eta\in H\} $
to such $\mathcal{B}(Q,s_{\eta}^{i})$ events satisfying $\mathbb
{P}_{H}[\mathcal{B}(Q,s_{\eta}^{i})]>0$.
Any two distinct $\mathcal{B}(Q,s_{\eta}^{i}),\mathcal{B}(\hat
{Q},\hat
{s}_{\eta}^{i})$
have an empty intersection because either $Q\neq\hat{Q}$ or if
$Q=\hat{Q}$
then $s_{\eta}^{i}\neq\hat{s}_{\eta}^{i}$ for some $\eta\in Q$.
Observe that $\mathcal{R}_{H}\setminus b^7$ is determined by $\mathcal{B}$,
and that by our construction (Claim~\ref{cladisjointtraversals}),
so is $\psi_{H}[\cdot]$. Since $\mathcal{H}$ is $(\mathcal
{R}_{H}\setminus b^7,\psi_{H}[\cdot])$
measurable, and the $\mathcal{B}$ events are a partition of the entire
probability space, those for which $\mathbb{P}_{H}[\mathcal
{B},\mathcal{H}]>0$
form a partition of $\mathcal{H}$. $\mathcal{H}\subset\mathcal
{D}_{\rho
}^{\sigma}$
so any positive probability $\mathcal{B}(Q,s_{\eta}^{i})\subset
\mathcal{H}$
has $|Q|\ge\rho\|b\|^{d-2}$. For each $\eta\in Q$ let
$\mathbf{a}(\eta)=s_{\eta}^{1}(\|s_{\eta}^{1}\|),\mathbf{z}(\eta
)=s_{\eta}^{2}(0)$
and let $h$ be a $b$-itinerary, $h=(\mathbf{a}(\eta),\mathbf
{z}(\eta))_{\eta\in Q}$
with order inherited from $H$. Since $S_{\eta}$ are independent
and $\mathcal{B}$ is a product of events on $\{S_{\eta}\}_{\eta\in H}$
($\{ \psi_{H}[b]=Q\} $ can be factored to each
$\{S_{\eta}\}_{\eta\in H}$), we have by the Markov property (see
Proposition~\ref{proproductformindependent}), that $\{ S_{\eta
}(t_{\eta},t_{\eta}^{+})\} _{\eta\in Q}$
under $\mathbb{P}_{H}[\cdot|\mathcal{B}]$ are
independent random vectors with the distribution of $S(0,\tau
_{\partial b^7})$
under $\mathbf{P}_{\mathbf{a}(\eta)}[\cdot|\tau_{\partial
b^7}=\tau
_{\mathbf{z}(\eta)}]$.
Thus $\bigcup_{\eta\in Q}\mathcal{R}_{\eta}(t_{\eta},t_{\eta}^{+})$
under {$\mathbb{P}_{H}[\cdot|\mathcal{B}]$} is
distributed like $\mathcal{R}_{h}$ under $\mathbb{P}_{h}[\cdot
]$.
Let $\hat{f}=\bigcup_{\eta\in H\setminus Q}s_{\eta}^{0}\cup
\bigcup_{\eta\in Q}s_{\eta}^{1}\cup\bigcup_{\eta\in
Q}s_{\eta}^{2}$
and let $f(t)=(\hat{f}\cup F(t))\cap b^7$. Since all elements in
the union are $B$-boundary-connected, $f(t)$ is a $b^7$-boundary-connected-path.
As $(\mathcal{R}_{H}\cup F(t))\cap b^7=f(t)\cup\bigcup_{\eta\in
Q}\{ \mathcal{R}_{\eta}(t_{\eta},t_{\eta}^{+})\} $,
we have for any $\mathcal{A}\subset\Pw(b^7)$
\[
\mathbb{P}_{H}\bigl[\bigl\{ \bigl(\mathcal{R}\cup F(t)\bigr)\cap
b^7\dvtx t\ge0 \bigr\} \subset\mathcal{A}|\mathcal{B}\bigr]=
\mathbb{P}_{h}\bigl[\bigl\{ \mathcal{R}\cup f(t)\dvtx t\ge0\bigr\}
\subset\mathcal{A}\bigr].
\]
Since $|h|=|Q|\ge\rho\|b\|^{d-2}$, $h$ is
$\rho$-dense, and as $\mathcal{B}$ is an arbitrary partition element
of $\mathcal{H}$, this proves the lemma by Proposition~\ref{propartition}.
\end{pf}

%s5.3 #&#
\subsection{Properties of the range of a random walk}

We will require the following large deviation estimate for sums of
independent indicators, a weak version of Lemma~4.3 from \cite
{bramson1991asymptotic}.
%
%le5.7 #&#
\begin{lem}
\label{lemconcentrationofindicators}Let $Q$ be a finite sum of
independent indicator ($\{ 0,1\} $-valued) random variables
with mean $\mu>0$. There is a $0<c_{f}<1$ such that
\[
\Pr[Q<\mu/2],\Pr[Q>2\mu]<\exp(-c_{f}\mu).
\]
\end{lem}
Recall $\mathcal{D}_{\rho}^{\sigma}$ from Definition~\ref{defDrhosigma}.
%
%th5.8 #&#
\begin{thmm}
\label{thmmdenseimpliessubboxdensewhp}There is a $\Dns(d)>0$
such that for any $q>0,\varrho>0$ there is a $C(q,\varrho)<\infty$
such that if $n>C$ and $H$ is a $\varrho$-dense $B(n)$-itinerary,
\[
\mathbb{P}_{H}\bigl[\mathcal{D}_{\Dns\varrho}^{\sigma}
\bigr]>1-q.
\]
\end{thmm}
\begin{pf}
Fix $b\in\sigma(B)$ and let $m=\|b\|=s(n)$. Lemma~\ref
{lemlowbndhitsmallbox}
tells us that for any $B$-traversal $\eta\in H$,
\[
\mathbf{P}^{\eta}\bigl[\mathcal{J}_{\eta}[b]
\bigr]>c_{\Dns
}\biggl(\frac{m}{n}\biggr)^{d-2}.
\]
Let $Q=\sum_{\eta\in H[b]}\ind_{\CJ_\eta[b]}$. For
all large enough $n$, $|H[b]|\ge60^{-d}\varrho n^{d-2}$,
so by linearity,
\[
\mathbb{E}_{H}\bigl[\bigl|\psi_{H}[b]\bigr|\bigr]=\mathbb
{E}_{H}[Q]>\varrho60^{-d}c_{\Dns}m^{d-2}.
\]
Since the random walks $S_{\eta}$ are mutually independent, by
Lemma~\ref{lemconcentrationofindicators} there is a $c_{f}$ such that
\[
\mathbb{P}_{H}\bigl[Q\le\mathbb{E}[Q]/2\bigr]\le\exp
\bigl(-c_{f}\mathbb{E}[Q]\bigr)\le e^{-cm^{d-2}}.
\]
Let $\Dns=60^{-d}c_{\Dns}/2$. For $m\ge[\frac{2d}{c_{\Dns
}c_{f}60^{-d}\varrho}\log n]^{{1}/{(d-2)}}$
we have
%
%e22 #&#
\begin{equation}
\label{eqqpoly} \mathbb{P}_{H}\bigl[Q\le\varrho\Dns m^{d-2}
\bigr]<n^{-4(d-2)}.
\end{equation}
If $b=b(\mathbf{x},m)\in\sigma(B)$ then $\mathbf{x}\in B^{6}$ and
$m=s(n)=\lceil\log^{4}n\rceil$ which is\break  $\omega(\log^{{1}/{(d-2)}}n)$
but $o(n)$. Thus, for all large $n$, a union bound on $\sigma(B)$
gives the result.
\end{pf}
%
%re5.9 #&#
\begin{rem}\label{remanypoly}
One can obtain any polynomial decay in \eqref{eqqpoly} by taking
$m=\log^kn$, for large enough $k$.
\end{rem}

%%%%%%%%%%%%%%%%%%%%%%%%%%%%%%%%%%%%%%%%%%%%%%%%
The below lemma shows that w.h.p., the union of those ranges of a dense
itinerary which intersect an interior set of low density, has size
of greater order than the size of the set itself.
%
%le5.10 #&#
\begin{lem}
\label{lemsetofsizenalphahit}Let $b=b(n)$, and let $M\subset b^{6}$
where $|M|\ge\beta n^{\alpha}$ with $\alpha,\beta>0$.
Let $h$ be a $\varrho$-dense $b$-itinerary, $\varrho>0$. Then
for $\gamma=1+2(\alpha^{-1}-d^{-1})$, $c_{g}(d)>0$ and all large
$n$,
\[
\mathbb{P}_{h}\biggl[\bigl|\mathcal{R}\bigl(\bigl\{ \eta\in h\dvtx
\mathcal {R}(\eta)\cap M\neq\varnothing\bigr\} \bigr)\cap b\bigr|<c_{g}
\varrho \beta n^{\gamma\alpha}\wedge\frac{n^{d}}{2}\biggr]\le\exp
\bigl(-c_{g}\varrho \beta n^{\alpha(1-2/d)}\bigr).
\]
\end{lem}
\begin{pf}
Fix $\eta=(\mathbf{a},\mathbf{z})\in h$ and $Q\subset b$. Let
$\mathcal
{B}(\eta,Q)=\{ |Q|<n^{d}/2\} \cup\{
|\mathcal{R}(\eta)\cap Q|>c_{0}n^{2}\} $,
$c_{0}>0$ determined below. Let $\tau_{M}(\eta)$ be the first hitting
time of $M$ by $S_{\eta}$, and let $\tau_{\mathcal{B}}(\eta)=\inf
\{ t\ge0\dvtx S_{\eta}(0,t)\subset\mathcal{B}\} $
be the first time the occurrence of $\mathcal{B}$ is implied by
$S_{\eta}(t)$.
By Proposition~\ref{proHitthenstop} for some $\mathbf{x}\in M$
\begin{eqnarray*}\hspace*{-8pt}
\mathbf{P}_{\mathbf{a}}[\tau_{M},\tau_{\mathcal{B}}<\tau
_{\partial
B^7}|\tau_{\partial B^7}=\tau_{\mathbf{z}}] & \ge& \mathbf
{P}_{\mathbf{a}}[\tau_{M}<\tau_{\partial B^7}|
\tau_{\partial
B^7}=\tau_{\mathbf{z}}]\mathbf{P}_{\mathbf{x}}[
\tau_{\mathcal
{B}}<\tau_{\partial B^7}|\tau_{\partial B^7}=
\tau_{\mathbf{z}}].
\end{eqnarray*}
By Lemma~\ref{lemlowbndhitmdmassinternal} and Corollary~\ref
{corlowbndPaxzorxyz},
\[
\mathbf{P}_{\mathbf{a}}[\tau_{M}<\tau_{\partial B^7}|\tau
_{\partial B^7}=\tau_{\mathbf{z}}]>c|M |^{1-2/d}n^{2-d}\ge
c_{1}\beta n^{\alpha(1-2/d)+(2-d)}.
\]
Since $M\subset b^{6}$ and assuming $|Q|\ge n^{d}/2$
for the nontrivial case, again by Lemma~\ref{lemlowbndhitmdmassinternal}
and Corollary~\ref{corlowbndPaxzorxyz}, we get for some
$c_{0},c_{2}>0$,
\[
\mathbf{P}_{\mathbf{x}}\bigl[\mathcal{R}(0,\tau_{\partial B^7})\cap
Q>c_{0}n^{2}|\tau_{\partial B^7}=\tau_{\mathbf{z}}
\bigr]>c_{2}.
\]
Recall, $h=(\eta_{1},\eta_{2},\ldots)$. Let $\mathcal
{R}_{k}^{M}=\mathcal{R}(\{ \eta_{i}\in h\dvtx1\le i\le k,\mathcal
{R}(\eta_{i})\cap M\neq\varnothing\} )$,
let $\chi(\eta,Q)$ be the indicator variable for the event $\{ \tau
_{M}(\eta),\tau_{\mathcal{B}}(\eta)<\tau_{\partial B^7}(\eta)\} $
and let $\mathcal{S}_{k}=\sum_{i=1}^{k}\chi(\eta_{i},b\setminus
\mathcal{R}_{i-1}^{M})$.
Then $|\mathcal{R}_{k}^{M}|$ stochastically dominates
$c_{0}n^{2}\mathcal{S}_{k}\wedge\frac{n^{d}}{2}$. By above bounds,
and independence of traversals, the sequence $\chi(\eta
_{i},b\setminus
\mathcal{R}_{i-1}^{M})$
dominates i.i.d. Bernoulli r.v.'s that are $1$ w.p. $c_{1}c_{2}\beta
n^{\alpha(1-2/d)+(2-d)}$.
Thus, by concentration of i.i.d. Bernoulli r.v.'s, for example, as
stated in Lemma~\ref{lemconcentrationofindicators},
\[
\mathbb{P}_{h}\biggl[\bigl|\mathcal{R}_{|h|}^{M}\bigr|<
\frac
{c_{0}}{2}n^{2}\mathbb{E}_{h}[\mathcal{S}_{|h|}
]\wedge\frac{n^{d}}{2}\biggr]<\exp\bigl(-c_{f}
\mathbb{E}_{h}[\mathcal {S}_{|h|}]\bigr).
\]
%
%This is true since by domination, there's a coupling under which $
%}{2}\wedge\frac{n^{d}}{2}\} $
%implies $\{ c_{0}n^{2}\mathcal{S}_{|\eta|}\wedge\frac
%{n^{d}}{2}<c_{0}n^{2}\frac{\mu}{2}\wedge\frac{n^{d}}{2}\} $
%which implies $\{ \mathcal{S}_{|\eta|}<\frac{\mu
%}{2}\} $.
%%
%{}
Since
\[
\mathbb{E}_{h}[\mathcal{S}_{|h|}]\ge\varrho
n^{d-2}c_{3}\beta n^{\alpha(1-2/d)+(2-d)},
\]
we get
\[
\mathbb{P}_{h}\biggl[\bigl|\mathcal{R}_{|h|}^{M}
\bigr|<c_{4}\varrho\beta n^{\alpha(1+2(\alpha^{-1}-d^{-1}))}\wedge\frac
{n^{d}}{2}\biggr]<
\exp\bigl(-c_{f}c_{3}\varrho\beta n^{\alpha(1-2/d)}\bigr),
\]
which proves the lemma.
\end{pf}

%%%%%%%%%%%%%%%%%%%%%%%%%%%%%%%%%%%%%%%%%%%
%
%le5.11 #&#
\begin{lem}
\label{lemallconnected}Let $b(n)$ be a box, let $F(t)$ be a
$b$-boundary-connected-path,
and let $H$ be a $\rho$-dense $b$-itinerary, $\rho>0$. There is
a $c(\rho)>0$ such that for all large $n$
\[
\mathbb{P}_{H}\bigl[\forall t\ge0,\bigl(\mathcal{R}\cup F(t)\bigr)
\cap b^5\mbox{ is connected in }\mathcal{R}\cup F(t)\bigr]>1-\exp
\bigl(-cn^{1-2/d}\bigr).
\]
\end{lem}
\begin{pf}
For any $h\subset H$, $F(t)\cup\mathcal{R}_{h}$ is also a
$b$-boundary-connected-path
and is independent from the traversals in $H\setminus h$. Thus, we
may assume w.l.o.g. that $|H|=\lceil\rho n^{d-2}\rceil$. Let
$H^{5.5}=\{ \eta\in H\dvtx\mathcal{R}(\eta)\cap b^{5.5}\neq
\varnothing
\} $.
We show that $\mathcal{R}(H^{5.5})$ is connected in $\mathcal
{R}=\mathcal{R}_{H}$
w.h.p. Set $D=\lceil\log_{(1+d^{-1})}\frac{2d}{3}\rceil$. If $
|H^{5.5}|\le1$
we are done. Otherwise given distinct traversals $\zeta,\varphi\in H^{5.5}$,
partition $H\setminus\{ \zeta,\varphi\} $ into sets
$H_{1}^{\zeta},H_{1}^{\varphi},\ldots,H_{D}^{\zeta},H_{D}^{\varphi}$
and $H_{*}$, where each of the $2D+1$ sets has size at least $
|H|/3D>c_{D}(\rho,d)n^{d-2}$.
Set $M_{0}^{\zeta}=\mathcal{R}(\zeta)\cap b^{6}$ and for $i=1,\ldots,D$
recursively define
\[
M_{i}^{\zeta}=\bigcup_{\eta\in H_{i}^{\zeta}\dvtx M_{i-1}^{\zeta
}\cap\mathcal
{R}(\eta)\neq\varnothing}
\mathcal{R}(\eta)\cap b^{6}.
\]
Define $M_{i}^{\varphi}$ analogously. Thus, the event $\mathcal{M}=
\{ \exists\eta\in H_{*}\dvtx|\mathcal{R}(\eta)\cap M_{D}^{\zeta}
|,|\mathcal{R}(\eta)\cap M_{D}^{\varphi}|>0\} $
implies that $\mathcal{R}(\zeta)$ is connected to $\mathcal
{R}(\varphi)$
in $\mathcal{R}$, an event we denote by $\zeta\leftrightarrow\varphi$.
For $i=0,\ldots,D$, let $\mathcal{M}_{i}^{\zeta}=\{
|M_{i}^{\zeta}|\ge c_{\rho}^{i}n^{(1+d^{-1})^{i}}\} $
where $c_{\rho}=c_{g}\rho\wedge\frac{1}{2}$, with $c_{g}$ from
Lemma~\ref{lemsetofsizenalphahit}. Define $\mathcal{M}_{i}^{\varphi}$
analogously. By independence of $\mathcal{M}_{D}^{\zeta}$ and~$\mathcal
{M}_{D}^{\varphi}$,
%
%e23 #&#
\begin{equation}
\mathbb{P}_{H}[\zeta\leftrightarrow\varphi]\ge\mathbb
{P}_{H}[\mathcal{M}]\ge\mathbb{P}_{H}\bigl[\mathcal{M} |
\mathcal {M}_{D}^{\zeta},\mathcal{M}_{D}^{\varphi}
\bigr]\mathbb {P}_{H}\bigl[\mathcal{M}_{D}^{\zeta}
\bigr]\mathbb{P}_{H}\bigl[\mathcal {M}_{D}^{\varphi}
\bigr].\label{eqlowbnd2trvsintersect}
\end{equation}
By Proposition~\ref{propartition}, for some $F_{D}^{\zeta
}(H_{*}),F_{D}^{\varphi}(H_{*})\subset b^{6}$
with $|F_{D}^{\zeta}|,|F_{D}^{\varphi}|\ge c_{\rho
}^{D}n^{2d/3}$,
%
%e24 #&#
\begin{equation}
\mathbb{P}_{H}\bigl[\mathcal{M} | \mathcal{M}_{D}^{\zeta},
\mathcal {M}_{D}^{\varphi}\bigr]\ge\mathbb{P}_{H}
\bigl[\mathcal{M} | M_{D}^{\zeta
}=F_{D}^{\zeta},M_{D}^{\varphi}=F_{D}^{\varphi}
\bigr]=1-q_{D}.\label{eq0goodintersectMqD}
\end{equation}
Given $\eta=(\mathbf{a},\mathbf{z})\in H_{*}$, let $\tau
_{\zeta}(\eta),\tau_{\varphi}(\eta)$
be the first hitting times of $F_{D}^{\zeta},F_{D}^{\varphi}$ by
$\eta$, respectively. By Proposition~\ref{proHitthenstop}, for
some $\mathbf{x}\in F_{D}^{\zeta}$,
\begin{eqnarray*}
\mathbf{P}_{\mathbf{a}}[\tau_{\zeta},\tau_{\varphi}<\tau
_{\partial
B^7}|\tau_{\partial B^7}=\tau_{\mathbf{z}}] & \ge& \mathbf
{P}_{\mathbf{a}}[\tau_{\zeta}<\tau_{\partial B^7}|
\tau_{\partial
B^7}=\tau_{\mathbf{z}}]\mathbf{P}_{\mathbf{x}}[
\tau_{\varphi
}<\tau_{\partial B^7}|\tau_{\partial B^7}=
\tau_{\mathbf{z}}].
\end{eqnarray*}
Since $F_{D}^{\zeta},F_{D}^{\varphi}\subset b^{6}$ and $
|F_{D}^{\zeta}|,|F_{D}^{\varphi}|\ge c_{\rho}^{D}n^{2d/3}$,
by Lemma~\ref{lemlowbndhitmdmassinternal} and Corollary~\ref
{corlowbndPaxzorxyz}
each term in the product above is at least
\[
c_{1}\bigl(c_{\rho}^{D}n^{2d/3}
\bigr)^{1-2/d}cn^{2-d}=c_{2}n^{(2-d)/3}.
\]
Let $\chi(\eta)$ be the indicator for the event $\{ \tau_{\zeta
}(\eta),\tau_{\varphi}(\eta)<\tau_{\partial B^7}(\eta)\} $,
and write $\mathcal{S}=\sum_{\eta\in H_{*}}\chi(\eta)$. Then
$q_{D}\le\mathbb{P}_{H}[\mathcal{S}=0]$. By concentration
of independent indicators in Lemma~\ref{lemconcentrationofindicators},
we get
%
%e25 #&#
\begin{equation}
q_{D}\le\exp\bigl(-c_{f}c_{D}n^{d-2}c_{2}^{2}n^{2(2-d)/3}
\bigr)\le\exp \bigl(-c_{3}(\rho )n^{(d-2)/3}\bigr).\label{eq0goodintersectqDbnd}
\end{equation}
We now lower bound $\mathbb{P}_{H}[\mathcal{M}_{D}^{\zeta}]$
and $\mathbb{P}_{H}[\mathcal{M}_{D}^{\varphi}]$ from \eqref
{eqlowbnd2trvsintersect}.
Since the bound is the same for both terms, we drop $\zeta,\varphi$
from the notation. Note that by connectedness of each traversal in
$H^{5.5}$, $\mathcal{M}_{0}$ is of probability one, thus by chaining
conditions
%
%e26 #&#
\begin{equation}
\mathbb{P}_{H}[\mathcal{M}_{D}]\ge\prod
_{i=0}^{D-1}\mathbb {P}_{H}[
\mathcal{M}_{i+1} | \mathcal{M}_{i},\ldots,\mathcal
{M}_{0}].\label{eq0goodintersectMproduct}
\end{equation}
By Proposition~\ref{propartition}, for some $F_{i}(H_{i+1})\subset b^{6}$
with $|F_{i}|\ge c_{\rho}^{i}n^{(1+d^{-1})^{i}}$ we have
\[
\mathbb{P}_{H}[\mathcal{M}_{i+1} | \mathcal{M}_{i},
\ldots,\mathcal{M}_{0}]\ge\mathbb{P}_{H}[
\mathcal{M}_{i+1} | M_{i}=F_{i}].
\]
By Lemma~\ref{lemsetofsizenalphahit}, for $0\le i\le D-1$,
if $F_{i}\subset b^{6}$ and $|F_{i}|\ge c_{\rho
}^{i}n^{(1+d^{-1})^{i}}$,
then for all large $n$ and some $c_{i}(\rho)>0$,
%
%e27 #&#
\begin{equation}
q_{i}(F_{i})=1-\mathbb{P}_{H}[
\mathcal{M}_{i+1} | M_{i}=F_{i}]\le \exp
\bigl(-c_{i}n^{1-2/d}\bigr).\label{eq0goodintersectqi}
\end{equation}
Let $q_{i}=1-p_{i}$ for $i=0,\ldots,D$. Writing $\zeta
\nleftrightarrow
\varphi$
for the event that $\mathcal{R}(\zeta)$ is not connected to $\mathcal
{R}(\varphi)$
in $\mathcal{R}$, and plugging \eqref{eq0goodintersectMqD},
\eqref{eq0goodintersectMproduct}, \eqref{eq0goodintersectqi}
into \eqref{eqlowbnd2trvsintersect} and using the bounds from
\eqref{eq0goodintersectqDbnd}, \eqref{eq0goodintersectqi}
we have for large $n$ and $c_{4}(\rho)$
\[
\mathbb{P}_{H}[\zeta\nleftrightarrow\varphi]\le1-\prod
_{i=0}^{D}(1-q_{i})^{2}\le2
\sum_{i=0}^{D}q_{i}\le2D\exp
\bigl(-c_{4}n^{1-2/d}\bigr).
\]
Since we assumed $|H|<2\rho n^{d-2}$, we union bound the
probability for $\{ \zeta\nleftrightarrow\varphi\} $ over
any two traversals in $H^{5.5}$, to get
%
%e28 #&#
\begin{equation}
\mathbb{P}_{H}\bigl[\mathcal{R}\bigl(H^{5.5}\bigr)\mbox{
is connected in }\mathcal {R}\bigr]\ge1-\exp\bigl(-c_{5}n^{1-2/d}
\bigr).\label{eqRH55conninRwhp}
\end{equation}

Let $\mathcal{F}$ be the event that for any $t\ge0$, any $\mathbf
{x}\in
F(t)\cap b^5$
is connected to $\mathcal{R}(H^{5.5})$ in $\mathcal{R}\cup F(t)$.
By \eqref{eqRH55conninRwhp}, to prove the lemma it remains
to show that $\mathcal{F}$ occurs w.h.p. Let $t_{\mathbf{x}}=\inf\{
t\ge0\dvtx\mathbf{x}\in F(t)\} $
and denote by $M_{\mathbf{x}}$ the component of $\mathbf{x}$ in
$F(t_{\mathbf{x}})$. If for fixed $t\ge0$ $\mathbf{x}\in F(t)$,
then $t_{\mathbf{x}}\le t$ and $M_{\mathbf{x}}\subset F(t)$. Thus,
$\mathcal{F}$ is implied by $\{ \forall\mathbf{x}\in b^5\dvtx
M_{\mathbf
{x}}\cap\mathcal{R}(H^{5.5})\ne\varnothing\} $,
which is in turn implied by $\{ \forall\mathbf{x}\in b^5\dvtx
M_{\mathbf
{x}}\cap b^{5.5}\cap\mathcal{R}_{H}\ne\varnothing\} $.
Since $F(t)$ is $b$-boundary-connected for all $t$, $M_{\mathbf
{x}}\cap b^{5.5}$
is of size at least $n/2$ for all $\mathbf{x}\in b^5$. By Lemma~\ref
{lemsetofsizenalphahit}, the probability none of the traversals
in $H$ hit $M_{\mathbf{x}}\cap b^{5.5}$ decays exponentially in
$n$. Thus, by union bound for some $c_{6}(\rho)$
\[
\mathbb{P}_{H}\bigl[\exists\mathbf{x}\in b^5\dvtx
M_{\mathbf{x}}\cap b^{5.5}\cap\mathcal{R}=\varnothing\bigr]
\le\bigl|b^5\bigr|\exp \bigl(-cn^{1-2/d}\bigr)<\exp\bigl(-c'n^{1-2/d}
\bigr),
\]
and we are done.
\end{pf}

%th5.12 #&#
\begin{thmm}
\label{thmmdenseboxesare0good}Fix $\rho>0$. Let $H$
be a $B(n)$-itinerary and let $F(t)$ be a $B$-boundary-connected-path.
There is a $C(\rho)$, $D$ such that for $n>C$
\[
\mathbb{P}_{H}\bigl[\bigl\{ \mathcal{R}\cup F(t)\dvtx t\ge0\bigr\}
\subset \Good_{0}^{\rho}(n) | \mathcal{D}_{\Dns\rho}^{\sigma}
\bigr]>1-\frac{D}{n^d}.
\]
\end{thmm}
\begin{pf}
See Section~\ref{subk-good} for the properties each subbox must possess
relative to $\mathcal{R}\cup F(t)$ for the above to hold. Using
Lemma~\ref{lemrhotraversedsubboxgerhodenseitinerary} and a union
bound on $\sigma(B)$, it suffices to show that for any fixed $b\in
\sigma(B)$,
any $\Dns\rho$-dense $b$-itinerary $h$ and any $b$-boundary-connected-path
$f(t)$,
%
%e29 #&#
\begin{eqnarray}
\mathbb{P}_{h}\biggl[|\mathcal{R}\cap b|\ge\biggl(\rho
c_{h}\wedge \frac{1}{2}\biggr)|b|\biggr]&>&1-n^{-2d},\label{eq0good1}
\\
\qquad\mathbb{P}_{h}\bigl[\forall t\ge0,\bigl(\mathcal{R}\cup f(t)\bigr)
\cap b^5\mbox{ is connected in }\bigl(\mathcal{R}\cup f(t)\bigr)\cap
b^7\bigr]&>&1-n^{-2d}.\label{eq0good2}
\end{eqnarray}
Let $m=\|b\|$. Using Lemma~\ref{lemsetofsizenalphahit} with
$M=b,\alpha=d,\beta=1$ and $c_{h}=\Dns c_{g}$, we get that the LHS
of \eqref{eq0good1} is greater than $1-\exp(-\rho c_{h}m^{d-2})$.
By Lemma~\ref{lemallconnected}, the LHS of \eqref{eq0good2}
is greater than $1-\exp(-cm^{(d-2)/d})$.\vspace*{1pt}

Since $m=\lceil\log^{4}n\rceil$ is $\omega(\log^{{d}/{(d-2)}}n)$, we
are done.
\end{pf}

%%%%%%%%%%%%%%%%%%%%%%%%%%%%%%%%%%%%%%%%%%%%%%%%%%%%%%%%%%%%%%%%%%%%%%%%%%%%%%%%%%%%%%%%
%%%%%%%%%%%%%%%%%%%%%%%%%%%% Section~3 Renormalization
%%%%%%%%%%%%%%%%%%%%%%%%%%%%%%%%%
%%%%%%%%%%%%%%%%%%%%%%%%%%%%%%%%%%%%%%%%%%%%%%%%%%%%%%%%%%%%%%%%%%%%%%%%%%%%%%%%%%%%%%%%

%s6 #&#
\section{Renormalization}\label{secRenormalization}

Refer to Sections~\ref{subRandom-walk-definitions}, \ref
{subIndependence-of-a-traversal}
and Definition~\ref{defDrhosigma} for the definitions of $\tau_{\rho}$,
an itinerary, a boundary-connected-path and $\mathcal{D}_{\rho
}^{\sigma}$,
used in this section.
%
%th6.1 #&#
\begin{thmm}
\label{thmmRenormalization}For any $u>0$, there is a $\rho(u)>0$
such that for any $k>0$,
\[
\mathbf{P}_{\mathbf{0}}\bigl[\forall t\ge uN^{d},\mathcal{T}(N)
\mbox{ is a }\bigl(\mathcal{R}_{N}(t),k,\rho\bigr)\mbox{-good torus}
\bigr]\stackrel {N} {\longrightarrow}1.
\]
\end{thmm}
\begin{pf}
Let $\mathcal{F}_{N}^{T}(b,k,\rho)$\label{pg23} be the event that a box $b$
is $(\proj^{-1}\circ\mathcal{R}_{N}(t),k,\rho)$-good
for all $t\ge T$. Since the number of top-level boxes for $\mathcal{T}(N)$
is bounded, the theorem follows by definition of a good torus if we
show that for some $\rho>0$, $\mathbf{P}_{\mathbf{0}}[\mathcal
{F}_{N}^{uN^{d}}(b)]\stackrel{N}{\to}1$
uniformly for an arbitrary top-level box $b$.

By translation invariance the above follows from showing that for
$B=B(\mathbf{0},\lceil N/10\rceil)$
\[
\mathbf{P}_{\mathbf{x}}\bigl[\mathcal{F}_{N}^{uN^{d}}(B)
\bigr]\stackrel {N} {\to}1,
\]
uniformly for $\mathbf{x}\in\mathbb{Z}^{d}$. For $\rho>0$, $\tau
_{\rho
}=\tau_{\rho}(N)$
is a random function of $S_{N}(\cdot)$ defined in Section~\ref
{subIndependence-of-a-traversal}.
Roughly, $\tau_{\rho}$ is the time it takes $S_{N}(\cdot)$ to make
$\rho
\|B\|^{d-2}$
top to bottom crossings of $\proj(B^7)$. In Lemma~\ref
{lemtoplevelhitnd-2times}, we show there is a $\rho(u)>0$ for
which $\mathbf{P}_{\mathbf{x}}[\tau_{\rho}>uN^{d}]\stackrel
{N}{\to}0$
uniformly for $\mathbf{x}\in\mathbb{Z}^{d}$. Since
\[
\mathcal{F}_{N}^{uN^{d}}(B)\supset\mathcal{F}_{N}^{\tau_{\rho
}}(B)
\setminus\bigl\{ \tau_{\rho}>uN^{d}\bigr\},
\]
it is thus enough to show $\mathbf{P}_{\mathbf{x}}[\mathcal
{F}_{N}^{\tau_{\rho}}(B)]\stackrel{N}{\to}1$
uniformly for $\mathbf{x}\in\mathbb{Z}^{d}$.

A $\rho$-dense $B$-itinerary (defined in Section~\ref{subRandom-walk-definitions})
is essentially a product space of $\rho\|B\|^{d-2}$ SRWs conditioned
to cross $B^7$ from top to bottom. A $B$-boundary-connected-path
$F(t)$ is a map from $\mathbb{Z}^{\ge0}$ to $2^{B^7}$ with certain
properties (defined before Lemma~\ref{lemPNgePrhoandFt}).
By Lemma~\ref{lemPNgePrhoandFt} there is a $\rho$-dense $B$-itinerary
$H$ (independent of $\mathbf{x}$) and a $B$-boundary-connected-path
$F(t)$ (dependent on $\mathbf{x}$) such that
\[
\mathbf{P}_{\mathbf{x}}\bigl[\mathcal{F}_{N}^{\tau_{\rho}}(B)
\bigr]\ge \mathbb{P}_{H}\bigl[\bigl\{ \mathcal{R}\cup F(t)\dvtx t\ge0
\bigr\} \subset \Good_{k}^{\rho}\bigr].
\]
By Corollary~\ref{correcursive} below, the RHS approaches one as
$N$ tends to infinity uniformly for $\rho$-dense $B(\lceil N/10\rceil
)$-itineraries
[and independently of $F(t)$].
\end{pf}
In the below lemma, we use a dimensional constant $p_{d}<1$ from
Corollary~\ref{corDominatingfieldpercolates}.
%
%le6.2 #&#
\begin{lem}
\label{lemHjgoodifhj-1good}Let $B=B(n)$, fix $j>0$ and $\rho>0$.
Let $C_{1}(\rho,j)$ be such that for any $\rho\Dns$-dense $B$-itinerary
$h$, $B$-boundary-connected-path $f(t)$, and $n>C_{1}$ we have
\[
\mathbb{P}_{h}\bigl[\bigl\{ \mathcal{R}\cup f(t)\dvtx t\ge0\bigr\}
\subset \Good_{j-1}^{\rho}\bigr]>p_{d}.
\]
Then for all $p<1$ there is a $C_{2}(p,\rho)$ such that for any
$\rho$-dense $B$-itinerary $H$, any $B$-boundary-connected-path
$F(t)$ and all $n>C_{2}$
\[
\mathbb{P}_{H}\bigl[\bigl\{ \mathcal{R}\cup F(t)\dvtx t\ge0\bigr\}
\subset \Good_{j}^{\rho}\bigr]>p.
\]
\end{lem}
\begin{pf}
Fix a $\rho$-dense $B(n)$-itinerary $H$ and a $B$-boundary-connected-path
$F(t)$. Let $\sigma_{j}=|\sigma(B)|^{{1}/{d}}$ and
let
\[
\mathcal{S}=\bigl\{ b\in\sigma(B)\dvtx\forall t\ge0,b\mbox{ is }\bigl(\mathcal
{R}\cup F(t),j-1,\rho\Dns\bigr)\mbox{-good}\bigr\}.\vadjust{\goodbreak}
\]
Observe that if $\Delta\mathcal{S}\in\mathcal{P}(\sigma_{j})$ and
$\{ \mathcal{R}\cup F(t)\dvtx t\ge0\} \subset\Good_{0}^{\rho}$,
this implies that $\{ \mathcal{R}\cup F(t)\dvtx t\ge0\} \subset
\Good_{j}^{\rho}$.

See Definition~\ref{defDrhosigma} for the definition of $\mathcal
{D}_{\Dns\rho}^{\sigma}$,\label{pg222}
which is roughly, the event that each $b\in\sigma(B)$ is traversed
top to bottom at least $\Dns\rho\|b\|^{d-2}$ times. By Theorem~\ref
{thmmdenseimpliessubboxdensewhp},
for any $q>0$ there is a $C_{1}(q,\rho)$ such that for all $n>C_{1}$,
\[
\mathbb{P}_{H}\bigl[\mathcal{D}_{\Dns\rho}^{\sigma}
\bigr]>1-q.
\]
By Theorem~\ref{thmmdenseboxesare0good}, for any $q>0$ and all
$n>C_{2}(q,\rho)$,
\[
\mathbb{P}_{H}\bigl[\bigl\{ \mathcal{R}\cup F(t)\dvtx t\ge0\bigr\}
\subset \Good_{0}^{\rho} | \mathcal{D}_{\Dns\rho}^{\sigma}
\bigr]>1-q.
\]
Thus, if we also prove that for all $n>C_{3}(q)$
%
%e31 #&#
\begin{equation}
\mathbb{P}_{H}\bigl[\Delta\mathcal{S}\in\mathcal{P}(
\sigma_{j}) | \mathcal{D}_{\Dns\rho}^{\sigma}
\bigr]>1-q,\label{eqb0condkgood}
\end{equation}
then for $q<(1-p)/4$ and $n>C_{1}(q,\rho)\vee C_{2}(q,\rho)\vee C_{3}(q)$
\[
\mathbb{P}_{H}\bigl[\bigl\{ \mathcal{R}\cup F(t)\dvtx t\ge0\bigr\}
\subset \Good_{j}^{\rho}\bigr]>p>1-4q,
\]
and we are done.

Let $F_{j}\subset B(\sigma_{j})$ and let $b\in\sigma(B)$. Write
$\mathcal{F}_{j}(F_{j},b)$ for the event that $\Delta(\mathcal
{S}\setminus B_{\Delta}(b,20))=F_{j}$.
By Corollary~\ref{corDominatingfieldpercolates} (a consequence
of the main theorem in \cite{liggett1997domination}), to prove \eqref
{eqb0condkgood}
for all $n>C(p)$, it is enough to show that for any such $F_{j}$
for which $\mathbb{P}_{H}[\mathcal{D}_{\Dns\rho}^{\sigma},\mathcal
{F}_{j}]>0$,
%
%e32 #&#
\begin{equation}
\mathbb{P}_{H}\bigl[b\in\mathcal{S} | \mathcal{D}_{\Dns\rho}^{\sigma
},
\mathcal{F}_{j}\bigr]>p_{d}.\label{eqinduction1}
\end{equation}
Since $\{ b\in\mathcal{S}\} $ is a function of $
(\mathcal{R}\cup F(t))\cap b^7$,
by Lemma~\ref{lemrhotraversedsubboxgerhodenseitinerary},
\eqref{eqinduction1} follows from our assumption.
\end{pf}
%
%co6.3 #&#
\begin{cor}
\label{correcursive}Fix $k>0$ and $p<1$. Let $B=B(n)$, let $H$
be a $\varrho$-dense $B$-itinerary, $\varrho>0$ and let $F(t)$
be a $B$-boundary-connected-path. Then for all $n>C(\varrho,k,p)$
\[
\mathbb{P}_{H}\bigl[\bigl\{ \mathcal{R}\cup f(t)\dvtx t\ge0\bigr\}
\subset \Good_{k}^{\varrho}\bigr]>p.
\]
\end{cor}
\begin{pf}
W.l.o.g. $p>p_{d}$. Let $b=b(s^{(k)}(n))$. By Theorem~\ref
{thmmdenseboxesare0good}
for any $\varrho\Dns^{k}$-dense $b$-itinerary $h$, $b$-boundary-connected-path
$f(t)$, and whenever\break  $s^{(k)}(n)>C(1-p,\rho)$
\[
\mathbb{P}_{h}\bigl[\bigl\{ \mathcal{R}\cup f(t)\dvtx t\ge0\bigr\}
\in\Good _{0}^{\varrho\Dns^{k}}\bigr]>p.
\]
Iterate Lemma~\ref{lemHjgoodifhj-1good} with above $p$ from
$j=1,\rho=\varrho\Dns^{k-1}$ to $j=k,\rho=\varrho$ to finish.
\end{pf}

%%%%%%%%%%%%%%%%%%%%%%%%%%%%%%%%%%%%%%%%%%%%%%%%%%%%%%%%%%%%%%%%%%%%%%%%%%%%%%%%%%%%%%%%%%%%
%%%%%%%%%%%%%%%%%%%% Random Interlacements
%%%%%%%%%%%%%%%%%%%%%%%%%%%%%%%%%%%%%%%%%%%
%%%%%%%%%%%%%%%%%%%%%%%%%%%%%%%%%%%%%%%%%%%%%%%%%%%%%%%%%%%%%%%%%%%%%%%%%%%%%%%%%%%%%%
%s7 #&#
\section{Random interlacements}\label{secinter}
In this section, we prove Theorem~\ref{thmmInterlacementmain}.
Notation and definition of random interlacements appear in Appendix
\ref{appC}.

%de7.1 #&#
\begin{defn}Denote by $\om_{u',u}(\operatorname{Top}\rightarrow\operatorname{Bot})$ the
set of trajectories $w\in\operatorname{Supp}(\om_{u',u}(W^{*}_\mathrm{Top}))$
such that the first exit position of $w$ from $B^7$ is in $\operatorname
{Bot}\subset\partial B^7$.
\end{defn}
%
%de7.2 #&#
\begin{defn}
Let $u_\rho=\inf\{u>0\dvtx|\om_u(\operatorname{Top}\rightarrow\operatorname
{Bot})|>\rho
N^{d-2}\}$.
%
%le7.3 #&#
\begin{lem}
$u_\rho$ is finite $\pr$ a.s.
\end{lem}
\begin{pf}
Denote by $p(N)=\pr[|\om_1(\operatorname{Top}\rightarrow\operatorname{Bot})|\ge1]$.
Then for every $N$, $p(N)>0$. By independence between $\om_{u,u'}$ and
$\om_{v,v'}$ for $u<u'\le v\le v'$, we obtain by the Borel--Cantelli
lemma that
\begin{eqnarray*}
&&\pr\bigl[\exists k, \bigl|\om_k(\operatorname{Top}\rightarrow\operatorname{Bot})\bigr|\ge\rho
N^{d-2}\bigr]\\
&&\qquad\ge\pr\Bigl[\limsup_{i\rightarrow\infty}\bigl|
\om_{i,i+1}(\operatorname {Top}\rightarrow\operatorname{Bot})\bigr|\ge1\Bigr]=1.
\end{eqnarray*}
\upqed\end{pf}
\end{defn}
%
%Let $N>0$, $\rho>0$,
We now prove the equivalent of Lemma~\ref{lemforeveryurthereisrho}.
%
%le7.4 #&#
\begin{lem}
For every $u>0$, there is a $\rho(u)>0$ such that
\[
\lim_{N\rightarrow\infty}\pr[ u_\rho<u]=1.
\]
\end{lem}
\begin{pf}
First, $|\om_{u}(W^{*}_\mathrm{Top}) |$ is $\operatorname{Poisson}
(u\mathrm{Cap}(\mathrm{Top}))$
distributed. There\break  exists a dimension dependent constant $c'_d$ such
that, $\mathrm{cap}(\mathrm{Top})=c'_dN^{d-2}$ (see \cite{lawler2010random}
Proposition~6.5.2). By the invariance principle, there is a dimension dependent
constant $c_d>0$, such that $\min_{x\in\operatorname{Top}}\prob^x[X_{\tau
_{B^7}}\in\operatorname{Bot}]\ge c_d$, for\break  large enough~$N$.
Thus, $|\om_{u}(\operatorname{Top}\rightarrow\operatorname{Bot})|$
stochastically
dominates a\break  $\operatorname{Poisson}(uc_dN^{d-2})$ distribution.
Now take any $\rho<uc_d$, and by Chebyshev's inequality we obtain that
\[
\pr\bigl[\bigl|\om_{u}(\operatorname{Top}\rightarrow\operatorname{Bot})\bigr|<\rho
N^{d-2}\bigr]\mathop{\rightarrow}_{N\rightarrow\infty}0,
\]
which concludes the lemma.
\end{pf}
%
%le7.5 #&#
\begin{lem}
For $N>0$ large enough fix a box $b=b(\lceil N/10\rceil)$,
$\rho>0$, $u>u_\rho$. Then for any $\mathcal{A}\subset2^{b^7}$
there is a $\rho$-dense $b$-itinerary $H=H(\mathcal{A})$ and a
$b$-boundary-connected-path $F(t)=F(\mathbf{x},t)$ such that
\[
\pr\bigl[\CI^u\cap b^7\in\CA\bigr]\ge
\mathbb{P}_{H}\bigl[\bigl\{ \mathcal{R}\cup F(t)\dvtx t\ge0\bigr\} \in
\mathcal{A}\bigr].
\]
\end{lem}
\begin{pf}
Since $u>u_\rho$ we know $|\om_u(\operatorname{Top}\rightarrow\operatorname
{Bot})|>\rho
N^{d-2}$. Order the trajectories in $\om_u(\operatorname{Top}\rightarrow
\operatorname
{Bot})$ by some arbitrary but fixed method. For every $1\le i\le\rho
N^{d-2}$ and trajectory $w_i\in\om_u(\operatorname{Top}\rightarrow\operatorname{Bot})$
denote by  $a(i)\in\operatorname{Top}$, the starting point of $w_i$ and by
$z(i)\in\operatorname{Bot}$, the exit point. For every $1\le i\le\rho N^{d-2}$
let $\eta_i=(a(i),z(i))$, and $H=(\eta_1,\ldots,\eta_{\lceil
\rho N^{d-2}\rceil})$. For all $t>0$, let $F(t)=\bigcup_{w\in\operatorname{Supp}(\om_{u}(W^{*}_{b^7}))}\operatorname{range}(w)\cap
b^7\setminus
\bigcup_{i=1}^{\rho N^{d-2}}\operatorname{range}(w_i)$. Then
\[
\pr\bigl[\operatorname{range}\bigl(\om\bigl({W^*_{b^7}}\bigr)\bigr)\cap
b^7\subset\mathcal{A} \bigr]=\mathbb{P}_{H}\bigl[\bigl\{
\mathcal{R}\cup F(t)\dvtx t\ge0\bigr\} \subset \mathcal{A}\bigr].
\]
\upqed\end{pf}

%th7.6 #&#
\begin{thmm}\label{thmmgoodbnd}
For every $k\in\BN$ and $u,\rho>0$, there exists a constant $\alpha
(k,\rho)$ such that
\[
\pr\bigl[\CI^u\notin\CG^\rho_k(n)\bigr]\le
\frac{\alpha}{n^{2}}.
\]
\end{thmm}
\begin{pf}
The proof follows Theorem~\ref{thmmRenormalization} without the union
on top level boxes.
\end{pf}

We now prove the bound on the heat kernel of random interlacements.
%
%th7.7 #&#
\begin{thmm}\label{thmmintheatbnd}
Let $u>0$ and let $X_n$ be a random walk on the graph $\CI^u$. For
large enough $N$, if $\CI^u\in\CG_k^\rho(N)$ and $0\in\CI^u$, there
exists a constant $C(k,\rho)$ such that
\[
\prob^u_0[X_N=0]\le\frac{C(k,\rho)\log^{(k-1)}(N)}{N^{{d}/{2}}}.
\]
\end{thmm}
\begin{pf}
Let $\varepsilon>0$. By \cite{morris2005evolving} (Theorem~2), there
exists a constant $\tilde{c}$ such that if for some $\varepsilon>0$
%
%e33 #&#
\begin{equation}
\tilde{c}\int_1^{{4} /{\varepsilon}}
\frac{dr}{r\phi^2(r)}\le n,
\end{equation}
then $\prob_0[X_n=0]\le\varepsilon$. In order to bound $\prob
_0[X_n=0]$ it
is enough to consider the isoperimetric constant of sets inside $B(n)$.
Indeed consider a new graph $\tilde{\CI}^u$ which is the same as $\CI
^u$ inside $B(n)$ but all the edges are open outside $B(n)$. Since a
random walk cannot leave $B(n)$ before time $n$, it is enough to prove
the theorem for the graph $\tilde{\CI}^u$. Next, we prove an
isoperimetric inequality for the graph $\tilde{\CI}^u$. For every set
$A\subset\Z^d$, such that $|A|>n^{1/3}$, if $A\cap B(n)=\phi$ then by
the isoperimetric inequality of $\Z^d$, $|\partial A|\ge|A|^{
{(d-1)}/{d}}$. If $A\cap B(n)\neq\phi$ and $|A\cap B(n)^c|\ge\frac
{1}{2}|A|$, by the triangle inequality and isoperimetric inequality of
$\Z^d$,
\[
|\partial A|\ge\bigl|\partial\bigl( A\cap B(n)^c\bigr)\bigr|\ge\bigl(\bigl|A\cap
B(n)^c\bigr|\bigr)^{ {(d-1)}/ {d}}\ge\biggl(\frac{1}{2}|A|
\biggr)^{{(d-1)}/ {d}}.
\]
If $A\cap B(n)\neq\phi$ and $|A\cap B(n)|>\frac{1}{2}|A|$, since
$|\partial A\cap B^c|\ge|A\cap\partial B|$ (a straight line between
two points is the shortest path)
\[
|\partial A|\ge\bigl|\partial(A\cap B)\bigr|>c(k,\rho) \bigl(\tfrac{1}{2}|A|
\bigr)^{{(d-1)}/{d}}\bigl(s^{(l)}(n)\bigr)^{d-1}.
\]

If $\phi(r)$ is realized by a set of size smaller than $N^{{1}/{3}}$,
then $\phi(r)\ge N^{-{1}/{3}}$.
By Theorem~\ref{lemphilowbnd},%\ref{corkgoodtorusisopbnd}
%
%e34 #&#
%e35 #&#
\begin{eqnarray}
\int_1^{{4}/ {\varepsilon}}
\frac{dr}{r\phi^2(r)}&=&\int_1^{ {4}/ {
\varepsilon }}\frac{dr}{r({1}/{N^{{2}/{3}}})}+\int_{N^{{1}/{3}}}^{ {4}/
{\varepsilon}}\frac{dr}{r\hat{\phi}^2(r)}\nonumber\\
& =& N^{{2}/ {3}}\log\biggl(
\frac{4}{\varepsilon}\biggr)+\int_{N^{
{1}/{3}}}^{{4} /{\ep}}
\frac{c(k,\rho)^{-1}(s^{(k)}(N)
)^{2d-2}\,dr}{rr^{-{2}/{d}}}
\\
&=&N^{{2}/{3}}\log\biggl(\frac{4}{\varepsilon}\biggr)+\frac{c'
(s^{(k)}(N))^{2d-2}}{\ep^{{d}/{2}}}.\nonumber
\end{eqnarray}
Thus, if $\ep\ge\frac{c''(s^{(k)}(N))^{2d-2}}{N^{
{d}/{2}}}$, $\prob_0[X_N=0]\le\frac{c''(s^{(k)}(N)
)^{2d-2}}{N^{{d}/{2}}}\le\frac{c'''\log^{k-1}(N)}{N^{{d}/{2}}}$.
\end{pf}
The proof of Theorem~\ref{thmmInterlacementmain} follows from
Theorems~\ref{thmmintheatbnd} and~\ref{thmmgoodbnd}.

%%%%%%%%%%%%%%%%%%%%%%%%%%%%%%%%%%%%%%%%%%%%%%%%%%%%%%%%%%%%%%%%%%%%%%%%%%%%%%%%%%%%%%%%%%%%%%
%%%%%%%%%%%%%%% Appendix
%%%%%%%%%%%%%%%%%%%%%%%%%%%%%%%%%%%%%%%%%%%%%%%%%%%%%%%%%%%%%%%%%
%%%%%%%%%%%%%%%%%%%%%%%%%%%%%%%%%%%%%%%%%%%%%%%%%%%%%%%%%%%%%%%%%%%%%%%%%%%%%%%%%%%%%%%%%%
%sA #&#
\begin{appendix}
%sB #&#
\section{}\label{appA}

Recall the notation from Section~\ref{subIndependence-of-a-traversal} and
let $\tau_{0}=0$, and for $\mathbf{z}(i)\in G,m_{i}\in\mathbb{N}$
where $i\ge1$ recursively define
\[
\tau_{i}=\tau_{i}\bigl(\bigl\{\mathbf{z}(i)\bigr\},
\{m_{i}\}\bigr)=\inf\bigl\{ t>\tau _{i-1}+m_{i-1}
\dvtx S(t)=\mathbf{z}(i)\bigr\}.
\]

%prB.1 #&#
\begin{prop}
\label{proproductformindependent}Fix $n\in\mathbb{N}$, $s_{0},\ldots,s_{n},g_{1},\ldots,g_{n}\in\mathbf{S}(G)$
and $C_{1},\ldots,\break  C_{n}\subset\mathbf{S}(G)$. Set $\mathbf
{z}(i)=s_{i}(0),m_{i}=\|s_{i}\|,\mathbf{a}(i)=s_{i}(m_{i})$.
Define the events $\mathcal{A}=\bigcap_{i=0}^{n}S(\tau_{i},\tau
_{i}+m_{i})=s_{i}$,
$\mathcal{B}(A_{1},\ldots,A_{n})=\bigcap_{i=1}^{n}S(\tau
_{i-1}+m_{i-1},\tau_{i})\in A_{i}$.
Writing $\mathcal{B}_{g}$ for $\mathcal{B}(\{g_{1}\},\ldots,\{g_{n}\})$
and $\mathcal{B}_{C}$ for $\mathcal{B}(C_{1},\ldots,C_{n})$ and
assuming\break  $\mathbf{P}_{\mathbf{z}(0)}[\mathcal{A},\mathcal
{B}_{C}]>0$
we have
\[
\mathbf{P}_{\mathbf{z}(0)}[\mathcal{A},\mathcal{B}_{g} | \mathcal
{A},\mathcal{B}_{C}]=\prod_{i=1}^{n}
\mathbf{P}_{\mathbf
{a}(i)}\bigl[S(0,\tau_{\mathbf{z}(i)})=g_{i} | S(0,
\tau_{\mathbf
{z}(i)})\in C_{i}\bigr].
\]
\end{prop}

\begin{pf}
See Figure~\ref{figproductform} for an illustration. Observe that
if for some $i$, $g_{i}\notin C_{i}$, then both sides are $0$,
thus we assume $g_{i}\in C_{i}$.
\[
\mathbf{P}_{\mathbf{z}(0)}[\mathcal{A},\mathcal{B}_{g} | \mathcal
{A},\mathcal{B}_{C}]=\mathbf{P}_{\mathbf{z}(0)}[\mathcal {A},
\mathcal{B}_{g}]/\mathbf{P}_{\mathbf{z}(0)}[\mathcal {A},
\mathcal{B}_{C}].
\]
Let $W_{1},\ldots,W_{n}\subset\mathbf{S}(G)$ be with the property
that for each $1\le i\le n$ and $w\in W_{i}$, $w=(\mathbf
{a}(i),\mathbf
{v}_{1},\ldots,\mathbf{v}_{k},\mathbf{z}(i))\in\mathbf{S}(G)$,
$\mathbf{v}_{j}\neq\mathbf{z}(i)\ \forall1\le j\le k$. For each
$(w_{1},\ldots,w_{n})\in W_{1}\times\cdots\times W_{n}$,
we decompose $\mathcal{A},\mathcal{B}(\{w_{1}\},\ldots,\{w_{n}\})$
according to the Markov property and sum to get
%
%eB.1 #&#
%eB.2 #&#
\begin{eqnarray}\label{eqeventtoproduct}
&&\mathbf{P}_{\mathbf{z}(0)}\bigl[\mathcal{A},\mathcal{B}(W_{1},\ldots,W_{n})
\bigr]\nonumber\\
&&\qquad =\Biggl(\prod_{i=1}^{n}
\mathbf{P}_{\mathbf{z}(i-1)}\bigl[S(0,m_{i-1})
=s_{i-1}\bigr]\mathbf{P}_{\mathbf{a}(i)}\bigl[S(0,
\tau_{\mathbf
{z}(i)})\in W_{i}\bigr]\Biggr)\\
&&\qquad\quad{}\times\mathbf{P}_{\mathbf{z}(n)}
\bigl[S(0,m_{n})=s_{n}\bigr].\nonumber
\end{eqnarray}

%f6 #&#
\begin{figure}

\includegraphics{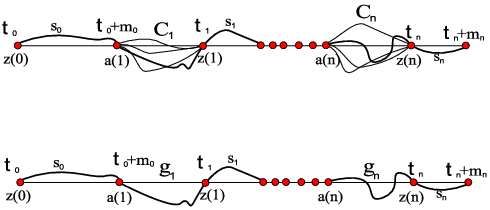}

\caption{Scheme of $\mathcal{A},\mathcal{B}_{C}$
on top and of $\mathcal{A},\mathcal{B}_{g}$ on bottom.}\label{figproductform}
\end{figure}

Since we assume $\mathbf{P}_{\mathbf{z}(0)}[\mathcal{A},\mathcal
{B}_{C}]>0$,
we have that each $C_{i}$ consists of paths with the constraints
above. Using \eqref{eqeventtoproduct} with $W_{i}=C_{i}$ and
$W_{i}=\{g_{i}\}$ we get
\[
\mathbf{P}_{\mathbf{z}(0)}[\mathcal{A},\mathcal{B}_{g} | \mathcal
{A},\mathcal{B}_{C}]=\prod_{i=1}^{n}
\bigl(\mathbf{P}_{\mathbf
{a}(i)}\bigl[S(0,\tau_{\mathbf{z}(i)})=g_{i}
\bigr]/\mathbf{P}_{\mathbf
{a}(i)}\bigl[S(0,\tau_{\mathbf{z}(i)})\in
C_{i}\bigr]\bigr),
\]
and are done.
\end{pf}
%
%prB.2 #&#
\begin{prop}
\label{propartition}Let $\mathcal{X},\mathcal{Y}$ be events in
some probability space, and let $\{ \mathcal{Y}_{\alpha}\}
_{\alpha\in I}$
be a partition of $\mathcal{Y}$ where $\forall\alpha\in I,\Pr
[\mathcal{Y}_{\alpha}]>0$.
Then for some $\gamma,\Gamma\in I$,
\[
\Pr[\mathcal{X}|\mathcal{Y}_{\gamma}]\le\Pr[\mathcal {X}|\mathcal{Y}]\le
\Pr[\mathcal{X}|\mathcal{Y}_{\Gamma}].
\]
\end{prop}
\begin{pf}
Follows from the identity,
\[
\Pr[\mathcal{X}|\mathcal{Y}]=\sum_{\alpha\in I}\Pr [
\mathcal{X}|\mathcal{Y}_{\alpha}]\Pr[\mathcal{Y}_{\alpha
}]/\Pr[
\mathcal{Y}].
\]
\upqed\end{pf}
Recall $\mathcal{J}_{B}[b]$ from Definition~\ref{defDrhosigma}.
%
%leB.3 #&#
\begin{lem}
\label{lemlowbndhitsmallbox}Let $B=B(n)$, let $b\in\sigma(B)$,
where we write $m=\|b\|$. There is a $c_{\Dns}(d)>0$, independent
of $n$, such that for any $\mathbf{a}\in A(B),\mathbf{z}\in Z(B)$
and all large~$n$,
\[
\mathbf{P}_{\mathbf{a}}\bigl[\mathcal{J}_{B}[b]|\tau
_{\partial B^7}=\tau_{\mathbf{z}}\bigr]>c_{\Dns}\biggl(
\frac{m}{n}\biggr)^{d-2}.
\]
\end{lem}
\begin{pf}
For $\mathbf{y}=(y_{1},\ldots,y_{d})\in\mathbb{Z}^{d}$, define
$\mathcal
{B}(\mathbf{y})$
to be the event $S(\cdot)$ hits $b$ at $\mathbf{y}$ and then the first
coordinate of $S(\cdot)$ hits $y_{1}+4m$ and then hits $y_{1}-8m$, while
the maximal change in the other coordinates is less than $m$. Let
$\tau_{\mathcal{B}}(\mathbf{y})=\inf\{ t\ge0\dvtx S(0,t)\subset
\mathcal
{B}(\mathbf{y})\} $
be the first time the occurrence of $\mathcal{B}(\mathbf{y})$ is
implied by $S(t)$. By Proposition~\ref{proHitthenstop} for some
$\mathbf{x}\in b$,
\begin{eqnarray*}
&&\mathbf{P}_{\mathbf{a}}\bigl[\tau_{b},\tau_{\mathcal{B}}\bigl(S(
\tau _{b})\bigr)<\tau_{\partial B^7}|\tau_{\partial B^7}=
\tau_{\mathbf{z}} \bigr]\\
&&\qquad\ge\mathbf{P}_{\mathbf{a}}[\tau_{b}<
\tau_{\partial B^7}|\tau _{\partial B^7}=\tau_{\mathbf{z}}]
\mathbf{P}_{\mathbf{x}} \bigl[\tau_{\mathcal{B}}(\mathbf{x})<
\tau_{\partial B^7}|\tau_{\partial
B^7}=\tau_{\mathbf{z}}\bigr].
\end{eqnarray*}
Using Lemma~\ref{lemlowbndhitmdmassinternal} together with
Corollary~\ref{corlowbndPaxzorxyz}, we have
\[
\mathbf{P}_{\mathbf{a}}[\tau_{b}<\tau_{\partial B^7}|\tau
_{\partial B^7}=\tau_{\mathbf{z}}]>c_{1}\biggl(\frac{m}{n}
\biggr)^{d-2}.
\]
Since $\{ \tau_{\mathcal{B}}(\mathbf{x})<\tau_{\partial B^7}
\} \subset\mathcal{J}_{B}[b]$,
we are done if we show for some $c_{2}(d)>0$
%
%eB.3 #&#
\begin{equation}
\mathbf{P}_{\mathbf{x}}\bigl[\tau_{\mathcal{B}}(\mathbf{x})<\tau
_{\partial B^7}|\tau_{\partial B^7}=\tau_{\mathbf{z}} \bigr]>c_{2}.\label{eqc3}
\end{equation}
Partitioning over $S(\tau_{\mathcal{B}}(\mathbf{x}))\in b^{10}$ and
using the Markov property, we have for some $\mathbf{y}\in b^{10}$,
%
%eB.4 #&#
\begin{eqnarray}\label{eqtaumathcalA}
&&\mathbf{P}_{\mathbf{x}}\bigl[\tau_{\mathcal{B}}(\mathbf{x})<\tau
_{\partial B^7}|\tau_{\partial B^7}=\tau_{\mathbf{z}}\bigr]\mathbf
{P}_{\mathbf{x}}[\tau_{\partial B^7}=\tau_{\mathbf{z}}]
\nonumber
\\[-8pt]
\\[-8pt]
\nonumber
&&\qquad\ge
\mathbf{P}_{\mathbf{x}}\bigl[\tau_{\mathcal{B}}(\mathbf{x})<\tau
_{\partial B^7}\bigr]\mathbf{P}_{\mathbf{y}}[\tau_{\partial
B^7}=
\tau_{\mathbf{z}}].
\end{eqnarray}
By the invariance principle, $\mathbf{P}_{\mathbf{x}}[\tau
_{\mathcal{B}}(\mathbf{x})<\tau_{\partial B^7}]$
is bounded away from zero by a dimensional constant independent of
$\mathbf{x}$. Since $\mathbf{x},\mathbf{y}$ from \eqref{eqtaumathcalA}
are contained in $B^{6}$ for all large $n$, we use Lemma~\ref
{lemupnlowbndaxorxz}
to get \eqref{eqc3}.
\end{pf}
In the lemma below, we look at the number vertices hit in an interior set
$M\subset B^{6}$ by a $B$-traversal, and lower bound the probability
for this number to be small in terms of $|M|$.
%
%leB.4 #&#
\begin{lem}
\label{lemlowbndhitmdmassinternal}Let $B=B(n)$, let $M\subset B^{6}$,
set $\mathbf{a}\in B^7$ and $\mathbf{z}\in Z(B)$. Let $X(M)=|
\{ \mathbf{v}\in M\dvtx\tau_{\mathbf{v}}<\tau_{\partial B^7}\} |$
and let $\mu_{X}=\mu_{X}(\mathbf{a},\mathbf{z})=\mathbf
{E}_{\mathbf
{a}}[X(M) | \tau_{\partial B^7}=\tau_{\mathbf{z}}]$.
There is a $c_{1}(d)>0$, independent of $n$, $\mathbf{a}$ and $\mathbf{z}$,
such that for all large $n$
\[
\mathbf{P}_{\mathbf{a}}\bigl[X(M)\ge\tfrac{1}{2}\mu_{X} |
\tau _{\partial B^7}=\tau_{\mathbf{z}}\bigr]>c_{1}
\mu_{X}|M|^{-2/d}.
\]
Thus, if \textup{$\mathbf{P}_{\mathbf{a}}[\tau_{\mathbf{v}}<\tau
_{\partial B^7} | \tau_{\partial B^7}=\tau_{\mathbf{z}}]>f(n)$
}for every $\mathbf{v}\in M$, then since $X(M)=\sum_{\mathbf{v}\in
M}\mathbf{1}_{\{ \tau_{\mathbf{v}}<\tau_{\partial B^7}\} }$,
we have
\[
\mathbf{P}_{\mathbf{a}}\bigl[X(M)\ge\tfrac{1}{2}\mu_{X} |
\tau _{\partial B^7}=\tau_{\mathbf{z}}\bigr]>c_{1}|M|^{1-2/d}f(n).
\]
\end{lem}
\begin{pf}
Write $\mu_{X^{2}}$ for $\mathbb{E}_{\mathbf{a}}[X^{2} | \tau
_{\partial B^7}=\tau_{\mathbf{z}}]$.
By the Paley--Zygmund inequality, $\mathbf{P}_{\mathbf{a}}[X\ge
\frac{1}{2}\mu_{X}]\ge\frac{1}{4}\mu_{X}^{2}/\mu_{X^{2}}$,
so enough to show
\[
\mu_{X^{2}}<Cm^{2}\mu_{X},
\]
where $m=|M|^{1/d}$.\vadjust{\goodbreak}

By linearity,
\[
\mu_{X^{2}}=\sum_{\mathbf{v}\in M}\sum
_{\mathbf{w}\in M}\mathbf {P}_{\mathbf{a}}[\tau_{\mathbf{w}},
\tau_{\mathbf{v}}<\tau _{\partial B^7} | \tau_{\partial B^7}=
\tau_{\mathbf{z}}].
\]
For two vertices $\mathbf{x},\mathbf{y,}$ let $\tau_{\mathbf
{v},\mathbf
{w}}=\inf\{ t\ge\tau_{\mathbf{x}}\dvtx S(t)=\mathbf{y}\} $.
By a union bound
\begin{eqnarray*}
\mathbf{P}_{\mathbf{a}}[\tau_{\mathbf{w}},\tau_{\mathbf{v}}<\tau
_{\partial B^7} | \tau_{\partial B^7}=\tau_{\mathbf{z}}] & \le &
\mathbf{P}_{\mathbf{a}}[\tau_{\mathbf{w}}\le\tau_{\mathbf
{w},\mathbf{v}}<
\tau_{\partial B^7} | \tau_{\partial B^7}=\tau _{\mathbf{z}}]
\\
& & {}+\mathbf{P}_{\mathbf{a}}[\tau_{\mathbf{v}}\le \tau_{\mathbf
{v},\mathbf{w}}<
\tau_{\partial B^7} | \tau_{\partial B^7}=\tau _{\mathbf{z}}].
\end{eqnarray*}
By Bayes theorem and the Markov property,
\begin{eqnarray*}
&&\mathbf{P}_{\mathbf{a}}[\tau_{\mathbf{w}}\le\tau_{\mathbf
{w},\mathbf{v}}<\tau_{\partial B^7} | \tau_{\partial B^7}=\tau
_{\mathbf{z}}]\\
&&\qquad =  \frac{\mathbf{P}_{\mathbf{a}}[\tau_{\mathbf{w}}<\tau
_{\partial B^7} | \tau_{\partial B^7}=\tau_{\mathbf{z}}]\mathbf
{P}_{\mathbf{a}}[\tau_{\mathbf{w}}\le\tau_{\mathbf{w},\mathbf
{v}}<\tau_{\partial B^7}=\tau_{\mathbf{z}} | \tau_{\mathbf
{w}}<\tau
_{\partial B^7}]}{\mathbf{P}_{\mathbf{a}}[\tau_{\partial
B^7}=\tau_{\mathbf{z}} | \tau_{\mathbf{w}}<\tau_{\partial B^7}
]}
\\
&&\qquad =  \mathbf{P}_{\mathbf{a}}[\tau_{\mathbf{w}}<\tau_{\partial
B^7} |
\tau_{\partial B^7}=\tau_{\mathbf{z}}]\mathbf{P}_{\mathbf
{w}}[
\tau_{\mathbf{v}}<\tau_{\partial B^7} | \tau_{\partial
B^7}=
\tau_{\mathbf{z}}].
\end{eqnarray*}
Again by the Markov property,
\[
\mathbf{P}_{\mathbf{w}}[\tau_{\mathbf{v}}<\tau_{\partial B^7} |
\tau_{\partial B^7}=\tau_{\mathbf{z}}]=\frac{\mathbf{P}_{\mathbf
{v}}[\tau_{\partial B^7}=\tau_{\mathbf{z}}]\mathbf
{P}_{\mathbf{w}}[\tau_{\mathbf{v}}<\tau_{\partial B^7}
]}{\mathbf{P}_{\mathbf{w}}[\tau_{\partial B^7}=\tau_{\mathbf
{z}}]}.
\]
So by Lemma~\ref{lemupnlowbndaxorxz}, since $\mathbf{v},\mathbf
{w}\in M\subset B^{6}$,
\[
\mathbf{P}_{\mathbf{w}}[\tau_{\mathbf{v}}<\tau_{\partial B^7} |
\tau_{\partial B^7}=\tau_{\mathbf{z}}]<C\mathbf{P}_{\mathbf
{w}}[
\tau_{\mathbf{v}}<\tau_{\partial B^7}]<C\mathbf {P}_{\mathbf{w}}[
\tau_{\mathbf{v}}<\infty].
\]
Thus, by symmetry,
\[
\mu_{X^{2}}\le2\sum_{\mathbf{w}\in M}
\mathbf{P}_{\mathbf{a}}[\tau _{\mathbf{w}}<\tau_{\partial B^7} |
\tau_{\partial B^7}=\tau _{\mathbf
{z}}]\sum_{\mathbf{v}\in M}C
\mathbf{P}_{\mathbf{w}}[\tau _{\mathbf{v}}<\infty].
\]
By Markov's inequality, $\mathbf{P}_{\mathbf{w}}[\tau_{\mathbf
{v}}<\infty]<G(\mathbf{w},\mathbf{v})$
where $G(\cdot,\cdot)$ is the Green's function of a simple random
walk on $\mathbb{Z}^{d}$. Standard estimates for $G(\cdot,\cdot)$
(see, e.g., Theorem~1.5.4 in \cite{lawler1996intersections}) give that
$\mathbf{P}_{\mathbf{w}}[\tau_{\mathbf{v}}<\infty]<C(d)\|
\mathbf{w}-\mathbf{v}\|_{2}^{2-d}$,
and thus
\[
\sum_{\mathbf{v}\in M}\mathbf{P}_{\mathbf{w}}[
\tau_{\mathbf
{v}}<\infty]<\sum_{\mathbf{v}\in M}\|\mathbf{w}-
\mathbf{v}\|_{2}^{2-d}.
\]
For some $\hat{c}(d)<\infty$, and all $r>0$, a ball of radius $\hat{c}r$
around the origin contains at least $r^{d}$ vertices in $\mathbb{Z}^{d}$.
Since the RHS above can only be increased by moving a vertex in $M$
closer to $\mathbf{w}$, we have
\[
\sum_{\mathbf{v}\in M}\mathbf{P}_{\mathbf{w}}[
\tau_{\mathbf
{v}}<\infty]<\hat{C}\sum_{r=1}^{\hat
{c}m}r^{d-1}r^{2-d}<Cm^{2}=C|M|^{2/d}.
\]
Since $\mu_{X}=\sum_{\mathbf{w}\in M}\mathbf{P}_{\mathbf
{a}}[\tau_{\mathbf{w}}<\tau_{\partial B^7} | \tau_{\partial
B^7}=\tau_{\mathbf{z}}]$,
we are done.
\end{pf}
Let $\tau_{0}\dvtx(\mathbb{Z}^{d})^{\mathbb{Z}^{\ge0}}\to\mathbb
{Z}^{\ge0}$
be a stopping time for the random walk $S(t)$. We denote by $\tau_{0}^{t}$
the stopping\vadjust{\goodbreak} time on the $t$-time shifted sequences, that is, $\tau
_{0}^{t}(\mathbf{a}_{0},\mathbf{a}_{1},\ldots)=\tau_{0}(\mathbf
{a}_{t},\mathbf{a}_{t+1},\ldots)+t$.
We call $\tau_{0}$ a \emph{simple stopping time} if $\tau_{0}^{t}\ge
\tau_{0}$
for every $t\ge0$.
%
%prB.5 #&#
\begin{prop}
\label{proHitthenstop}Let $B=B(n)$, set $\mathbf{a}\in B^7$
and $\mathbf{z}\in Z(B)$. Let $\tau_{1},\tau_{2}$ be simple stopping
times (see above). Then there exists a $\mathbf{x}$ satisfying
$\mathbf
{P}_{\mathbf{a}}[S(\tau_{1})=\mathbf{x},\tau_{1}<\tau_{\partial
B^7} | \tau_{\partial B^7}=\tau_{\mathbf{z}}]>0$
such that
\[
\mathbf{P}_{\mathbf{a}}[\tau_{1},\tau_{2}\le
\tau_{\partial B^7} | \tau_{\partial B^7}=\tau_{\mathbf{z}}]\ge
\mathbf{P}_{\mathbf
{a}}[\tau_{1}<\tau_{\partial B^7} |
\tau_{\partial B^7}=\tau _{\mathbf{z}}]\mathbf{P}_{\mathbf{x}}[
\tau_{2}\le\tau _{\partial B^7} | \tau_{\partial B^7}=
\tau_{\mathbf{z}}].
\]
\end{prop}
\begin{pf}
Let
\[
\pi_{\mathbf{y}}=\mathbf{P}_{\mathbf{a}}\bigl[S(\tau_{1})=
\mathbf {y},\tau_{1}<\tau_{\partial B^7} | \tau_{\partial B^7}=\tau
_{\mathbf
{z}}\bigr].
\]
For $\mathbf{y}$ satisfying $\pi_{\mathbf{y}}>0$, we have by Bayes
%
%eB.5 #&#
\begin{eqnarray}\label{eqstoptime1}
&&\mathbf{P}_{\mathbf{a}}\bigl[\tau_{1},\tau_{2}\le
\tau_{\partial
B^7},S(\tau_{1})=\mathbf{y} | \tau_{\partial B^7}=
\tau_{\mathbf
{z}}\bigr]
\nonumber
\\[-8pt]
\\[-8pt]
\nonumber
&&\qquad=\frac{\pi_{\mathbf{y}}\mathbf{P}_{\mathbf{a}}[\tau
_{1},\tau_{2}\le\tau_{\partial B^7}=\tau_{\mathbf{z}} | S(\tau
_{1})=\mathbf{y},\tau_{1}<\tau_{\partial B^7}]}{\mathbf
{P}_{\mathbf{a}}[\tau_{\partial B^7}=\tau_{\mathbf{z}} | S(\tau
_{1})=\mathbf{y},\tau_{1}<\tau_{\partial B^7}]}.
\end{eqnarray}
Since $\tau_{2}$ is a simple stopping time,
\begin{eqnarray*}
&&\mathbf{P}_{\mathbf{a}}\bigl[\tau_{1},\tau_{2}\le
\tau_{\partial
B^7}=\tau_{\mathbf{z}} | S(\tau_{1})=\mathbf{y},
\tau_{1}<\tau _{\partial B^7}\bigr] \\
&&\qquad \ge \mathbf{P}_{\mathbf{a}}
\bigl[\tau_{1}\le \tau_{2}^{\tau_{1}}\le
\tau_{\partial B^7}=\tau_{\mathbf{z}} | S(\tau _{1})=\mathbf{y},
\tau_{1}<\tau_{\partial B^7}\bigr].
\end{eqnarray*}
Plugging the above into (\ref{eqstoptime1}) and using the strong
Markov property, we get
\[
\mathbf{P}_{\mathbf{a}}\bigl[\tau_{1},\tau_{2}\le
\tau_{\partial
B^7},S(\tau_{1})=\mathbf{y} | \tau_{\partial B^7}=
\tau_{\mathbf
{z}}\bigr]\ge\pi_{\mathbf{y}}\mathbf{P}_{\mathbf{y}}[
\tau_{2}\le \tau_{\partial B^7} | \tau_{\partial B^7}=
\tau_{\mathbf{z}}].
\]
Let $\mathbf{x}\in\{ \mathbf{y}\dvtx\pi_{\mathbf{y}}>0\} $
be the vertex for which $\mathbf{P}_{\mathbf{x}}[\tau_{2}\le\tau
_{\partial B^7} | \tau_{\partial B^7}=\tau_{\mathbf{z}}]$
is minimal. Summing both sides over
$\{ \mathbf{y}\dvtx\pi_{\mathbf{y}}>0\}$,
we are done.
\end{pf}
We quote the Harnack principle for $\mathbb{Z}^{d}$ from Theorem~1.7.6
in \cite{lawler1996intersections}.
%
%prB.6 #&#
\begin{prop}
\label{proHarnack}Let $U$ be a compact subset of $\mathbb{R}^{d}$
contained in a connected open set $V$. Then there exists a
$c=c(U,V)<\infty$
such that if $A_{n}=nU\cap\mathbb{Z}^{d}$, $D_{n}=nV\cap\mathbb{Z}^{d}$,
and $f\dvtx D_{n}\cup\partial D_{n}\to[0,\infty)$ is harmonic in $D_{n}$,
then
\[
f(x)\le cf(y), x,y\in A_{n}.
\]
\end{prop}
%
%leB.7 #&#
\begin{lem}
\label{lemupbndxtoz}Let $B=B(n)$ and let $F$ be the union of
all hyperplanes in $\mathbb{Z}^{d}$ that intersect $B^{6}$ and are
parallel to $Z(B)$. There is a $C>0$ such that for any $\mathbf{y}\in
Z(B)\cup A^{+}(B)$
and any $\mathbf{v}\in F\cap B^7$,
\[
\mathbf{P}_{\mathbf{v}}[\tau_{\partial B^7}=\tau_{\mathbf{y}}
]<Cn^{1-d}.
\]
\end{lem}
\begin{pf}
We prove for $\mathbf{y}\in Z(B)$. The proof $A^{+}(B)$ is the same
so we omit it. Let $H$ be the infinite hyperplane in $\mathbb{Z}^{d}$
that contains $Z(B)$, and let $H_{0}$ a parallel hyperplane,
which\vadjust{\goodbreak}
is the component of $\partial_{B^7}F$ closer to $Z(B)$. Let $h(\mathbf{y})$
be the $l_{1}$-closest vertex to $\mathbf{y}$ in $H_{0}$. By vertex
transitivity, there is a function $g(n)$ such that for any $\mathbf
{y}\in Z(B)$,
$\mathbf{P}_{h(\mathbf{y})}[\tau_{H}=\tau_{\mathbf{y}}]=g(n)$.
Observe that $\mathbf{P}_{(\cdot)}[\tau_{H}=\tau_{\mathbf{y}}]$
is a nonnegative harmonic function in the component of $\mathbb
{Z}^{d}\setminus H$
containing $H_{0}$, so by the Harnack principle for $\mathbb{Z}^{d}$
(Proposition~\ref{proHarnack}), for some $c>0$, any $\mathbf{v}\in
F\cap B^7,\mathbf{y}\in Z(B)$
satisfies
\[
c\mathbf{P}_{\mathbf{v}}[\tau_{H}=\tau_{\mathbf{y}}]>\mathbf
{P}_{h(\mathbf{y})}[\tau_{H}=\tau_{\mathbf{y}}]=g(n).
\]
Summing both sides over $\mathbf{y}\in Z(B)$, we get
\[
g(n)<Cn^{1-d}.
\]
Since $\{ \tau_{\partial B^7}=\tau_{\mathbf{y}}\} \subset
\{ \tau_{H}=\tau_{\mathbf{y}}\} $,
another application of the Harnack principle finishes the proof.
\end{pf}
%
%coB.8 #&#
\begin{cor}
\label{corupbndatoz}Let $B=B(n)$. There is a $C>0$ such that
for any $\mathbf{a}\in A(B),\mathbf{z}\in Z(B)$
\[
\mathbf{P}_{\mathbf{a}}[\tau_{\mathbf{z}}=\tau_{\partial B^7}
]<Cn^{-d}.
\]
\end{cor}
\begin{pf}
Using the notation of Lemma~\ref{lemupbndxtoz}, by the Markov
property $F$
\[
\mathbf{P}_{\mathbf{a}}[\tau_{\mathbf{z}}=\tau_{\partial B^7} ]=\sum
_{\mathbf{x}}\mathbf{P}_{\mathbf{a}}\bigl[
\tau_{F}<\tau _{\partial
B^7},S(\tau_{F})=\mathbf{x}
\bigr]\mathbf{P}_{\mathbf
{x}}[\tau_{\mathbf{z}}=\tau_{\partial B^7}].
\]
The right term is uniformly bounded by $Cn^{1-d}$ by Lemma~\ref
{lemupbndxtoz}.
Summing over $\mathbf{x}$, the event $\{ \tau_{F}<\tau_{\partial
B^7}\} $
implies that a one dimensional random walk starting at $1$ hits $n$
before hitting $0$, an event of probability $n^{-1}$.
\end{pf}
%
%prB.9 #&#
\begin{prop}
\label{proSymmgreen}Let $B=B(n)$. There is a $c(d)>0$ such that
for any $A\subset\mathbb{Z}^{d}$ and $\mathbf{v,}\mathbf{w}\in
\mathbb
{Z}^{d}\setminus A$
\[
c<\frac{\mathbf{P}_{\mathbf{v}}[\tau_{\mathbf{w}}<\tau_{A}
]}{\mathbf{P}_{\mathbf{w}}[\tau_{\mathbf{v}}<\tau_{A}]}<c^{-1}.
\]
\end{prop}
\begin{pf}
Write $G_{A}(\mathbf{v},\mathbf{w})$ for the Green's function
of a random walk killed on hitting $A$, that is, the expected number
of visits to $\mathbf{w}$ for a walk starting at $\mathbf{v}$ before
it hits $A$. By elementary Markov theory, we have symmetry of Green's
function, $G_{A}(\mathbf{v},\mathbf{w})=G_{A}(\mathbf
{w},\mathbf{v})$
and the following identity:
\[
\mathbf{P}_{\mathbf{v}}[\tau_{\mathbf{w}}<\tau_{\partial B^7}
]G_{A}(\mathbf{w},\mathbf{w})=\mathbf{P}_{\mathbf{w}} [
\tau_{\mathbf{v}}<\tau_{\partial B^7}]G_{A}(\mathbf {v},
\mathbf{v}).
\]
For any $\mathbf{v}\in\mathbb{Z}^{d}\setminus A$, $G_{A}(\mathbf
{v},\mathbf{v})\ge1$,
and is bounded above by the reciprocal of the probability a simple
random walk never returns to $\mathbf{v}$, which by transience in
$d>2$, is a finite dimensional constant.
\end{pf}
%
%leB.10 #&#
\begin{lem}
\label{lemupnlowbndaxorxz}Let $B=B(n)$. There is a $c>0$
such that for any $\mathbf{a}\in A(B),\mathbf{z}\in Z(B)$ and any
$\mathbf{x}\in B^{6}$,
%
%eB.6 #&#
\begin{equation}
cn^{1-d}<\mathbf{P}_{\mathbf{a}}[\tau_{\mathbf{x}}<
\tau_{\partial
B^7}],\qquad \mathbf{P}_{\mathbf{x}}[\tau_{\mathbf{z}}=\tau
_{\partial B^7}]<c^{-1}n^{1-d}.\label{equplowbnd}\vadjust{\goodbreak}
\end{equation}
\end{lem}
\begin{pf}
Let $D_{\mathbf{x}}(r)=\{ \mathbf{v}\in\mathbb{Z}^{d}\dvtx\|
\mathbf
{v}-\mathbf{x}\|_{2}\le r\} $.
Lemma~1.7.4 in \cite{lawler1996intersections} tells us there is a
$c_{1}(d)>0$ such that for any $\mathbf{r}\in\partial D_{\mathbf{0}}(r)$,
%
%eB.7 #&#
\begin{equation}
\mathbf{P}_{\mathbf{0}}[\tau_{\partial D_{\mathbf{0}}(r)}=\tau _{\mathbf{r}}]>c_{1}n^{1-d}.
\end{equation}
Fix $\mathbf{y}\in A^{+}(B)$. Then there is a $\mathbf{v}\in B^{6}$
such that $\mathbf{y}\in\partial D_{\mathbf{v}}(n),D_{\mathbf
{v}}(n)\subset B^7$.
Since $\{ \tau_{\partial D_{\mathbf{v}}(n)}=\tau_{\mathbf{y}}
\} $
implies $\{ \tau_{\partial B^7}=\tau_{\mathbf{y}}\}$, we
get that
%
%eB.8 #&#
\begin{equation}
\mathbf{P}_{\mathbf{v}}[\tau_{\partial B^7}=\tau_{\mathbf{y}}
]>c_{1}n^{1-d}.\label{eqlowbndxtoz}
\end{equation}
The probability to exit $B^7$ at $\mathbf{y}$ is a nonnegative
harmonic function in $B^7$. Thus, by the Harnack principle for $\mathbb{Z}^{d}$
(Proposition~\ref{proHarnack}), and since $c_{1}$ is independent
of $\mathbf{y}$, the above is true for any $\mathbf{v}\in B^{6}$
and any $\mathbf{y}\in A^{+}(B)$ with an appropriate constant $c_{2}>0$
replacing $c_{1}$.

The same argument proves the lower bound in (\ref{equplowbnd}) for
$\mathbf{P}_{\mathbf{x}}[\tau_{\mathbf{z}}=\tau_{\partial
B^7}]$.

Next, by Proposition~\ref{proSymmgreen} we have $\mathbf{P}_{\mathbf
{a}}[\tau_{\mathbf{x}}<\tau_{\partial B^7}]>c\mathbf
{P}_{\mathbf{x}}[\tau_{\mathbf{a}}<\tau_{\partial B^7}]$.
Let $\mathbf{a}^{+}$ be $\mathbf{a}$'s neighbor in $A^{+}(B)$.
Since $\{ \tau_{\mathbf{a}}<\tau_{\partial B^7}\} \supset
\{ \tau_{\mathbf{a}^{+}}=\tau_{\partial B^7}\} $
and by \eqref{eqlowbndxtoz}, we get
\[
\mathbf{P}_{\mathbf{x}}[\tau_{\mathbf{a}}<\tau_{\partial B^7} ]\ge
\mathbf{P}_{\mathbf{x}}[\tau_{\mathbf{a}^{+}}=\tau_{\partial
B^7}]>cn^{1-d},
\]
which proves the lower bound in (\ref{equplowbnd}) for $\mathbf
{P}_{\mathbf{a}}[\tau_{\mathbf{x}}<\tau_{\partial B^7}]$
as well.

The upper bound for $\mathbf{P}_{\mathbf{x}}[\tau_{\mathbf{z}}=\tau
_{\partial B^7}]$
is immediate from Lemma~\ref{lemupbndxtoz}. To prove for $\mathbf
{P}_{\mathbf{a}}[\tau_{\mathbf{x}}<\tau_{\partial B^7}]$,
we first use the lemma to get
\[
\mathbf{P}_{\mathbf{x}}[\tau_{\mathbf{a}^{+}}=\tau_{\partial
B^7}]<Cn^{1-d},
\]
which implies the bound for $\mathbf{P}_{\mathbf{x}}[\tau_{\mathbf
{a}}<\tau_{\partial B^7}]$,
since by the Markov property, the probability for exiting $B^7$
one step after hitting $\mathbf{a}$ for the first time is
\[
\mathbf{P}_{\mathbf{x}}[\tau_{\mathbf{a}}<\tau_{\partial B^7} ]\cdot
\frac{1}{2d}\le\mathbf{P}_{\mathbf{x}}[\tau_{\mathbf
{a}^{+}}=
\tau_{\partial B^7}].
\]
Using Proposition~\ref{proSymmgreen} again, we get the bound with
a new factor for $\mathbf{P}_{\mathbf{a}}[\tau_{\mathbf{x}}<\tau
_{\partial B^7}]$.
\end{pf}
%
%coB.11 #&#
\begin{cor}
\label{corlowbndPaxzorxyz}Let $B=B(n)$. There is a $c>0$
such that for any $\mathbf{a}\in A(B)\cup B^{6},\mathbf{z}\in Z(B)$
and any $\mathbf{x}\in B^{6}$,
%
%eB.9 #&#
\begin{equation}
\mathbf{P}_{\mathbf{a}}[\tau_{\mathbf{x}}<\tau_{\partial B^7} |
\tau_{\partial B^7}=\tau_{\mathbf{z}}]>cn^{2-d}.
\end{equation}
\end{cor}
\begin{pf}
By the Markov property,
\begin{eqnarray*}
\mathbf{P}_{\mathbf{a}}[\tau_{\mathbf{x}}<\tau_{\partial B^7} |
\tau_{\partial B^7}=\tau_{\mathbf{z}}]\mathbf{P}_{\mathbf{a}} [
\tau_{\partial B^7}=\tau_{\mathbf{z}}] & = & \mathbf{P}_{\mathbf
{a}}[
\tau_{\partial B^7}=\tau_{\mathbf{z}} | \tau_{\mathbf
{x}}<
\tau_{\partial B^7}]\mathbf{P}_{\mathbf{a}}[\tau _{\mathbf{x}}<
\tau_{\partial B^7}]
\\
& = & \mathbf{P}_{\mathbf{x}}[\tau_{\partial B^7}=\tau_{\mathbf
{z}}]
\mathbf{P}_{\mathbf{a}}[\tau_{\mathbf{x}}<\tau _{\partial B^7}].
\end{eqnarray*}
If $\mathbf{a}\in A(B)$, then Lemma~\ref{lemupnlowbndaxorxz}
and Corollary~\ref{corupbndatoz} give the bound. If $\mathbf{a}\in B^{6}$,
then Lemma~\ref{lemupnlowbndaxorxz} gives us the LHS is greater
than $c\mathbf{P}_{\mathbf{a}}[\tau_{\mathbf{x}}<\tau_{\partial
B^7}]$.

For $r>0$, let $\mathfrak{b}^{r}=\{ \mathbf{x}\in\mathbb
{R}^{d}\dvtx\forall i,1\le i\le d,|x_{i}|<r/2\} $
and for $\mathbf{y}\in\mathbb{R}^{d}$ let $\mathfrak{d}(\mathbf
{y},r)=\{ \mathbf{x}\in\mathbb{R}^{d}\dvtx\|\mathbf{x}-\mathbf
{y}\|
_{2}<r\} $.
Choose $K(d)$ points $\mathbf{y}_{1},\ldots,\mathbf{y}_{K}\in
\mathfrak{b}^{6}$
such that $\mathfrak{b}^{6}\subset\bigcup_{i=1}^{K}\mathfrak
{d}(\mathbf
{y}_{i},0.1)\subset\mathfrak{b}^{6.1}$.
Let $D_{i}^{\alpha}(n)=\mathfrak{d}(n\mathbf{y}_{i},\alpha n)\cap
\mathbb
{Z}^{d}$.
Then for $\alpha\ge0.1$ and all $n$
\[
B^{6}\subset\bigcup_{i=1}^{K}D_{i}^{\alpha}
\subset B^{6+\alpha}.
\]
Let $p_{\mathbf{a}}(\mathbf{x})=\mathbf{P}_{\mathbf{a}}[\tau
_{\mathbf{x}}<\tau_{\partial B^7}]$.
To show $p_{\mathbf{a}}(\mathbf{x})>cn^{2-d}$ uniformly in $\mathbf
{x},\mathbf{a}\in B^{6}$,
it is enough to show, w.l.o.g., that there is a $c_{1}>0$ such that for
any $\mathbf{x}\in D_{1}^{0.1},\mathbf{a}\in B^{6}$, $p_{\mathbf
{a}}(\mathbf{x})>c_{1}n^{2-d}$.
Since $p_{\mathbf{a}}(\mathbf{x})$ is harmonic as a function of
$\mathbf{a}$
in $B^{7}\setminus\{ \mathbf{x}\} $, by the maximum (minimum)
principle,
\[
\min_{\mathbf{a}\in D_{1}^{0.2}\setminus\{ \mathbf{x}\}
}p_{\mathbf{a}}(\mathbf{x})\ge\min
_{\mathbf{a}\in\partial
D_{1}^{0.2}\cup\{ \mathbf{x}\} }p_{\mathbf{a}}(\mathbf{x}).
\]
Since $p_{\mathbf{x}}(\mathbf{x})=1$, and $\partial
D_{1}^{0.2}\subset B^{6.5}$,
it is thus enough to lower bound $p_{\mathbf{a}}(\mathbf{x})$ for
$\mathbf{a}\in B^{6.5}\setminus D_{1}^{0.2}$. Since $p_{\mathbf
{a}}(\mathbf{x})$
is harmonic and positive in $B^{7}\setminus D_{1}^{0.1}$, by the
Harnack principle for $\mathbb{Z}^{d}$ (Proposition~\ref{proHarnack}),
there is a $c_{2}(d)>0$ such that for any $\mathbf{a},\mathbf{b}\in
B^{6.5}\setminus D_{1}^{0.2}$
\[
p_{\mathbf{b}}(\mathbf{x})\ge c_{2}p_{\mathbf{a}}(\mathbf{x}).
\]
Thus, it is enough to bound for some fixed $\mathbf{a}\in\partial
D_{1}^{0.2}\cap B^{6}$.
Let $D_{*}=\mathfrak{d}(\mathbf{a},0.6n)\cap\mathbb{Z}^{d}$ and note
that $\mathbf{x}\in D_{*}\subset B^{7}$, implying $p_{\mathbf
{a}}(\mathbf{x})\ge\mathbf{P}_{\mathbf{a}}[\tau_{\mathbf{x}}<\tau
_{\partial D_{*}}]$.
By Proposition~1.5.9 in \cite{lawler1996intersections}, since $\mathbf
{x}\in\mathfrak{d}(\mathbf{a},0.4n)\cap\mathbb{Z}^{d}$,
$\mathbf{P}_{\mathbf{x}}[\tau_{\mathbf{a}}<\tau_{\partial
D_{*}}]\ge cn^{2-d}$,
and by Proposition~\ref{proSymmgreen} we are done.
\end{pf}

%%%%%%%%%%%%%%%%%%%%%%%%%%%%%%%%%%%%%%%%%%%%%%%%%%%%%%%%%%%%%%%%%%%%%%%%%%%%%%%%%%%%%%%%%
%%%%%%%%%%%%%%%%%%%%%%%%%%% Section~4 Distance bound
%%%%%%%%%%%%%%%%%%%%%%%%%%%%%%%%%%%
%%%%%%%%%%%%%%%%%%%%%%%%%%%%%%%%%%%%%%%%%%%%%%%%%%%%%%%%%%%%%%%%%%%%%%%%%%%%%%%%%%%%%%%%%

%sC #&#
\section{Distance bound}\label{secDistance-bound}\label{appB}

In this section, we prove the following theorem.
%
%thC.1 #&#
\begin{thmm}
\label{thmmDistancebound}Let $\omega_{0}\subset\mathcal{T}(N)$. If
$\omega_0\subset\mathcal{T}(N)$ is $(N,k,\rho)$-good (see Section~\ref
{subk-goodtorus})
where $k\ge1,\rho>0$, then there is a $C(k,\rho)<\infty$ such that
for all large $N$ and any two vertices $\mathbf{x},\mathbf{y}\in
\omega_{0}$
\[
d_{\omega_{0}}(\mathbf{x},\mathbf{y})<Cd_{\mathcal{T}}(\mathbf {x},
\mathbf{y})\log^{(k-1)}N+C(\log N)^{4d+2},
\]
where $\log^{(m)}N$ is $\log(\cdot)$ iterated $m$ times of $N$.
\end{thmm}
We start by reducing from the torus to top-level.boxes. To prove the
theorem, it is enough to show that there exists a $C(k,\rho)<\infty$
such that for all large $n$, any $\omega\in\Good_{k}^{\rho}(n)$
and any $\mathbf{x},\mathbf{y}\in\omega\cap b^5(n)$ satisfy
%
%eC.1 #&#
\begin{equation}
d_{\omega}(\mathbf{x},\mathbf{y})<Cd_{b^7(n)}(\mathbf{x},\mathbf
{y})\log ^{(k-1)}n+C(\log n)^{4d+2}.\label{eqdistbnd2}
\end{equation}
Note that while $\omega_{0}$ is a subgraph of $\mathcal{T}$ as far
as graph distance, we require \eqref{eqdistbnd2} to hold for $\omega$
as a subgraph of $\mathbb{Z}^{d}$ (no wrap around).\vadjust{\goodbreak} To see why this
is enough, let $\mathbf{x},\mathbf{y}\in\mathcal{T}(N)$ and set
$n=\lceil N/10\rceil$.
First assume there is a top-level box $b_{*}(\mathbf{a},n)$ and $\hat
{\mathbf{x}},\hat{\mathbf{y}}\in b_{*}^{3}$
such that $\mathbf{x}=\proj(\hat{\mathbf{x}}),\mathbf{y}=\proj
(\hat
{\mathbf{y}})$.
Let $\omega=\proj^{-1}(\omega_{0})\cap b_{*}^{7}$. Note
$d_{\omega}(\hat{\mathbf{x}},\hat{\mathbf{y}})\ge d_{\omega
_{0}}(\mathbf
{x},\mathbf{y})$
but since $\|b_{*}^{3}\|<N/2$,%
%
%$l_{1}$ distance is sum of absolute value coordinate differences
%(where mod$N$ is added for torus distance). Since all coordinates
%in $b_{*}^{3}$ are at most $3\lceil N/10\rceil$ difference, wrapping
%around is longer.
%%
%{}
$d_{b_{*}^7}(\hat{\mathbf{x}},\hat{\mathbf{y}})=d_{\mathcal
{T}}(\mathbf{x},\mathbf{y})$.
By \eqref{eqdistbnd2}, since $b_{*}$ is $(\proj^{-1}(\omega
_{0}),k,\rho)$-good
by definition, we are done. If no such $b_{*}$ exists, then by our
construction of top-level boxes, $d_{\mathcal{T}}(\mathbf{x},\mathbf{y})>n$.
Let $b_{\mathbf{x}},b_{\mathbf{y}}$ be the top-level boxes such that
$\mathbf{x}\in b_{\mathbf{x}},\mathbf{y}\in b_{\mathbf{y}}$. We can
make a $\Delta$-connected path of top-level boxes from $\mathbf{x}$
to $\mathbf{y}$ of length at most $10d$. Since $b_{1},b_{2}$ that
are $\Delta$-neighbors satisfy that $b_{1}\subset b_{2}^5$, by
Remark~\ref{rem0goodconnected}, \eqref{eqdistbnd2} implies the
theorem.

To simplify notation, we fix $k,\rho,n$ and $\omega\in\Good
_{k}^{\rho}(n)$
for the remainder of the section. We write $\Good_{i}$ (resp., $i$-good)
for $\Good_{i}^{\rho\Dns^{k-i}}$ [resp., $(\omega,i,\rho\Dns
^{k-i})$-good].

We now utilize the recursive goodness properties of $\omega$ to extract
a single connected cluster of $\omega$ which is a power of log $\omega
$-distance
from its complement in $\omega$ and is ``nicely'' embedded in
$\mathbb{Z}^{d}$. Given an $(i+1)$-good box $B$ where $0\le i<k$,
we write
\[
\mathcal{S}(B,i)=\bigl\{ b\in\sigma(B)\dvtx b\mbox{ is }i\mbox{-good} \bigr\},
\]
and let $\sigma_{B}=\|\Delta(\sigma(B))\|=|\sigma(B)|^{{1}/{d}}$.
Since $B$ is $(i+1)$-good, by definition we have that $\Delta\mathcal
{S}(B,i)\in\mathcal{P}(\sigma_{B})$.
Thus, there exists a good cluster $\mathcal{C}(\Delta\mathcal{(S}(B,i)))$
satisfying Percolation properties 1, 2, 3
(see Section~\ref{subperc}). Let $\Bx C_{i}(B)\subset\sigma(B)$ be the
set for which $(\Delta\Bx C_{i}(B))=\mathcal{C}(\Delta(\mathcal{S}(B,i)))$.
For $i=0,\ldots,k$ let us define $\beta_{i}=\beta_{i}(\omega,n)$.
Set $\beta_{k}=\{ B(n)\} $ and for $i=k-1,\ldots,0$ recursively
define $\beta_{i}=\{ b\in\Bx C_{i}(B)\dvtx B\in\beta_{i+1}\} $.
See Figure~\ref{figbetak} for a schematic illustration.

Let $n_{j}=s^{(k-j)}(n)$. Thus, for $b\in\beta_{j}$ we have $\|b\|=n_{j}$
and also $|\sigma(b)|^{{1}/{d}}<6n_{j}/n_{j-1}$ for
all large $n$. Note that by Percolation property~1,
$\{ \beta_{j}(n)\} _{j=0}^{k}$ are nonempty for all large
$n$. Roughly, $\bigcup\beta_{0}$ is the nicely embedded cluster
referred to above. Its precise properties follow.%
%
%f7 #&#
\begin{figure}

\includegraphics{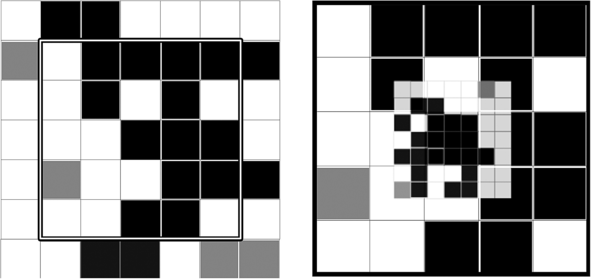}

\caption{On the left is a schematic example
of $B^7(n)$
where the black boxes represent $\beta_{k}=\Bx C_{k-1}(B(n))$ and
the gray ones $\mathcal{S}(B,k-1)\setminus\beta_{k}$. On the right
is a blowup of the framed region on the left where the small black
boxes are a part of $\beta_{k-1}$.}\label{figbetak}
\end{figure}

Given an $(i+1)$-good box $B$, let $\Bx C_{i}^5(B)=\{ b\in\Bx
C_{i}(B)\dvtx b\cap B^5\neq\varnothing\} $.
%
%leC.2 #&#
\begin{lem}
\label{lemptinbeta0}Set $b_{k}=B(n)$. There is a $C(k)$ such
that for any $\mathbf{x}_{k}\in b_{k}^5\cap\omega$, there are
boxes $\{ b_{i}\} _{i=0}^{k-1}$ satisfying: \textup{(i)} $b_{i}\in\Bx
C_{i}^5(b_{i+1})\subset\beta_{i}$,
and \textup{(ii)} there is a $\mathbf{x}_{0}\in\omega\cap b_{0}$ such that
\[
d_{\omega}(\mathbf{x}_{k},\mathbf{x}_{0})<C(k) (
\log n)^{4d+2}.
\]
\end{lem}
\begin{pf}
We use backward induction. For $1\le j\le k$, we prove that if $\mathbf
{x}_{j}\in B_{j}^5\cap\omega$
where $B_{j}\in\beta_{j}$, then there is a $b_{j-1}\in\Bx
C_{j-1}^5(B_{j})\subset\beta_{j-1}$
and a $\mathbf{x}_{j-1}\in b_{j-1}$ satisfying
%
%eC.2 #&#
\begin{equation}
d_{\omega}(\mathbf{x}_{j},\mathbf{x}_{j-1}
)<c(d)n_{j-1}^{d}(\log n_{j})^{2}.\label{eqxjtoxj-1}
\end{equation}
Since the conditions of the lemma provide us with an initial $\mathbf
{x}_{k}\in B^5(n)$
where by definition $B(n)\in\beta_{k}$, the bound on $d_{\omega
}(\mathbf
{x}_{k},\mathbf{x}_{0})$
is proved by connecting $\mathbf{x}_{k},\mathbf{x}_{k-1},\ldots,\mathbf{x}_{0}$.

We assumed $B_{j}\in\beta_{j}$, so in particular, $B_{j}$ is $j$-good.
Let $b_{*}\in\sigma(B_{j})$ be the subbox of $B_{j}$ containing
$\mathbf{x}_{j}$ and assume $b_{*}\notin\Bx C_{j-1}^5(B_{j})$
as otherwise we are done. Consider
\[
\Bx B=\bigl\{ b\in B_{\Delta}\bigl(b_{*},\log\bigl|
\sigma(B_{j}) \bigr|\bigr)\dvtx b\cap B_{j}^5\neq
\varnothing\bigr\}.
\]
Since $b_{*}\cap B_{j}^5\neq\varnothing$ by assumption, $|\Bx B
|>\log^{2}|\sigma(B_{j})|$,
and thus by Percolation property~2
(see Section~\ref{subperc}), there is a $b_{j-1}\in\Bx B\cap\Bx C_{j-1}^5(B_{j})$.
Thus, there is a $\Delta$-path $\Bx p\subset\sigma(B_{j})$ of length
at most $d\log|\sigma(B_{j})|$ starting at $b_{*}$ and
ending at $b_{j-1}$. By Remark~\ref{rem0goodconnected} on $\Good
_{0}^{\rho}$
(see Section~\ref{subk-good}), for any $\Delta$-neighboring boxes $b_{\alpha
},b_{\beta}$
in the path, $\omega\cap b_{\alpha}$ is connected to $\omega\cap
b_{\beta}$
in $\omega\cap b_{\alpha}^{+}$. Choosing some $\mathbf{x}_{j-1}\in
\omega
\cap b_{j-1}$
and using the volume of $\bigcup_{b\in\Bx p}b^7$ as a trivial distance
bound, we get (\ref{eqxjtoxj-1}).
\end{pf}
For $0\le j<k$, note that although $\|b_{1}\|=\|b_{2}\|$ for any
$b_{1},b_{2}\in\beta_{j}$, since they can be subboxes of different
$B_{1},B_{2}$, $b_{1}$ is not in general an element of $\operatorname{sp}
\{ b_{2}\} $.
Thus, for each $0\le j\le k$, we add a graph structure to $\beta_{j}$
by defining a neighbor relation ($\stackrel{5}{\sim}$) between boxes
$b_{1},b_{2}\in\beta_{j}$. We define that $b_{1}\stackrel{5}{\sim}b_{2}$
if and only if $b_{1}\subset b_{2}^5$ and $b_{2}\subset b_{1}^5$.
Note this relation is reflexive, and that for $(j+1)$-good $B$ and
$b_{1},b_{2}\in\Bx C_{j}(B)$, $d_{\beta_{j}}(b_{1},b_{2})\le
d_{\Delta
}(b_{1},b_{2})$.
For the remainder of the section, any graph properties of $\beta_{j}$
referred to, such as connectivity or distance, use the graph structure
created by $\stackrel{5}{\sim}$.
%
%leC.3 #&#
\begin{lem}
\label{lembetajconpartial}There is a $C_{d}$ such that for each
$0<j\le k$, if $B_{1},B_{2}\in\beta_{j}$ are $\stackrel{5}{\sim
}$-connected, and
we have $b_{1}\in\Bx C_{j-1}^5(B_{1}),b_{2}\in\Bx C_{j-1}^5(B_{2})$,
then\break  $d_{\beta_{j-1}}(b_{1},b_{2})<C_{d}(d_{\beta
_{j}}(B_{1},B_{2})\vee
1)n_{j}/n_{j-1}$.
In particular, $b_{1},b_{2}$ are $\stackrel{5}{\sim}$-con\-nected in
$\beta_{j-1}$.
\end{lem}
\begin{pf}
We prove the lemma for the special case of $B_{1}\stackrel{5}{\sim}B_{2}$
[i.e., $d_{\beta_{j}}(B_{1},\break  B_{2})\le1$]. The general lemma follows
by applying the neighbor case over a path in $\beta_{j}$ realizing
the $\stackrel{5}{\sim}$-distance between two fixed boxes. By
definition, $\Bx C_{j-1}(B_{1})$
and $\Bx C_{j-1}(B_{2})$ are each $\Delta$-connected sets, and thus
$\stackrel{5}{\sim}$-connected. By Percolation property~3,
for $i=1,2$ and any $b,b'\in\Bx C_{j-1}^5(B_{i})$, $d_{\Bx
C_{j-1}(B_{i})}(b,b')<C_{d}n_{j}/n_{j-1}$.
Since $\Bx C_{j-1}(B_{i})\subset\beta_{j-1}$ and $\stackrel{5}{\sim
}$-distance is
at most $\Delta$-distance, to complete the proof it is enough to show
existence of $\hat{b}_{1}\in\Bx C_{j-1}^5(B_{1})$ and $\hat
{b}_{2}\in
\Bx C_{j-1}^5(B_{2})$
such that $\hat{b}_{1}\subset\hat{b}_{2}^5$ and $\hat{b}_{2}\subset
\hat
{b}_{1}^5$.
For $i=1,2$, let $\Bx D_{i}=\{ b\in\sigma(B_{i})\dvtx b\subset
B_{2}\} $
and let $\Bx E_{i}=\Bx D_{i}\cap\Bx C_{j-1}(B_{i})$. Let $D_{i}=\bigcup
\Bx D_{i}$
and let $E_{i}=\bigcup\Bx E_{i}$. Since $E_{i}\subset D_{i}$, we
have $B_{2}\setminus E_{i}=(B_{2}\setminus D_{i})\cup
(D_{i}\setminus E_{i})$.
By a volume bound, $|B_{2}\setminus D_{i}|\le2dn_{j-1}n_{j}^{d-1}$
and by Percolation property~1, $|\Bx
D_{i}\setminus\Bx E_{i}|<10^{-d}|\sigma(B_{i})|$.
Since $|\sigma(B_{i})|^{{1}/{d}}<6n_{j}/n_{j-1}$,
this implies $|D_{i}\setminus E_{i}|<(0.6n_{j})^{d}$.
As $|B_{2}|=n_{j}^{d}$ and $d>2$, we have by the bound
on $|B_{2}\setminus E_{i}|$ for $i=1,2$ that there is
a $\mathbf{x}\in E_{1}\cap E_{2}$. The containing boxes $\mathbf
{x}\in
\hat{b}_{i}\in\Bx C_{j-1}(B_{i})$
for $i=1,2$ are thus $\stackrel{5}{\sim}$-neighbors.
\end{pf}
We now prove the theorem by showing there exists a $C(k,\rho)<\infty$
such that for any $\mathbf{x},\mathbf{y}\in\omega\cap B^5(n)$,
\eqref{eqdistbnd2} holds for all large $n$.
\begin{pf*}{Proof of Theorem~\ref{thmmDistancebound}}
We demonstrate there is a path from $\mathbf{x}$ to $\mathbf{y}$
in $\omega$ shorter than the RHS of \eqref{eqdistbnd2}.
Let $b_{\mathbf{x},k}=B(n)$ and apply Lemma~\ref{lemptinbeta0}
to $\mathbf{x}$ to get boxes $\{ b_{\mathbf{x},i}\} _{i=0}^{k-1}$
satisfying: (i) $b_{\mathbf{x},i}\in\Bx C_{i}^5(b_{\mathbf
{x},i+1})\subset\beta_{i}$,
and (ii) there is a $\mathbf{x}_{0}\in\omega\cap b_{\mathbf{x},0}$
such that $d_{\omega}(\mathbf{x},\mathbf{x}_{0})<C(k)(\log n
)^{4d+2}$.
Observe that (i) implies $\mathbf{x}_{0}\in b_{\mathbf{x},k-1}^{6}$
for all large $n$. Set $b_{\mathbf{y},k}=B(n)$ and apply the lemma
to $\mathbf{y}$ as well to get $b_{\mathbf{y},i}$ and $\mathbf{y}_{0}$
with analogous properties.

By Lemma~\ref{lembetajconpartial}, $\beta_{k-1}$ is $\stackrel
{5}{\sim}$-connected,
and more specifically,
\[
d_{\beta_{k-2}}(b_{\mathbf{x},k-2},b_{\mathbf
{y},k-2})<C_{d}
\bigl(d_{\beta
_{k-1}}(b_{\mathbf{x},k-1},b_{\mathbf{y},k-1})\vee1\bigr)
\frac{n_{k-1}}{n_{k-2}}.
\]
Iterating the lemma, we get
%
%eC.3 #&#
\begin{equation}
d_{\beta_{0}}(b_{\mathbf{x},0},b_{\mathbf
{y},0})<C_{d}^{k-1}
\bigl(d_{\beta
_{k-1}}(b_{\mathbf{x},k-1},b_{\mathbf{y},k-1})\vee1\bigr)
\frac
{n_{k-1}}{n_{0}}.\label{eqdbeta0I}
\end{equation}
Since $b_{\mathbf{x},k-1},b_{\mathbf{y},k-1}\in\Bx C_{k-1}^5(B(n))$,
by Percolation property~3
%
%eC.4 #&#
\begin{equation}\quad
d_{\Bx C_{k-1}(B(n))}(b_{\mathbf{x},k-1},b_{\mathbf{y},k-1} )<c_{a}d_{\sigma(B(n))}(b_{\mathbf{x},k-1},b_{\mathbf
{y},k-1})
\vee c_{a}\log\frac{n_{k}}{n_{k-1}}.\label{eqdCk-1II}
\end{equation}
Where both are defined, $\stackrel{5}{\sim}$-distance is at most
$\Delta
$-distance,
and thus we may replace $d_{\Bx C_{k-1}(B(n))}(\cdot,\cdot)$ in
\eqref
{eqdCk-1II}
by $d_{\beta_{k-1}}(\cdot,\cdot)$. Since $n_{k-1}\cdot d_{\sigma
(B(n))}(\cdot,\cdot)$
and $d_{B^7(n)}(\cdot,\cdot)$ are comparable, and using that $\mathbf
{x}_{0}\in b_{\mathbf{x},k-1}^{6},\mathbf{y}_{0}\in b_{\mathbf{y},k-1}^{6}$
we have
\[
n_{k-1}d_{\beta_{k-1}}(b_{\mathbf{x},k-1},b_{\mathbf
{y},k-1})<c_{a}^{\prime}
\bigl(d_{B^7(n)}(\mathbf{x}_{0},\mathbf {y}_{0})\vee
n_{k-1}\log n_{k}\bigr).\vadjust{\goodbreak}
\]
Plugging this into \eqref{eqdbeta0I}, we get
\[
d_{\beta_{0}}(b_{\mathbf{x},0},b_{\mathbf{y},0})<C(k)
\bigl(d_{B^7(n)}(\mathbf{x}_{0},\mathbf{y}_{0})\vee
\log^{5}n\bigr)/n_{0}.
\]
By properties of $\Good_{0}^{\rho}$ (see Section~\ref{subk-good}), vertices
in $\stackrel{5}{\sim}$-neighboring boxes in $\beta_{0}$ are connected
in $\omega$
in a path which is at most twice the volume of one box, and thus we
get
\[
d_{\omega}(\mathbf{x}_{0},\mathbf{y}_{0})<C(k)
\bigl(d_{B^7(n)}(\mathbf{x}_{0},\mathbf{y}_{0})\vee
\log^{5}n \bigr) (n_{0})^{d-1}.
\]
We pay a $C(\log n)^{4d+2}$ term to connect $\mathbf
{x},\mathbf{y}$
to $\mathbf{x}_{0},\mathbf{y}_{0}$, respectively. This terms also
absorbs the $(n_{0})^{d-1}\log^{5}n$ factor above. Since
$(n_{0})^{d-1}$ is\break  $o(\log^{(k-1)}n)$, we are
done.
\end{pf*}

%%%%%%%%%%%%%%%%%%%%%%%%%%%%%%%%%%%%%%%%%%%%%%%%%%%%%%%%%%%%%%%%%%%%%%%%%%%5
%%%%%%%%%%%%%%5 Random Interlacements notation
%%%%%%%%%%%%%%%%%%%%%%%%%%%%%%5
%%%%%%%%%%%%%%%%%%%%%%%%%%%%%%%%%%%%%%%%%%%%%%%%%%%%%%%%%%%%%%%%%%%%%%%
%sD #&#
\section{Random interlacements notation}\label
{apri}\label{appC}
We try to follow as much as possible the canonical notation of
Alain-Sol Sznitman \cite{MR2680403}.
Let $W$ and $W_+$ be the spaces of doubly infinite and infinite
trajectories in $\Z^d$
that spend only a finite amount of time in finite subsets of $\Z^d$:
\begin{eqnarray*}
W&=&\Bigl\{\gamma\dvtx\Z\to\Z^d; \bigl|\gamma(n)-\gamma(n+1)\bigr|=1, \forall n
\in \Z; \lim_{n\to\pm\infty}\bigl|\gamma(n)\bigr|=\infty\Bigr\},
\\
W_+&=&\Bigl\{\gamma\dvtx{\mathbb N}\to\Z^d; \bigl|\gamma(n)-\gamma(n+1)\bigr|=1,
\forall n\in \Z; \lim_{n\to\infty}\bigl|\gamma(n)\bigr|=\infty\Bigr\}.
\end{eqnarray*}
The canonical
coordinates on ${ W}$ and ${ W}_+$ will be denoted
by $X_n$, \mbox{$n\in{\mathbb Z}$} and $X_n$, $n\in{\mathbb N}$,
respectively. Here, we use the convention that ${\mathbb N}$ includes
$0$. We endow $W$ and $W_+$ with the sigma-algebras
${\mathcal W}$ and ${\mathcal W}_+$, respectively, which are
generated by the canonical coordinates. For $\gamma\in W$, let range
$(\gamma)=\gamma({\mathbb Z})$. Furthermore, consider the
space $W^*$ of trajectories in $W$ modulo time shift:
\[
W^*=W/\sim\qquad \mbox{where }w\sim w'\Longleftrightarrow w(
\cdot)=w'(\cdot+k)\mbox{ for some }k\in\Z.
\]
Let $\pi^*$ be the
canonical projection from $W$ to $W^*$, and let ${\mathcal W}^*$
be the sigma-algebra on $W^*$ given by $\{A\subset
W^* \dvtx(\pi^*)^{-1}(A)\in{\mathcal W}\}$. Given
$K\subset{\mathbb Z}^d$ and $\gamma\in W_+$, let
$\tilde{H}_K(\gamma)$ denote the hitting time of $K$ by $\gamma$:
%
%eD.1 #&#
\begin{equation}
\label{hitting} \tilde{H}_K(\gamma)=\inf\bigl\{n\ge1 \dvtx
X_n(\gamma)\in K\bigr\}.
\end{equation}
For $x\in{\mathbb Z}^d$, let $P_x$ be the law on $(W_+,{\mathcal
W}_+)$ corresponding to simple random walk started at $x$, and for
$K\subset\BZ^d$, let $P_x^K$ be the law of simple random walk,
conditioned on not hitting $K$. Define the equilibrium measure of
$K$:
%
%eD.2 #&#
\begin{eqnarray}
e_K(x)=\cases{
P_x[
\tilde{H}_K=\infty], &\quad $x\in K,$
\vspace*{2pt}\cr
0, &\quad $x\notin K.$}
\end{eqnarray}
Define the capacity of
a set $K\subset\BZ^d$ as
%
%eD.3 #&#
\begin{equation}
\operatorname{cap}(K)=\sum_{x\in\BZ^d} e_K(x).\vadjust{\goodbreak}
\end{equation}

Next, we
define a Poisson point process on $W^*\times{\mathbb R}_+$. The
intensity measure of the Poisson point process is given by the
product of a certain measure $\nu$ and the Lebesque measure on
${\mathbb R}_+$. The measure $\nu$ was constructed by Sznitman in
\cite{MR2680403}, and now we characterize it. For
$K\subset{\mathbb Z}^d$, let $W_K$ denote the set of trajectories
in $W$ that enter $K$. Let $W_K^{*}=\pi^*(W_K)$ be the set of
trajectories in $W^*$ that intersect $K$. Define $Q_K$ to
be the finite measure on $W_K$ such that for $A,B\in{\mathcal
W}_+$ and $x\in{\mathbb Z}^d$,
%
%eD.4 #&#
\begin{equation}
\label{eQdef}Q_K\bigl[(X_{-n})_{n\ge0}\in
A,X_0=x,(X_n)_{n\ge0}\in B\bigr]=P_x^K[A]e_K(x)P_x[B].
\end{equation}
The measure $\nu$ is the unique $\sigma$-finite measure such that
%
%eD.5 #&#
\begin{equation}
\label{enudef} \ind_{W^*_K}\nu=\pi^*\circ Q_K\qquad
\forall K\subset{\mathbb Z}^d\mbox{ finite}.
\end{equation}
The existence and uniqueness of the measure was proved in Theorem~1.1
of \cite{MR2680403}. Consider the set of point measures
in $W^{*}\times{\mathbb R}_+$:
%
%eD.6 #&#
%eD.7 #&#
\begin{eqnarray}
\Omega&=&\Biggl\{\omega=\sum_{i=1}^{\infty}
\delta_{(w_i^*,u_i)}; w_i^*\in W^*, u_i\in{\mathbb
R}_+,
\nonumber
\\[-8pt]
\\[-8pt]
\nonumber
&&\hspace*{8pt}\omega\bigl(W_K^*\times[0,u]\bigr)<\infty,\mbox{ for every finite
}K\subset \Z ^d\mbox{ and }u\in{\mathbb R}_+\Biggr\}.
\end{eqnarray}

Also consider the space of point measures on $W^*$:
%
%eD.8 #&#
\begin{equation}\quad
\tilde{\Omega}=\Biggl\{\sigma=\sum_{i=1}^{\infty}
\delta_{w_i^*}; w_i^*\in W^*, \sigma\bigl(W_K^*
\bigr)<\infty, \mbox{ for every finite }K\subset\Z^d\Biggr\}.
\end{equation}
For $u>u'\ge0$, we define the mapping $\omega_{u',u}$ from $\Omega$ into
$\tilde{\Omega}$ by
%
%eD.9 #&#
\begin{equation}
\label{eomegaudef}\omega_{u',u}=\sum_{i=1}^\infty
\delta_{w_i^*}{\ind}\bigl\{u'\le u_i\le u
\bigr\}\qquad\mbox{for }\omega=\sum_{i=1}^{\infty}
\delta_{(w_i^*,u_i)}\in\Omega.
\end{equation}
If $u'=0$, we write
$\omega_u$. On $\Omega$ we let
${\mathbb P}$ be the law of a Poisson point process with intensity
measure given by $\nu(dw^*)\,dx$. Observe that under ${\mathbb P}$,
the point process $\omega_{u,u'}$ is a Poisson point process on
$\tilde{\Omega}$ with intensity measure $(u-u') \nu(dw^*)$. Given
$\sigma\in\tilde{\Omega}$, we define
%
%eD.10 #&#
\begin{equation}
\label{enicenotation} {\mathcal I}(\sigma)=\bigcup_{w^*\in\operatorname{supp}(\sigma)}
\operatorname{range}\bigl(w^*\bigr).
\end{equation}
For $0\le u'\le u$, we define
%
%eD.11 #&#
\begin{equation}
\label{eridef} {\mathcal I}^{u',u}={\mathcal I}(\omega_{u',u}),
\end{equation}
which we call the \emph{random interlacement set} between levels
$u'$ and $u$. In case $u'=0$, we write ${\mathcal I}^u$.

Finally, we can define the measure of the random walk described in
Theorem~\ref{thmmInterlacementmain}. Let $\pr^u_0[\cdot]=\pr[\cdot
|0\in
\CI^u]$. For every $\CI^u$ distributed according to $\pr_0^u$, let
$\mathbf{P}_0^{u}$ be the law of a SRW on $\CI^u$ starting from $0$.

%%%%%%%%%%%%%%%%%%%%%%%%%%%%%%%%%%%%%%%%%%%%%%%%%%%%%%%%%%%%%%%%%%%%%%%%%%%%%
%%%%%%%%%%%%%%%%%%%%List of Symbols%%%%%%%%%%%%%%%%%%%%%%%%%%%%%%%%%%%%
%%%%%%%%%%%%%%%%%%%%%%%%%%%%%%%%%%%%%%%%%%%%%%%%%%%%%%%%%%%%%%%%%%%%%%%

%sE #&#
\section{Index of symbols by order of appearance}\label{secinde}
\vspace*{10pt}
\def\arraystretch{1.1}

\centering{\begin{tabular*}{\textwidth}{@{\extracolsep{\fill}}llp{256pt}@{}}
\hline
\multicolumn{1}{@{}l}{\textbf{Symbol}}&\multicolumn{1}{c}{\textbf{Page}}&\multicolumn{1}{c}{\textbf{Definition}}\\
\hline
$\CT(N,d)$&\pageref{pg1} & $d$-dimensional torus. \\
$\proj$&\pageref{pg2} & For $x\in\Z^d$, $\CO_N(x)=(x_1\operatorname{ mod }N,\ldots,x_d\operatorname
{ mod }N)$.\\
$\CR(t)$& \pageref{pg3}& The range of SRW on the torus.\\
$\partial$&\pageref{pg4}&Outer vertex boundary. \\
$\partial^{\mathrm{in}}$&\pageref{pg5}&Inner vertex boundary.\\
$B(\mathbf{x},n)$ &\pageref{pg6} &$\{\mathbf{y}\in\Z^d\dvtx\forall i,1\le i\le
d,-n/2\le{\mathbf x}(i)-\mathbf{y}(i)<n/2\}$.\\
$\operatorname{sp}\{B(\mathbf{x},n)\}$& \pageref{pg7}& $\{B(\mathbf{x}+\sum_i\mathbf
{e}_ik_in,n)\dvtx(k_1,\ldots,k_d)\in\Z^d\}$.\\
$\Delta$& \pageref{pg8}& The isomorphism, $\Delta$: $\operatorname{sp}\{B\}\rightarrow\Z^d$. \\
$B^\alpha$& \pageref{pg9}& For a box $B=B(x,n)$, $B^\alpha=B(x,\alpha n)$.\\
$s(n)$ & \pageref{pg10}& $\lceil\log n\rceil^4$.\\
$s^{(i)}(n)$& \pageref{pg11} & $s(\cdot)$ iterated $i$ times.\\
$\sigma(B(\mathbf{x},n))$&\pageref{pg12}& $\operatorname{sp}\{b(\mathbf{x},s(n))\}
\cap\{b({\mathbf y},s(n))\dvtx\mathbf{y}\in B(x,5n+3\lceil\log n\rceil^6)\}$.\\
$\CP(n)$ & \pageref{pg13}& Percolation configurations.\\
$\CG_k^\rho(n)$& \pageref{pg14}& $k$-good configurations.\\
$\hat{\phi}(r)$&\pageref{pg15}& $\inf\{\Phi_S\dvtx N^{1/3}<\pi(S)\le r\wedge
(1-1/4d)|\om
_0|\}$.\\
$\Phi_S$&\pageref{pg16}&$\frac{Q(S,S^c)}{\pi(S)}$.\\
$\Phi(u)$ &\pageref{pg17} &$\inf\{\Phi_S\dvtx0<\pi(S)\le u\wedge\frac{1}{2}\}$.\\
Top, Bot&\pageref{pg18}&Top and bottom projections of $B^3$ on $B^7$.\\
B-$\mathit{traversal}$&\pageref{pg19}&An ordered pair $\eta=(a,z)$, $a\in\operatorname{Top}$, $z\in
\operatorname{Bot}$.\\
B-$\mathit{itinerary}$&\pageref{pg20}&An ordered sequence of B-$\mathit{traversals}$.\\
$\tau_\rho(b)$& \pageref{pg21} &$\tau_\rho(b)=\gamma^+_{\lceil\rho
n^{d-2}\rceil}$.\\
$\CD^\rho_{\Lambda\rho}$ & \multirow{2}{20pt}{\pageref{pg221}, \pageref{pg222}} & Each $b\in\sigma(B)$ is traversed
top to bottom at least $\Lambda\rho\|b\|^{d-2}$ times.\\
$\CF_N^{T}(b,k,\rho)$& \pageref{pg23} & The event $\{b\cap\CO_N^{-1}\circ\CR
_N(t)\in
\CG_k^\rho(n)\dvtx\forall t\ge T\}$.\\
\hline
\end{tabular*}}
\end{appendix}

%%%%%%%%%%%%%%%%%%%%%%%%%%%%%%%%%%%%%%%%%%%%%%%%%%%%%%%%%%%%%%%%%%%%%%%%%%%%%%%%%%%%%%%%%%%%%%%%%%%%%%%%%%%%%%%%%%%%%%%%%%%%%%%%%%%%%%%%%
%%%%%%%%%%%%%%%%%%%%%%%%%%%%%%%%%%%%%%%%%%%%%%%%%%%%%%%%%%%%%%%%%%%%%%%%%%%%%%%%%%%%%%%%%%%%%%%%%%%%%%%%%%%%%%%%%%%%%%%%%%%%%%%%%%%%%%%%%
%%%%%%%%%%%%%%%%%%%%%%%%%%%%%%%%%%%%%%%%%%%%%%%%%%%%%%%%%%%%%%%%%%%%%%%%%%%%%%%%%%%%%%%%%%%%%%%%%%%%%%%%%%%%%%%%%%%%%%%%%%%%%%%%%%%%%%%%%
%%%%%%%%%%%%%%%%%%%%%%%%%%%%%%%%%%%%%%%%%%%%%%%%%%%%%%%%%%%%%%%%%%%%%%%%%%%%%%%%%%%%%%%%%%%%%%%%%%%%%%%%%%%%%%%%%%%%%%%%%%%%%%%%%%%%%%%%%
%%%%%%%%%%%%%%%%%%%%%%%%%%%%%%%%%%%%%%%%%%%%%%%%%%%%%%%%%%%%%%%%%%%%%%%%%%%%%%%%%%%%%%%%%%%%%%%%%%%%%%%%%%%%%%%%%%%%%%%%%%%%%%%%%%%%%%%%%
%%%%%%%%%%%%%%%%%%%%%%%%%%%%%%%%%%%%%%%%%%%%%%%%%%%%%%%%%%%%%%%%%%%%%%%%%%%%%%%%%%%%%%%%%%%%%%%%%%%%%%%%%%%%%%%%%%%%%%%%%%%%%%%%%%%%%%%%%
%%%%%%%%%%%%%%%%%%%%%%%%%%%%%%%%%%%%%%%%%%%%%%%%%%%%%%%%%%%%%%%%%%%%%%%%%%%%%%%%%%%%%%%%%%%%%%%%%%%%%%%%%%%%%%%%%%%%%%%%%%%%%%%%%%%%%%%%%
\section*{Acknowledgements}
Thanks goes to Itai Benjamini for suggesting
this problem and for fruitful discussions, and also to Gady Kozma
who suggested the renormalization method and provided examples and
counterexamples whenever they were needed.

%
%
%accubitu suo, nardus mea dedit odorem suavitatis. Quoniam confortavit
%seras portarum tuarum, benedixit filiis tuis in te. Qui posuit fines
%tuos}

% imsref loaded by akundreckaite, 2014-04-15 15:14:55
%

%


%suskaldyti doi

\printaddresses


\begin{thebibliography}{17}
% pybtex-1.04. Style name=ims, version=2.8, label_style=nolabel,
%sorting_style=complex, cfg=None, language=None.

%b1 ###
%b1 #&#
\bibitem{antal1996chemical}
%
\begin{barticle}[mr]
\bauthor{\bsnm{Antal},~\bfnm{Peter}\binits{P.}} \AND
\bauthor{\bsnm{Pisztora},~\bfnm{Agoston}\binits{A.}}
(\byear{1996}).
\btitle{On the chemical distance for supercritical {B}ernoulli percolation}.
\bjournal{Ann. Probab.}
\bvolume{24}
\bpages{1036--1048}.
\bid{doi={10.1214/aop/1039639377}, issn={0091-1798}, mr={1404543}}
\end{barticle}
%
\bptok{imsref}%
% NOT OUTPUTED:
% issn = 0091-1798
% url = http://dx.doi.org/10.1214/aop/1039639377
% number = 2
% coden = APBYAE
% fjournal = The Annals of Probability
\endbibitem

%b2 ###
%b2 #&#
\bibitem{benjamini2003mixing}
%
\begin{barticle}[mr]
\bauthor{\bsnm{Benjamini},~\bfnm{Itai}\binits{I.}} \AND
\bauthor{\bsnm{Mossel},~\bfnm{Elchanan}\binits{E.}}
(\byear{2003}).
\btitle{On the mixing time of a simple random walk on the super
critical percolation cluster}.
\bjournal{Probab. Theory Related Fields}
\bvolume{125}
\bpages{408--420}.
\bid{doi={10.1007/s00440-002-0246-y}, issn={0178-8051}, mr={1967022}}
\end{barticle}
%
\bptok{imsref}%
% NOT OUTPUTED:
% issn = 0178-8051
% url = http://dx.doi.org/10.1007/s00440-002-0246-y
% number = 3
% coden = PTRFEU
% fjournal = Probability Theory and Related Fields
\endbibitem

%b3 ###
%b3 #&#
\bibitem{benjamini2008giant}
%
\begin{barticle}[mr]
\bauthor{\bsnm{Benjamini},~\bfnm{Itai}\binits{I.}} \AND
\bauthor{\bsnm{Sznitman},~\bfnm{Alain-Sol}\binits{A.-S.}}
(\byear{2008}).
\btitle{Giant component and vacant set for random walk on a discrete torus}.
\bjournal{J. Eur. Math. Soc. (JEMS)}
\bvolume{10}
\bpages{133--172}.
\bid{doi={10.4171/JEMS/106}, issn={1435-9855}, mr={2349899}}
\end{barticle}
%
\bptok{imsref}%
% NOT OUTPUTED:
% issn = 1435-9855
% url = http://dx.doi.org/10.4171/JEMS/106
% number = 1
% fjournal = Journal of the European Mathematical Society (JEMS)
\endbibitem

%b4 ###
%b4 #&#
\bibitem{bramson1991asymptotic}
%
\begin{barticle}[mr]
\bauthor{\bsnm{Bramson},~\bfnm{Maury}\binits{M.}} \AND
\bauthor{\bsnm{Lebowitz},~\bfnm{Joel~L.}\binits{J.~L.}}
(\byear{1991}).
\btitle{Asymptotic behavior of densities for two-particle annihilating
random walks}.
\bjournal{J. Stat. Phys.}
\bvolume{62}
\bpages{297--372}.
\bid{doi={10.1007/BF01020872}, issn={0022-4715}, mr={1105266}}
\end{barticle}
%
\bptok{imsref}%
% NOT OUTPUTED:
% issn = 0022-4715
% url = http://dx.doi.org/10.1007/BF01020872
% number = 1-2
% coden = JSTPSB
% fjournal = Journal of Statistical Physics
\endbibitem

%b5 ###
%b5 #&#
\bibitem{cerny2011internal}
%
\begin{bmisc}[author]
\bauthor{\bsnm{\v{C}ern{\'{y}}},~\bfnm{J.}\binits{J.}} \AND
\bauthor{\bsnm{Popov},~\bfnm{S.}\binits{S.}}
(\byear{2011}).
\bhowpublished{On the internal distance in the interlacement set}.
\end{bmisc}
%
\bptok{imsref}%
\endbibitem

%b6 ###
%b6 #&#
\bibitem{deuschel1996surface}
%
\begin{barticle}[mr]
\bauthor{\bsnm{Deuschel},~\bfnm{Jean-Dominique}\binits{J.-D.}} \AND
\bauthor{\bsnm{Pisztora},~\bfnm{Agoston}\binits{A.}}
(\byear{1996}).
\btitle{Surface order large deviations for high-density percolation}.
\bjournal{Probab. Theory Related Fields}
\bvolume{104}
\bpages{467--482}.
\bid{doi={10.1007/BF01198162}, issn={0178-8051}, mr={1384041}}
\end{barticle}
%
\bptok{imsref}%
% NOT OUTPUTED:
% issn = 0178-8051
% url = http://dx.doi.org/10.1007/BF01198162
% number = 4
% coden = PTRFEU
% fjournal = Probability Theory and Related Fields
\endbibitem

%b7 ###
%b7 #&#
\bibitem{grimmett1999percolation}
%
\begin{bbook}[mr]
\bauthor{\bsnm{Grimmett},~\bfnm{Geoffrey}\binits{G.}}
(\byear{1999}).
\btitle{Percolation},
\bedition{2nd} ed.
\bseries{Grundlehren der Mathematischen Wissenschaften}
\bvolume{321}.
\bpublisher{Springer},
\blocation{Berlin}.
\bid{mr={1707339}}
\end{bbook}
%
\bptok{imsref}%
% NOT OUTPUTED:
% isbn = 3-540-64902-6
% fpage = xiv+444
\endbibitem

%b8 ###
%b8 #&#
\bibitem{lawler1996intersections}
%
\begin{bbook}[author]
\bauthor{\bsnm{Lawler},~\bfnm{G.~F.}\binits{G.~F.}}
(\byear{1996}).
\btitle{{Intersections of Random Walks}}.
\bpublisher{Birkh\"{a}user},
\blocation{Basel}.
\end{bbook}
%
\bptok{imsref}%
\endbibitem

%b9 ###
%b9 #&#
\bibitem{lawler2010random}
%
\begin{bbook}[mr]
\bauthor{\bsnm{Lawler},~\bfnm{Gregory~F.}\binits{G.~F.}} \AND
\bauthor{\bsnm{Limic},~\bfnm{Vlada}\binits{V.}}
(\byear{2010}).
\btitle{Random Walk: A Modern Introduction}.
\bseries{Cambridge Studies in Advanced Mathematics}
\bvolume{123}.
\bpublisher{Cambridge Univ. Press},
\blocation{Cambridge}.
\bid{mr={2677157}}
\end{bbook}
%
\bptok{imsref}%
% NOT OUTPUTED:
% isbn = 978-0-521-51918-2
% fpage = xii+364
\endbibitem

%b10 ###
%b10 #&#
\bibitem{liggett1997domination}
%
\begin{barticle}[mr]
\bauthor{\bsnm{Liggett},~\bfnm{T.~M.}\binits{T.~M.}},
\bauthor{\bsnm{Schonmann},~\bfnm{R.~H.}\binits{R.~H.}} \AND
\bauthor{\bsnm{Stacey},~\bfnm{A.~M.}\binits{A.~M.}}
(\byear{1997}).
\btitle{Domination by product measures}.
\bjournal{Ann. Probab.}
\bvolume{25}
\bpages{71--95}.
\bid{doi={10.1214/aop/1024404279}, issn={0091-1798}, mr={1428500}}
\end{barticle}
%
\bptok{imsref}%
% NOT OUTPUTED:
% issn = 0091-1798
% url = http://dx.doi.org/10.1214/aop/1024404279
% number = 1
% coden = APBYAE
% fjournal = The Annals of Probability
\endbibitem

%b11 ###
%b11 #&#
\bibitem{lovasz1999faster}
%
\begin{bincollection}[mr]
\bauthor{\bsnm{Lov{\'a}sz},~\bfnm{L{\'a}szl{\'o}}\binits{L.}} \AND
\bauthor{\bsnm{Kannan},~\bfnm{Ravi}\binits{R.}}
(\byear{1999}).
\btitle{Faster mixing via average conductance}.
In \bbooktitle{Annual {ACM} {S}ymposium on {T}heory of {C}omputing
({A}tlanta, {GA}, 1999)}
\bpages{282--287}.
\bpublisher{ACM},
\blocation{New York}.
\bid{doi={10.1145/301250.301317}, mr={1798047}}
\end{bincollection}
%
\bptok{imsref}%
% NOT OUTPUTED:
% url = http://dx.doi.org/10.1145/301250.301317
\endbibitem

%b12 ###
%b12 #&#
\bibitem{mathieu2004isoperimetry}
%
\begin{barticle}[mr]
\bauthor{\bsnm{Mathieu},~\bfnm{Pierre}\binits{P.}} \AND
\bauthor{\bsnm{Remy},~\bfnm{Elisabeth}\binits{E.}}
(\byear{2004}).
\btitle{Isoperimetry and heat kernel decay on percolation clusters}.
\bjournal{Ann. Probab.}
\bvolume{32}
\bpages{100--128}.
\bid{doi={10.1214/aop/1078415830}, issn={0091-1798}, mr={2040777}}
\end{barticle}
%
\bptok{imsref}%
% NOT OUTPUTED:
% issn = 0091-1798
% url = http://dx.doi.org/10.1214/aop/1078415830
% number = 1A
% coden = APBYAE
% fjournal = The Annals of Probability
\endbibitem

%b13 ###
%b13 #&#
\bibitem{morris2005evolving}
%
\begin{barticle}[mr]
\bauthor{\bsnm{Morris},~\bfnm{B.}\binits{B.}} \AND
\bauthor{\bsnm{Peres},~\bfnm{Yuval}\binits{Y.}}
(\byear{2005}).
\btitle{Evolving sets, mixing and heat kernel bounds}.
\bjournal{Probab. Theory Related Fields}
\bvolume{133}
\bpages{245--266}.
\bid{doi={10.1007/s00440-005-0434-7}, issn={0178-8051}, mr={2198701}}
\end{barticle}
%
\bptok{imsref}%
% NOT OUTPUTED:
% issn = 0178-8051
% url = http://dx.doi.org/10.1007/s00440-005-0434-7
% number = 2
% coden = PTRFEU
% fjournal = Probability Theory and Related Fields
\endbibitem

%b14 ###
%b14 #&#
\bibitem{pete2007note}
%
\begin{barticle}[mr]
\bauthor{\bsnm{Pete},~\bfnm{G{\'a}bor}\binits{G.}}
(\byear{2008}).
\btitle{A note on percolation on {$\mathbb{Z}^d$}: Isoperimetric
profile via exponential cluster repulsion}.
\bjournal{Electron. Commun. Probab.}
\bvolume{13}
\bpages{377--392}.
\bid{doi={10.1214/ECP.v13-1390}, issn={1083-589X}, mr={2415145}}
\end{barticle}
%
\bptok{imsref}%
% NOT OUTPUTED:
% issn = 1083-589X
% url = http://dx.doi.org/10.1214/ECP.v13-1390
% fjournal = Electronic Communications in Probability
\endbibitem

%b15 ###
%b15 #&#
\bibitem{procaccia2011concentration}
%
\begin{bmisc}[author]
\bauthor{\bsnm{Procaccia},~\bfnm{E.~B.}\binits{E.~B.}} \AND
\bauthor{\bsnm{Rosenthal},~\bfnm{R.}\binits{R.}}
(\byear{2011}).
\bhowpublished{Concentration estimates for the isoperimetric constant
of the super critical percolation cluster.
Preprint. Available at \arxivurl{arXiv:1110.6006}}.
\end{bmisc}
%
\bptok{imsref}%
\endbibitem

%b16 ###
%b16 #&#
\bibitem{rath2011transience}
%
\begin{barticle}[mr]
\bauthor{\bsnm{R{\'a}th},~\bfnm{Bal{\'a}zs}\binits{B.}} \AND
\bauthor{\bsnm{Sapozhnikov},~\bfnm{Art{\"e}m}\binits{A.}}
(\byear{2011}).
\btitle{On the transience of random interlacements}.
\bjournal{Electron. Commun. Probab.}
\bvolume{16}
\bpages{379--391}.
\bid{doi={10.1214/ECP.v16-1637}, issn={1083-589X}, mr={2819660}}
\end{barticle}
%
\bptok{imsref}%
% NOT OUTPUTED:
% issn = 1083-589X
% url = http://dx.doi.org/10.1214/ECP.v16-1637
% fjournal = Electronic Communications in Probability
\endbibitem

%b17 ###
%b17 #&#
\bibitem{MR2680403}
%
\begin{barticle}[mr]
\bauthor{\bsnm{Sznitman},~\bfnm{Alain-Sol}\binits{A.-S.}}
(\byear{2010}).
\btitle{Vacant set of random interlacements and percolation}.
\bjournal{Ann. of Math.~(2)}
\bvolume{171}
\bpages{2039--2087}.
\bid{doi={10.4007/annals.2010.171.2039}, issn={0003-486X}, mr={2680403}}
\end{barticle}
%
\bptok{imsref}%
% NOT OUTPUTED:
% issn = 0003-486X
% url = http://dx.doi.org/10.4007/annals.2010.171.2039
% number = 3
% coden = ANMAAH
% fjournal = Annals of Mathematics. Second Series
\endbibitem

\end{thebibliography}
\end{document}